\documentclass[a4paper,12pt]{amsart}

\usepackage{amssymb}
\usepackage{amsmath}
\usepackage{enumerate}

\makeatletter
\@namedef{subjclassname@2020}{%
  \textup{2020} Mathematics Subject Classification}
\makeatother

\newtheorem{theorem}{Theorem}[section]
\newtheorem{corollary}[theorem]{Corollary}
\newtheorem{lemma}[theorem]{Lemma}
\newtheorem{proposition}[theorem]{Proposition}

\theoremstyle{definition}
\newtheorem{definition}[theorem]{Definition}
\newtheorem{remark}[theorem]{Remark}
\newtheorem{example}[theorem]{Example}

\numberwithin{equation}{section}

\newcommand\inner[2]{\left\langle #1, #2 \right\rangle}

\makeatletter
\newcommand\xleftrightarrow[2][]{%
  \ext@arrow 9999{\longleftrightarrowfill@}{#1}{#2}}
\newcommand\longleftrightarrowfill@{%
  \arrowfill@\leftarrow\relbar\rightarrow}
\makeatother

\newcommand{\tco}{\mathcal{T}}
\newcommand{\bo}{\mathcal{L}(L^2)}

\newcommand{\beauty}{\mathcal{B}}
\newcommand{\beast}{\mathcal{B}'}
\newcommand{\HS}{\mathcal{HS}}

\newcommand{\C}{\mathbb{C}}
\newcommand{\R}{\mathbb{R}}
\newcommand{\Rd}{\mathbb{R}^d}
\newcommand{\Rdd}{\mathbb{R}^{2d}}
\newcommand{\Z}{\mathbb{Z}}
\newcommand{\N}{\mathbb{N}}

\newcommand{\F}{\mathcal{F}}
\newcommand{\tr}{\mathrm{tr}}
\newcommand{\frameop}{\mathfrak{S}}
\newcommand{\kernel}{k}
\newcommand{\weyl}{a}

\begin{document}




\title[Gabor g-frames]{On Gabor g-frames and Fourier series of operators}

\author[E. Skrettingland]{Eirik Skrettingland}
\address{Department of Mathematics\\ NTNU Norwegian University of Science and
Technology\\ NO–7491 Trondheim\\Norway}
\email{eirik.skrettingland@ntnu.no}

\begin{abstract}
  We show that Hilbert-Schmidt operators can be used to define frame-like structures for $L^2(\Rd)$ over lattices in $\Rdd$ that include multi-window Gabor frames as a special case. These frame-like structures are called Gabor g-frames, as they are examples of g-frames as introduced by Sun. We show that Gabor g-frames share many properties of Gabor frames, including a Janssen representation and Wexler-Raz biorthogonality conditions. A central part of our analysis is a notion of Fourier series of periodic operators based on earlier work by Feichtinger and Kozek, where we show in particular a Poisson summation formula for trace class operators. By choosing operators from certain Banach subspaces of the Hilbert Schmidt operators, Gabor g-frames give equivalent norms for modulation spaces in terms of weighted $\ell^p$-norms of an associated sequence, as previously shown for localization operators by D\"orfler, Feichtinger and Gr\"ochenig. 
\end{abstract}



\maketitle

\section{Introduction}
The study of Gabor frames is today an essential part of time-frequency analysis. By fixing a window function $\varphi\in L^2(\Rd)$, a signal $\psi \in L^2(\Rd)$ is analyzed by considering its projections onto copies of $\varphi$ shifted in time and frequency. In other words, one considers the \textit{short-time Fourier transform}
\begin{equation*}
  V_\varphi \psi(z) = \inner{\psi}{\pi(z)\varphi}_{L^2} \quad \text{ for } z\in \Rdd,
\end{equation*}
 where $\pi(z)$ is the time-frequency shift operator defined by $\pi(z)\varphi(t)=e^{2\pi i \omega\cdot t} \varphi(t-x)$ for $z=(x,\omega)\in \Rdd$. If $\varphi$ is well-behaved, one interprets $|V_\varphi \psi(x,\omega)|^2$ as a measure of the contribution of the frequency $\omega$ at the time $x$ in the signal $\psi$. Given a lattice $\Lambda=A\Z^{2d}$ for $A\in GL(2d,\R)$, $\varphi$ generates a \textit{Gabor frame} over  $\Lambda$ if the $\ell^2$-norm of the sequence $\{V_\varphi \psi(\lambda)\}_{\lambda \in \Lambda}$ is equivalent to the $L^2$-norm of $\psi$, i.e. there should exist constants $A,B>0$ such that
 \begin{equation} \label{eq:intro:gaborframeop}
  A\|\psi\|_{L^2}^2 \leq \sum_{\lambda\in \Lambda} |V_\varphi \psi(\lambda)|^2  \leq B \|\psi\|_{L^2}^2 \quad \text{ for any } \psi\in L^2(\Rd).
\end{equation}
In the usual terminology of frames, see for instance the monographs \cite{Grochenig:2001,Christensen:2016,Heil:2011}, this simply means that $\{\pi(\lambda)\varphi\}_{\lambda \in \Lambda}$ is a frame for $L^2(\Rd)$, and \eqref{eq:intro:gaborframeop} is equivalent to the fact that the \textit{frame operator} 
\begin{equation*}
  \psi \mapsto \sum_{\lambda \in \Lambda} V_\varphi \psi(\lambda)\pi(\lambda) \varphi
\end{equation*}
is bounded and invertible on $L^2(\Rd)$. Research over the last thirty years has revealed several intriguing features of Gabor frames, among them the Janssen representation of the frame operator \cite{Janssen:1995,Feichtinger:1998,Rieffel:1988}, the Wexler-Raz biorthogonality conditions \cite{Wexler:1990,Janssen:1995,Daubechies:1995,Feichtinger:1998} and that for well-behaved windows $\varphi$ summability conditions on the coefficients $\{V_\varphi \psi(\lambda)\}_{\lambda \in \Lambda}$ characterize smoothness and decay properties of $\psi$ \cite{Feichtinger:1989,Grochenig:2004,Feichtinger:1997}.

The aim of this paper is to show that Gabor frames over a lattice $\Lambda\subset \Rdd$ are a special case of a more general situation, namely that Hilbert-Schmidt operators on $L^2(\Rd)$ can be used to define a frame-like structure for $L^2(\Rd)$. These structures are obtained by shifting a "window" operator $S$ over $\Lambda$ by the operation
\begin{equation*}
  \alpha_z(S)=\pi(z)S\pi(z)^* \quad \text{ for }z\in \Rdd.
\end{equation*}
 Following Werner \cite{Werner:1984} and Kozek \cite{Kozek:1992} we consider $\alpha_\lambda(S)$ to be a translation of $S$ by $\lambda$.  Our main definition is that $S$ generates a \textit{Gabor g-frame} for $L^2(\Rd)$ if there exist constants $A,B>0$ such that
\begin{equation} \label{eq:intro:framedef}
  A\|\psi\|_{L^2}^2\leq \sum_{\lambda \in \Lambda} \|\alpha_\lambda(S)\psi\|_{L^2}^2 \leq B\|\psi\|_{L^2}^2 \quad \text{ for } \psi\in L^2(\Rd).
\end{equation}
When $S$ is a rank-one operator we recover the definition of Gabor frames -- more generally we obtain multi-window Gabor frames \cite{Zibulski:1997} if $S$ is of finite rank.  If \eqref{eq:intro:framedef} holds, the associated \textit{g-frame operator} $\frameop_S$ given by 
\begin{equation} \label{eq:intro:frameop}
  \frameop_S(\psi)=\sum_{\lambda\in \Lambda} \alpha_ \lambda (S^*S)\psi, 
\end{equation}
is bounded and invertible on $L^2(\Rd)$, and we show that this operator is the composition of two other natural operators: the analysis and synthesis operators. A major goal of this paper is to show that although Gabor g-frames are not frames, they nevertheless share much of the structure of Gabor frames. Our terminology stems from the fact that Gabor g-frames are examples of g-frames as introduced by Sun \cite{Sun:2006}, but apart from terminology the abstract theory of g-frames does not feature much in this paper. 

\subsubsection*{Fourier series of operators and the Janssen representation} 

Our investigations into the structure of Gabor g-frames naturally lead to the study of a notion of Fourier series of operators, inspired by the analysis of periodic operators by Feichtinger and Kozek \cite{Feichtinger:1998} and the quantum harmonic analysis of Werner \cite{Werner:1984}. By Fourier series for operators we mean that a $\Lambda$-periodic operator $T$ -- meaning that $\alpha_\lambda(T)=T$ for all $\lambda\in \Lambda$ -- has an expansion of the form
\begin{equation} \label{eq:intro:fourierseries}
  T= \sum_{\lambda^\circ \in \Lambda^\circ} c_{\lambda^\circ} e^{-\pi i \lambda^\circ _x \cdot \lambda^\circ _\omega} \pi(\lambda^\circ).
\end{equation}
Here $\Lambda^\circ$ is the adjoint lattice of $\Lambda$ defined in Section \ref{sec:janssen}, and we write $\lambda^\circ=(\lambda^\circ_x,\lambda^\circ_\omega)$. Such expansions have also been studied in \cite{Feichtinger:1998}, and the interpretation that this is a Fourier series of operators follows from considering the operator $e^{-i\pi x\cdot \omega}\pi(z)$ for $z=(x,\omega)\in \Rdd$ as the operator-analogue of the character $t\mapsto e^{2\pi i z\cdot t}$ on $\Rdd$. This interpretation is strengthened by the fact that an analogue of Wiener's classical lemma for absolutely summable Fourier series also holds for operators, by a result of Gr\"ochenig and Leinert \cite{Grochenig:2004}.  We show that any $\Lambda$-periodic bounded operator on $L^2(\Rd)$ has a Fourier series expansion \eqref{eq:intro:fourierseries}. This is not the only possible approach to Fourier series of operators, see for instance \cite{Bochner:1942,deLeeuw:1973,deLeeuw:1975,deLeeuw:1977}, and we also remark that periodic operators have been studied in \cite[Prop. 5.5]{Balazs:2014inverse}.

Due to the form of the Gabor g-frame operator \eqref{eq:intro:frameop} it is particularly interesting to study the Fourier series expansion of periodic operators $T$ given by a periodization over $\Lambda$: $$T=\sum_{\lambda \in \Lambda} \alpha_\lambda(R)$$ for some operator $R$. This leads to the following \textit{Poisson summation formula for trace class operators}: if $R$ is a trace class operator, then
\begin{equation} \label{eq:intro:poisson}
  \sum_{\lambda \in \Lambda} \alpha_\lambda(R)=\frac{1}{|\Lambda|} \sum_{\lambda^\circ \in \Lambda^\circ}\F_W(R)(\lambda^\circ)e^{-\pi i \lambda^\circ _x \cdot \lambda^\circ _\omega} \pi(\lambda^\circ),
\end{equation}
where $\F_W$ is the Fourier-Wigner transform of $R$ defined by 
\begin{equation*}
  \F_W(R)(z)=e^{-\pi i x\cdot \omega} \tr(\pi(-z)R) \quad \text{ for } z=(x,\omega)\in \Rdd,
\end{equation*}
which Werner \cite{Werner:1984} argued is a Fourier transform of operators. Showing that \eqref{eq:intro:poisson} holds for all trace class operators requires a careful study of the continuity of several mappings. Equation \eqref{eq:intro:poisson} is an analogue of the usual Poisson summation formula for functions: the Fourier coefficients of a periodization $\sum_{\lambda \in \Lambda} \alpha_\lambda(R)$ is given by the samples of the Fourier transform of $R$. Comparing \eqref{eq:intro:poisson} with \eqref{eq:intro:frameop}, we obtain an alternative expression for the g-frame operator of a Gabor g-frame which generalizes the Janssen representation for Gabor frames. This generalized Janssen representation allows us to deduce an extension of the Wexler-Raz biorthogonality conditions to Gabor g-frames, and to establish painless procedures for making Gabor g-frames using \textit{underspread} operators.

\subsubsection*{Time-frequency localization and Gabor g-frames}
The definition \eqref{eq:intro:framedef} has a particularly interesting interpretation if $\alpha_\lambda(S)\psi$ can, in some sense, be interpreted as the part of the signal $\psi$ localized around the point $\lambda$ in the time-frequency plane $\Rdd$. In this case, one may interpret $\|\alpha_\lambda(S)\psi\|_{L^2}$ as a measure of the part of $\psi$ localized around $\lambda$ in the time-frequency plane. For instance, picking a rank-one operator $S=\varphi\otimes \varphi$ for $\varphi\in L^2(\Rd)$, one finds that $\|\alpha_\lambda(S)\psi\|_{L^2}=|V_\varphi \psi(\lambda)|$, which is the measure of localization of $\psi$ around $\lambda$ used in Gabor frames. Another prime example of operators $S$ where $\alpha_\lambda(S)\psi$ has this interpretation are the localization operators $A_{\chi_\Omega}^\varphi$ with domain $\Omega\subset \Rdd$ and window $\varphi\in L^2(\Rd)$ introduced by Daubechies \cite{Daubechies:1988,Cordero:2003,Dorfler:2014}, and the inequalities \eqref{eq:intro:framedef} have been studied for such operators by D\"orfler, Feichtinger and Gr\"ochenig \cite{Dorfler:2006,Dorfler:2011}. The results of \cite{Dorfler:2006,Dorfler:2011} are therefore a second important class examples of Gabor g-frames in addition to (multi-window) Gabor frames.

In our terminology, \cite{Dorfler:2006,Dorfler:2011} showed that if $A_{\chi_\Omega}^\varphi$ generates a Gabor g-frame with well-behaved window $\varphi$, then weighted $\ell^p$-norms of $\{\|\alpha_\lambda(A_{\chi_\Omega}^\varphi)\psi\|_{L^2}\}_{\lambda\in \Lambda}$ are equivalent to the norm of $\psi$ in  modulation spaces. By the properties of modulation spaces, this implies that smoothness and decay properties of $\psi$ are captured by the coefficients $\{\|\alpha_\lambda(A_{\chi_\Omega}^\varphi)\psi\|_{L^2}\}_{\lambda \in \Lambda}$. A similar result is well-known for Gabor frames \cite{Feichtinger:1989,Grochenig:2001,Feichtinger:1997}, and in Corollary \ref{cor:equivnorms} we extend this to a result for Gabor g-frames that includes Gabor frames and localization operators as special cases.

The fact that the results of \cite{Dorfler:2006,Dorfler:2011} can be incorporated into the theory of Gabor g-frames allows us to understand exactly how a signal $\psi$ is recovered from its time-frequency localized components $\psi_\lambda:=\alpha_\lambda(A_{\chi_\Omega}^\varphi) \psi$ for $\lambda \in \Lambda$. In fact, we show that $A_{\chi_\Omega}^\varphi$ has a canonical dual operator $R$, such that
\begin{equation*}
  \psi=\sum_{\lambda \in \Lambda}\alpha_\lambda(R^*)\psi_\lambda \quad \text{ for any } \psi \in L^2(\Rd).
\end{equation*}
This is a generalization of a well-known fact for Gabor frames to Gabor g-frames (and in particular the localization operators of \cite{Dorfler:2006,Dorfler:2011}), namely that if $\varphi\in L^2(\Rd)$ generates a Gabor frame, then there is a canonical dual window $\varphi'\in L^2(\Rd)$ with 
\begin{equation*}
  \psi=\sum_{\lambda\in \Lambda} V_\varphi \psi(\lambda) \pi(\lambda)\varphi' \quad \text{ for any } \psi \in L^2(\Rd).
\end{equation*}

\subsubsection*{Cohen's class and Gabor g-frames} A different perspective on Gabor g-frames uses Cohen's class of time-frequency distributions \cite{Cohen:1966}. In the formalism of \cite{Luef:2018b}, $\|\alpha_\lambda(S)\psi\|_{L^2}^2$ equals $Q_{S^*S}(\psi)(\lambda)$, where $Q_{S^*S}$ is the Cohen's class distribution associated with the operator $S^*S$ as defined in \cite{Luef:2018b}. Hence equation \eqref{eq:intro:framedef} states that the $\ell^1$-norm of the samples $\{Q_{S^*S}(\psi)(\lambda)\}_{\lambda \in \Lambda}$ should be an equivalent norm on $L^2(\Rd)$. A simple example of a Cohen's class distribution is the spectrogram $|V_{\varphi}\psi(z)|^2$ for a window $\varphi$, which corresponds to picking rank-one $S$. Hence the move from Gabor frames to Gabor g-frames corresponds to replacing the spectrogram by a more general Cohen's class distribution, and we show that much of the structure of Gabor frames is preserved.

\subsubsection*{Technical tools}
We give a brief overview of the non-standard technical tools needed to prove the results of the paper. We will utilize a Banach subspace $\beauty$ of the trace class operators, as studied by \cite{Feichtinger:1998,Cordero:2008,Feichtinger:2009}. The space $\beauty$ consists of operators with kernel (as integral operators) in the so-called Feichtinger algebra \cite{Feichtinger:1981}, and we aim to show readers with backgrounds in other areas than time-frequency analysis the usefulness of $\beauty$. For instance, if $R\in \beauty$ the sum on the right hand side of \eqref{eq:intro:poisson} converges absolutely in the operator norm. The same will hold if we pick $R$ from the smaller space of Schwartz operators \cite{Keyl:2016}, but the Schwartz operators do not form a Banach space. Hence $\beauty$ combines desirable features from the trace class operators and the Schwartz operators: it is a Banach space, yet small enough to have properties not shared by arbitrary trace class operators. A new aspect in this paper is that we also develop a theory of weighted versions of $\beauty$, and we use the projective tensor product of Banach spaces to establish a decomposition of operators in the weighted $\beauty$-spaces in terms of rank-one operators. 

We will also use the dual space $\beast$ with its weak* topology. The sums in the Poisson summation formula \eqref{eq:intro:poisson} for trace class operators converge in this topology, but not necessarily in the weak* topology of the bounded operator $\bo$ -- hence $\beast$ is necessary even for studying trace class operators. 

In order to write the g-frame operator \eqref{eq:intro:frameop} as the composition of an analysis operator and a synthesis operator we will need the $L^2$-valued sequence spaces $\ell^p_m(\Lambda;L^2)$, consisting of sequences $\{\psi_\lambda\}_{\lambda \in \Lambda}\subset L^2(\Rd)$ such that $\sum_{\lambda\in \Lambda} \|\psi_\lambda\|_{L^2}^pm(\lambda)^p <\infty$, where $m$ is a weight function. The use of these Banach spaces is key to reducing statements about Gabor g-frames to known results for Gabor frames in Section \ref{sec:eqnorms}. 

\subsubsection*{Organization}  
We recall some definitions and results from time-frequency analysis, pseudodifferential operators and g-frames in Section \ref{sec:prelim}. Section \ref{sec:operators} is devoted to introducing and studying one of our main tools: Banach spaces of operators with kernels in certain weighted function spaces and their  decomposition into rank-one operators. The definition and basic properties of Gabor g-frames are given in Section \ref{sec:ggf}. The theory of Fourier series of operators and its applications to Gabor g-frames, including a Janssen representation and Wexler-Raz biorthogonality for Gabor g-frames, is explored in Section \ref{sec:janssen}. Section \ref{sec:eqnorms} is devoted to using Gabor g-frames to obtain equivalent norms for modulation spaces. Finally the relation of Gabor g-frames to countably generated multi-window Gabor frames using the singular value decomposition is explained in Section \ref{sec:singval}.

\section{Notation and conventions}
By a lattice $\Lambda$ we mean a full-rank lattice in $\Rdd$, i.e. $\Lambda=A\Z^{2d}$ for some  $A\in GL(2d,\R)$. The volume of $\Lambda=A\Z^{2d}$ is $|\Lambda|:=\det(A)$. The Haar measure on $\Rdd/\Lambda$ will always be normalized so that $\Rdd/\Lambda$ has total measure $1$. 

If $X$ is a Banach space and $X'$ its dual space, the action of $y\in X'$ on $x\in X$ is denoted by the bracket $\inner{y}{x}_{X',X}$, where the bracket is antilinear in the second coordinate to be compatible with the notation for inner products in Hilbert spaces. This means that we are identifying the dual space $X'$ with \textit{anti}linear functionals on $X$. For two Banach spaces $X,Y$ we denote by $\mathcal{L}(X,Y)$ the Banach space of bounded linear operators $S:X\to Y$, and if $X=Y$ we simply write $\mathcal{L}(X)$. The notation $X\hookrightarrow Y$ denotes a norm-continuous embedding of Banach spaces. 

For $p\in [1,\infty]$, $p'$ denotes the conjugate exponent, i.e. $\frac{1}{p}+\frac{1}{p'}=1$. The notation $P \lesssim Q$ means that there is some $C>0$ such that $P\leq C\cdot Q$, and $P\asymp Q$ means that $Q\lesssim P$ and $P\lesssim Q$. For $\Omega \subset \Rdd$, $\chi_\Omega$ is the characteristic function of $\Omega.$
 
\section{Preliminaries} \label{sec:prelim}
\subsection{Time-frequency analysis and modulation spaces}
The fundamental operators in time-frequency analysis are the translation operators $T_x$ and the modulation operators $M_\omega$ for $x,\omega\in \Rd$, defined by
\begin{align*}
 & (T_x\psi)(t)=\psi(t-x), & (M_\omega \psi)(t)=e^{2\pi i \omega \cdot t}\psi(t) \quad \text{ for } \psi \in L^2(\Rd).
\end{align*}
By composing these operators, we get the time-frequency shifts $\pi(z):=M_\omega T_x$ for $z=(x,\omega)\in \Rdd$, given by
\begin{equation*}
  (\pi(z)\psi)(t)=e^{2\pi i \omega \cdot t}\psi(t-x) \quad \text{ for } \psi \in L^2(\Rd).
\end{equation*}
The time-frequency shifts $\pi(z)$ are unitary operators on $L^2(\Rd)$, with adjoint $\pi(z)^*=e^{-2\pi i x\cdot \omega}\pi(-z)$ for $z=(x,\omega)$. For $\psi,\phi\in L^2(\Rd)$ we use the time-frequency shifts to define the \textit{short-time Fourier transform } $V_\phi \psi$ of $\psi$ with window $\phi$ by 
\begin{equation} \label{eq:stft}
  V_\phi \psi (z)=\inner{\psi}{\pi(z)\phi}_{L^2} \quad \text{ for } z\in \Rdd.
\end{equation}
The short-time Fourier transform satisfies an orthogonality condition, sometimes called Moyal's identity \cite{Grochenig:2001,Folland:1989}.
\begin{lemma}[Moyal's identity]
If $\psi_1, \psi_2, \phi_1, \phi_2 \in L^2(\R^d)$, then $V_{\phi_i}\psi_j \in L^2(\R^{2d})$ for $i,j\in \{1,2\}$ and 
\begin{equation*}
	\inner{V_{\phi_1}\psi_1}{V_{\phi_2}\psi_2}_{L^2}=\inner{\psi_1}{\psi_2}_{L^2}\overline{\inner{\phi_1}{\phi_2}}_{L^2},
\end{equation*}
 where the leftmost inner product is in $L^2(\R^{2d})$ and those on the right are in $L^2(\Rd)$.
\end{lemma}

\subsubsection{Weight functions}
To define the appropriate function spaces for our setting -- the modulation spaces -- we need to consider weight functions on $\Rdd$. In this paper, a \textit{weight function} is a continuous and positive function on $\Rdd$. We will always let $v$ denote a \textit{submultiplicative weight function satisfying the GRS-condition.} That $v$ is submultiplicative means that $$v(z_1+z_2)\leq v(z_1)v(z_2) \text{ for any } z_1,z_2\in \Rdd,$$ and the GRS-condition says that $$\lim_{n\to \infty} (v(nz))^{1/n}=1 \text{ for any } z\in \Rdd.$$ Furthermore, we will assume that $v$ is symmetric in the sense that $v(x,\omega)=v(-x,\omega)=v(x,-\omega)=v(-x,-\omega)$ for any $(x,\omega)\in \Rdd$, which along with submultiplicativity implies that $v\geq 1$ \cite{Grochenig:2007w}. 

By $m$ we will always mean a weight function that is \textit{v-moderate}; this means that 
\begin{equation} \label{eq:vmoderate}
  m(z_1+z_2)\lesssim m(z_1)v(z_2) \text{ for any } z_1,z_2\in \Rdd.
\end{equation}
The interested reader is encouraged to consult the survey \cite{Grochenig:2007w} for an excellent exposition of the reasons for making these assumptions in time-frequency analysis. The less interested reader may safely assume that all weights are polynomial weights 
$
  v_s(z)=(1+|z|)^s 
$
for some $s\geq 0.$ 

\subsubsection{Modulation spaces}
 Let  $\phi_0$ be the normalized (in $L^2$-norm) Gaussian $\phi_0(x)=2^{d/4}e^{-\pi x\cdot x}$ for $x\in \Rd$, and let $v$ be a submultiplicative, symmetric GRS-weight. We first define the space $M^1_v(\Rd)$ to be the space of $\psi\in L^2(\Rd)$ such that
\begin{equation*}
  \|\psi\|_{M^1_v} := \int_{\Rd} \int_{\Rd}|V_{\phi_0}\psi(z)| v(z) \,dz <\infty.
\end{equation*}
 For $p\in [1,\infty]$ and a $v$-moderate weight function $m$ we then define the \textit{modulation space} $M_{m}^{p}(\Rd)$ to be the set of $\psi$ in the (antilinear) dual space $\left(M^1_v(\Rd)\right)'$ with
\begin{equation} \label{eq:modspacenorm}
  \|\psi\|_{M^{p}_m} := \left(\int_{\Rd} \int_{\Rd}|V_{\phi_0}\psi(z)|^p m(z)^p \,dz \right)^{1/p}<\infty,
\end{equation}
where the integral is replaced by a supremum in the usual way when $p=\infty.$
In \eqref{eq:modspacenorm}, $V_{\phi_0}\psi$ must be interpreted by (antilinear) duality, meaning that we extend the definition in equation \eqref{eq:stft} by defining 
\begin{equation*}
  V_{\phi_0}\psi(z)=\inner{\psi}{\pi(z)\phi_0}_{(M^1_v)',M^1_v}.
\end{equation*}
For $m\equiv 1$ we will write $M^p(\Rd):=M^p_{m}(\Rd)$. We summarize a few of the useful properties of modulation spaces in a proposition, see \cite{Grochenig:2001} for the proofs.

\begin{proposition} \label{prop:modulationspaces}
Let $m$ be a $v$-moderate weight and $p\in [1,\infty]$.
	\begin{enumerate}[(a)]
		\item $M^p_{m}(\Rd)$ is a Banach space with the norm defined in \eqref{eq:modspacenorm}.
		\item If we replace $\phi_0$ with another function $0\neq \phi\in M^1_{v}(\Rd)$ in \eqref{eq:modspacenorm}, we obtain the same space $M_{m}^p(\Rd)$ as with $\phi_0$, with equivalent norms.
		\item If $1\leq p_1\leq p_2\leq \infty$ and $m_2\lesssim m_1$, then $M_{m_1}^{p_1}(\Rd)\hookrightarrow M_{m_2}^{p_2}(\Rd)$. 
		\item If $p<\infty$ and $\frac{1}{p}+\frac{1}{p'}=1$, then $M_{1/m}^{p'}(\Rd)$ is the dual space of $M^{p}_{m}(\Rd)$ with
 \begin{equation} \label{eq:duality}
  \inner{\phi}{\psi}_{M^{p'}_{1/m},M^{p}_m}=\int_{\Rdd} V_{\phi_0}\phi(z)\overline{V_{\phi_0}\psi(z)} \ dz.
\end{equation}
		\item The operators $\pi(z)$ can be extended to bounded operators on $M^{p}_m(\Rd)$ with $\|\pi(z)\psi\|_{M^{p}_m}\lesssim v(z)\|\psi\|_{M^{p}_m}$ for $\psi\in M^{p}_m(\Rd)$.
		\item $L^2(\Rd)=M^2(\Rd)$ with equivalent norms.
		\item $M^1_v(\Rd)$ is dense in $M^p_m(\Rd)$ for $p<\infty$ and weak*-dense in $M^\infty_m(\Rd)$.
	\end{enumerate}
\end{proposition}

\begin{remark}
\begin{enumerate}[(a)]
\item Assume that $p<\infty$. If $\phi \in L^2(\Rd)\cap M^{p'}_{1/m}(\Rd)$ and $\psi\in M^p_m(\Rd)\cap L^2(\Rd)$, then Moyal's identity and \eqref{eq:duality} implies that $\inner{\phi}{\psi}_{M^{p'}_{1/m},M^p_m}=\inner{\phi}{\psi}_{L^2}$. We will use this fact several times in the rest of the paper.  
	\item We defined modulation spaces as subspaces of $(M^1_v(\Rd))'=M^\infty_{1/v}(\Rd)$. If one restricts to weights $v$ of at most polynomial growth, then $M^1_v(\Rd)$ contains the Schwartz functions $\mathcal{S}(\Rd)$ and $M^\infty_{1/v}(\Rd)$ is a subspace of the tempered distributions $\mathcal{S}'(\Rd)$ \cite{Grochenig:2007w}.
	\item If $m$ is $v$-moderate, then so is $1/m$ since we assume that $v$ is symmetric: for $w_1,w_2\in \Rdd$ we find by choosing $z_1=w_1+w_2$ and $z_2=-w_2$ in \eqref{eq:vmoderate} that
	$
  m(w_1)\lesssim m(w_1+w_2)v(w_2),
$
hence
\begin{equation*}
  \frac{1}{m(w_1+w_2)} \lesssim \frac{1}{m(w_1)}v(w_2). 
\end{equation*}
The class of modulation spaces is therefore closed under duality for $p<\infty$. 
\end{enumerate}
	
\end{remark}

 \subsubsection{Wiener amalgam spaces and sampling estimates}
 Some close relatives of the modulation spaces are the \textit{Wiener amalgam spaces}. For our purposes, these spaces are interesting because they are associated with certain sampling estimates. We first define, for $1\leq p < \infty$, any lattice $\Lambda$ and weight function $m$,  the weighted sequence spaces
 $$\ell^p_{m}(\Lambda)=\left\{\{c_\lambda\}_{\lambda \in \Lambda}\subset \C : \|c\|^p_{\ell^p_{m}}:= \sum_{\lambda \in \Lambda} |c_\lambda|^pm(\lambda)^p<\infty\right\},$$ 
 and $\ell^\infty_{m}(\Lambda)$ is defined by replacing the sum by a supremum in the usual way. 
 
 Given any function $f:\Rdd\to \C$ we define a sequence $\{a_{(k,l)}\}_{(k,l)\in \mathbb{Z}^{2d}}$ by 
 \begin{equation*}
  a_{(k,l)}=\sup_{x,\omega\in [0,1]^d} |f(x+k,\omega+l)|;
\end{equation*}
the Wiener amalgam space $W(L^p_m)$ on $\Rdd$ is then the Banach space of $f:\Rdd\to \C$ such that
\begin{equation*}
  \|f\|_{W(L^p_m)}:= \|\{a_{(k,l)} \}\|_{\ell^p_m(\Z^{2d})}<\infty.
\end{equation*}
 The following is Proposition 11.1.4 in \cite{Grochenig:2001}. 
 \begin{lemma} \label{lem:wienersampling}
	Let $\Lambda$ be a lattice in $\Rdd$, and assume that $f\in W(L^p_m)$ is continuous. Then
	\begin{equation*}
		\|f \vert_{\Lambda}\|_{\ell^p_{m}} \lesssim \|f\|_{W(L^p_m)},
	\end{equation*}
	where the implicit constant may be chosen to be independent of $p$ and $m$.
	Since $M^1(\Rdd)\hookrightarrow W(L^1_m)$ for $m\equiv 1$, it follows that $\|f \vert_{\Lambda}\|_{\ell^1} \lesssim \|f\|_{M^1}$ for $f\in M^1(\Rdd)$. 
\end{lemma}
 
   By combining \cite[Lem. 4.1]{Cordero:2003} with Lemma \ref{lem:wienersampling}, one obtains the following result.
 \begin{lemma} \label{lem:lpstft}
	Let $\Lambda$ be a lattice, $\phi \in M^1_{v}(\Rd)$ and $\psi \in M^{p}_{m}(\Rd)$ where $p\in [1,\infty]$. Then
\begin{equation*}
  \|V_\phi \psi\vert_{\Lambda}\|_{\ell^{p}_{m}(\Lambda)}\lesssim \|\phi\|_{M^1_{v}} \|\psi\|_{M^{p}_{m}},
\end{equation*}
where the implicit constant may be chosen to be independent of $p$ and $m$.
\end{lemma}

\subsubsection{The symplectic Fourier transform}
As the Fourier transform of functions $f$ on $\Rdd$, we will use the \textit{symplectic} Fourier transform $\F_{\sigma} f$, given by 
\begin{equation*}
\F_{\sigma} f(z)=\int_{\R^{2d}} f(z') e^{-2 \pi i \sigma(z,z')} \ dz' \quad \text{ for } f\in L^1(\Rdd),z\in \R^{2d},
\end{equation*}
 where $\sigma$ is the standard symplectic form $\sigma((x_1,\omega_1),(x_2, \omega_2))=\omega_1\cdot x_2-\omega_2 \cdot x_1$. Then $\F_\sigma$ is an isomorphism on $M^{1}(\Rdd)$, and extends to a unitary operator on $L^2(\Rdd)$ and an isomorphism on $M^\infty(\Rdd)$ \cite[Lem. 7.6.2]{Feichtinger:1998}.

\subsection{Trace class and Hilbert-Schmidt operators} \label{sec:tchs}
By the singular value decomposition, see Chapter 3.2 of \cite{Busch:2016}, any compact operator $S$ on $L^2(\Rd)$ may be written as 
	\begin{equation*} 
  S=\sum_{n=1}^{N_0} s_n \xi_n \otimes \varphi_n
\end{equation*}
for some $N_0\in \mathbb{N}\cup \{\infty\}$, two orthonormal systems $\left\{ \xi_n \right\}_{n=1}^{N_0},$ $\left\{ \varphi_n \right\}_{n=1}^{N_0}$ in $L^2(\Rd)$ and a sequence of positive numbers $\left\{ s_n \right\}_{n=1}^{N_0}\in \ell^\infty$ called the \textit{singular values} of $S$. Here $\xi\otimes \varphi$ denotes the rank-one operator $\xi\otimes \varphi(\psi)=\langle \psi,\varphi \rangle_{L^2} \xi$ for $\varphi,\xi,\psi \in L^2(\mathbb{R}^d)$. We assume that $s_{n+1}\geq s_n$ for $n\in \N$.

Imposing summability conditions on the singular values of $S$ allows us to define two important classes of operators. The \textit{trace class operators} $\mathcal{T}$ are the operators $S$ whose singular values satisfy $\{s_n\}_{n=1}^{N_0}\in \ell^1$. The norm $\|S\|_{\tco}=\|\{s_n\}\|_{\ell^1}$ makes $\tco$ into a Banach space \cite{Busch:2016}. We may define a bounded linear functional on $\tco$ called the \textit{trace} by 
\begin{equation*}
  \tr(S):=\sum_{n\in \N} \langle S\eta_n,\eta_n \rangle,
\end{equation*}
where $\{\eta_n\}_{n\in \N}$ is an orthonormal basis for $L^2(\Rd)$ -- the value of $\tr(S)$ can be shown to be independent of the orthonormal basis used in its definition \cite{Busch:2016}. We also mention that the norm on $\tco$ may be expressed by $\|S\|_\tco=\tr(|S|)$.

The \textit{Hilbert-Schmidt operators} $\HS$ are the operators $S$ where $\{s_n\}_{n=1}^{N_0}\in \ell^2$. The norm on $\HS$ can be expressed as the $\ell^2$ norm of the singular values, but it will be more useful to note that $ST\in \tco$ for any $S,T\in \HS$ and that $\HS$ becomes a Hilbert space with respect to the inner product \cite{Busch:2016}
\begin{equation*}
  \langle S,T \rangle_{\HS}:=\tr(ST^*).
\end{equation*} 
Another description of $\HS$ is obtained by noting that it is isomorphic to the Hilbert space tensor product $L^2(\Rd) \otimes L^2(\Rd)$, where the isomorphism is obtained by associating rank-one operators $\psi\otimes \varphi \in \HS$ with elementary tensors $\psi\otimes \varphi$ \cite[Appendix 3]{Folland:2016}.

\subsection{Pseudodifferential operators} 
We will consider different ways to associate functions on $\Rdd$ with operators $M^1(\Rd)\to M^\infty(\Rd)$.
\subsubsection{Integral operators}
 For $\kernel\in L^2(\Rdd)$, we define a necessarily bounded integral operator $S:L^2(\Rd)\to L^2(\Rd)$ by 
 \begin{equation}\label{eq:hilbertschmidt}
  S\psi(x) = \int_{\Rd} \kernel(x,y) \psi(y) \ dy \quad \text{ for } \psi \in L^2(\Rd).
\end{equation}
Here $\kernel=\kernel_S$ is the \textit{kernel} of $S$, and one can extend the definition above to $\kernel\in M^\infty(\Rdd)$ by defining $S:M^1(\Rd)\to M^\infty(\Rd)$ by duality:
\begin{equation*} 
  \inner{S\psi}{\phi}_{M^\infty,M^1}=\inner{\kernel}{\phi \otimes \overline{\psi}}_{M^\infty,M^1} \quad \text{ for } \phi,\psi \in M^1(\Rd),
\end{equation*}
where $\phi \otimes \overline{\psi}(x,y)=\phi(x) \overline{\psi(y)}.$ By the kernel theorem for modulation spaces \cite[Thm. 14.4.1]{Grochenig:2001}, any continuous linear operator $S:M^1(\Rd)\to M^\infty(\Rd)$ is induced by a unique kernel $\kernel=\kernel_S\in M^\infty(\Rdd)$ in this way. Writing operators using a kernel $\kernel$ will be particularly useful for us because 
\begin{equation} \label{eq:kernelrankone}
  \kernel_{\phi \otimes \psi}=\phi\otimes \overline{\psi}\quad \text{ for } \psi,\phi \in L^2(\Rd),
\end{equation}
where $\phi \otimes \psi$ on the left side denotes the rank-one operator $\phi\otimes \psi(\xi)=\inner{\xi}{\psi}_{L^2}\phi$, and on the right side the \textit{function} $\phi\otimes \overline{\psi}(x,y)=\phi(x)\overline{\psi}(y)$. The Hilbert-Schmidt operators are precisely those operators  $S:M^1(\Rd)\to M^\infty(\Rd)$ such that $\kernel_S\in L^2(\Rdd)$.
\subsubsection{The Weyl calculus}
For $\xi,\eta \in L^2(\Rd)$, the \textit{cross-Wigner distribution} $W(\xi,\eta)$ is given by
\begin{equation*}
  W(\xi,\eta)(x,\omega)=\int_{\R^d} \xi\left(x+\frac{t}{2}\right)\overline{\eta\left(x-\frac{t}{2}\right)} e^{-2 \pi i \omega \cdot t} \ dt \quad \text{ for } (x,\omega)\in \Rdd.
 \end{equation*}	
Using the cross-Wigner distribution we introduce the \textit{Weyl calculus}. For $f \in M^\infty(\R^{2d})$ and $\xi,\eta \in M^1(\R^d)$, we define the \textit{Weyl transform} $L_{f}$ of $f$ to be the operator $L_f:M^1(\Rd)\to M^\infty(\Rd)$ given by 
\begin{equation*}
  \inner{L_{f}\eta}{\xi}_{M^\infty,M^1}=\inner{f}{W(\xi,\eta)}_{M^\infty,M^1}.
\end{equation*}
$f$ is called the \textit{Weyl symbol} of the operator $L_{f}$. In general we will use $\weyl_S$ to denote the Weyl symbol of an operator $S$, in other words $L_{\weyl_S}=S$. By the kernel theorem for modulation spaces, the Weyl transform is a bijection from $M^\infty(\Rdd)$ to the continuous  linear operators $M^1(\Rd)\to M^\infty(\Rd)$. As above, $\HS$ has a simple description in terms of the Weyl symbol: $S\in \HS$ if and only if $\weyl_S\in L^2(\Rdd)$. 

\subsubsection{Translation of operators}

Several authors have considered the idea of translating operators by a point $z\in \Rdd$ by conjugation with $\pi(z)$ \cite{Kozek:1992, Feichtinger:1998, Werner:1984}: if $S:M^1(\Rd)\to M^\infty(\Rd)$ is a continuous operator, we define the translation of $S$ by $z\in \Rdd$ to be 
\begin{equation*}
  \alpha_z(S)=\pi(z)S\pi(z)^*.
\end{equation*}
This corresponds to a translation of the Weyl symbol \cite[Lem. 3.2]{Luef:2017},  
\begin{equation} \label{eq:translateweyl}
  \alpha_z(S)=L_{T_z(\weyl_S)},
\end{equation}
which is a major reason why the Weyl symbol is useful for us when considering Fourier series of operators in Section \ref{sec:janssen}. Since $\pi(z)$ is unitary, $\alpha$ also respects the product of two operators in the sense that
\begin{equation} \label{eq:translateproduct}
\alpha_z(ST)=\alpha_z(S)\alpha_z(T) \quad \text{ for } S,T \in \bo.
\end{equation}

It is easily shown that $\alpha_z$ is an isometry on $\tco,\HS$ and $\bo$ for any $z\in \Rdd$ and that applying $\alpha_z$ to a rank-one operator $\psi \otimes \phi$ amounts to a time-frequency shift of $\psi$ and $\phi:$
\begin{equation} \label{eq:shiftrankone}
  \alpha_z(\psi\otimes \phi)=(\pi(z)\psi)\otimes(\pi(z)\phi).
\end{equation}
Furthermore, the map $z\mapsto \alpha_z$ is a representation of the locally compact abelian group $\Rdd$ on the space of Hilbert-Schmidt operators. In fact, if we identify the Hilbert-Schmidt operators with the Hilbert space tensor product $L^2(\Rd) \otimes L^2(\Rd)$, then $\alpha$ is the tensor product representation $\pi\otimes \overline{\pi}$ of $\Rdd$ on $L^2(\Rd) \otimes L^2(\Rd)$, which is the notation for $\alpha$ used in \cite{Feichtinger:1998}.
\subsubsection{The Fourier-Wigner transform}
 For a trace class operator $S\in \tco$, the \emph{Fourier-Wigner transform} $\F_W(S)$ of $S$ is the function
	\begin{equation*}
	\F_W (S)(z)=e^{-\pi i x \cdot \omega}\tr(\pi(-z)S) \quad \text{ for } z=(x,\omega)\in \R^{2d}.
	\end{equation*}
	As a special case, if $\psi,\phi \in L^2(\Rd)$ we have \cite[Lem. 6.1]{Luef:2017} that
	\begin{equation} \label{eq:fwrankone}
  \F_W(\phi\otimes \psi)(z)=e^{\pi i x \cdot \omega} V_\psi \phi (z) \quad \text{ for } z=(x,\omega) \in \Rdd,
\end{equation}
and we also mention the easily verified relation 
\begin{equation} \label{eq:fwofadjoint}
  \F_W(S^*)(z)=\overline{\F_W(S)(-z)} \quad \text{ for } z=(x,\omega) \in \Rdd.
\end{equation}
	Werner \cite{Werner:1984} has shown that in many respects $\F_W$ behaves like a Fourier transform for operators, which is the interpretation we will often rely on. For instance, a Riemann-Lebesgue lemma holds: if $S\in \tco$, then $\F_W(S)\in C_0(\Rdd)$ and  
	\begin{equation} \label{eq:riemannlebesgue}
  \|\F_W(S)\|_{L^\infty}\leq \|S\|_{\tco}.
\end{equation}
		The Fourier-Wigner transform and Weyl transform are related by a symplectic Fourier transform:
	\begin{equation} \label{eq:fwsymp}
  \F_W(S)=\F_\sigma (\weyl_S),
\end{equation}
which can be used to show that $S\in \HS$ if and only if $\F_W(S)\in L^2(\Rdd)$.
Finally, we remark that $\F_W(S)$ differs only by a phase factor $e^{-\pi i x\cdot \omega}$ from the \textit{spreading function} of $S$ \cite{Bello:1963,Feichtinger:1998}.

\subsubsection{Localization operators} 
An important class of examples of pseudodifferential operators in this paper will be the \textit{localization operators}. Given $\varphi\in L^2(\Rd)$ and $h\in L^1(\Rdd)$, the localization operator $A_h^\varphi \in \bo$ is defined by 
\begin{equation*}
  A_h^\varphi \psi=\int_{\Rdd} h(z) V_{\varphi}\psi(z) \pi(z) \varphi \ dz\quad \text{ for } \psi \in L^2(\Rd),
\end{equation*}
where the integral is an absolutely convergent Bochner integral in $L^2(\Rd)$. Localization operators interact nicely with the various aspects of pseudodifferential operators considered above: their Weyl symbol is given by a convolution \cite{Boggiatto:2004}
\begin{equation*}
  \weyl_{A_h^\varphi}=h\ast W(\varphi,\varphi)
\end{equation*}
and they satisfy the translation covariance property \cite[Lem. 4.3 and theorem. 5.1]{Luef:2017}
\begin{equation} \label{eq:translatelocop}
  \alpha_z(A_h^\varphi)=A_{T_z h}^\varphi
\end{equation}

\subsection{Frames and g-frames}
We will briefly recall the basic definitions of frame theory in the Hilbert space $L^2(\Rd)$, referring the details to the monographs \cite{Grochenig:2001,Christensen:2016,Heil:2011}. 
Recall that a sequence $\{\xi_i\}_{i\in I}\subset L^2(\Rd)$ is a \textit{frame} for $L^2(\Rd)$ if there exist constants $A,B>0$ such that 
\begin{equation} \label{eq:framedef}
  A \|\psi\|_{L^2}^2\leq \sum_{i\in I} |\langle \psi,\xi_i \rangle_{L^2}| \leq B \|\psi\|_{L^2}^2  \quad \text{ for any } \psi \in L^2(\Rd).
\end{equation}
Here $A$ and $B$ are called the  lower and upper frame bound, respectively.
If \eqref{eq:framedef} holds with $A=B$, we say that $\{\xi_i\}_{i\in I}$ is a \textit{tight} frame, and if $A=B=1$ we call $\{\xi_i\}_{i\in I}$ a \textit{Parseval frame}. Whenever the rightmost inequality in \eqref{eq:framedef} holds for some $B>0$, $\{\xi_i\}_{i\in I}$ is a \textit{Bessel system}.

When $\{\xi_i\}_{i\in I}$ is a Bessel system, we associated with $\{\xi_i\}_{i\in I}$ several bounded operators: the \textit{analysis operator} $C:L^2(\Rd)\to \ell^2(I)$ given by
\begin{equation*}
  C\psi=\{\langle \psi,\xi_i \rangle_{L^2}\}_{i\in I} \quad \text{ for } \psi \in L^2(\Rd),
\end{equation*}
the \textit{synthesis operator} $D:\ell^2(I)\to L^2(\Rd)$ given by
\begin{equation*}
  D\{c_i\}_{i\in I}=\sum_{i\in I} c_i \xi_i \quad \text{ for } \{c_i\}_{i\in I} \in \ell^2(I)
\end{equation*}
and the \textit{frame operator} $\frameop=DC\in \bo$ defined by
\begin{equation*} 
  \frameop(\psi)=\sum_{i\in I} \langle \psi,\xi_i \rangle_{L^2} \xi_i \quad \text{ for } \psi \in L^2(\Rd).
\end{equation*}

In the introduction, see equation \eqref{eq:intro:gaborframeop}, we introduced a special class of frames called \textit{Gabor frames}, which are frames of the form $\{\pi(\lambda)\varphi\}_{\lambda \in \Lambda}$ for some lattice $\Lambda$ and $\varphi\in L^2(\Rd)$. More generally, a \textit{multi-window Gabor frame }  \cite{Zibulski:1997} is a frame of the form $\{\pi(\lambda)\varphi_n\}_{\lambda \in \Lambda,n=1,...N}$ where $\varphi_n\in L^2(\Rd)$ for $n=1,...,N.$ We call the set $\{\pi(\lambda)\varphi_n\}_{\lambda \in \Lambda,n=1,...N}$ the \textit{multi-window Gabor system generated by } $\{\varphi_n\}_{n=1}^N$, even when $\{\pi(\lambda)\varphi_n\}_{\lambda \in \Lambda,n=1,...N}$ is not a frame.

\subsubsection{g-frames} 

In \cite{Sun:2006}, Sun introduced \textit{g-frames} as a generalization of frames for Hilbert spaces. We state a special case\footnote{More generally, we could consider $A_i\in \mathcal{L}(\mathcal{H},V_i)$ where $\mathcal{H}$ is a Hilbert space and $V_i$ is a closed subspace of another Hilbert space $\mathcal{H}'$, see \cite{Sun:2006}.} for the Hilbert space $L^2(\Rd)$. A sequence $\{A_i\}_{i\in I}\subset \bo$ is a g-frame for $L^2(\Rd)$ with respect to $L^2(\Rd)$ if there exist positive constants $A,B$ such that
\begin{equation*}
  A \|\psi\|_{L^2}^2\leq \sum_{i\in I} \|A_i\psi\|_{L^2}^2 \leq B \|\psi\|_{L^2}^2  \quad \text{ for any } \psi \in L^2(\Rd).
\end{equation*}
If we can choose $A=B$, we say that the g-frame is \textit{tight}. When the above inequality holds, the g-frame operator $\frameop$ defined by
\begin{equation*}
  \frameop\psi = \sum_{i\in \N} A_i^* A_i \psi
\end{equation*}
is positive, bounded and invertible on $L^2(\Rd)$ with $A\leq \|\frameop\|_{\bo}\leq B$. 

\section{The space $\beauty_{v\otimes v}$ of operators with kernel in $M^1_{v\otimes v}$}  \label{sec:operators}
To define a suitable class of operators for our purposes, we will consider modulation spaces on $\Rdd$. The short-time Fourier transform on phase space $\Rdd$ is $$\mathcal{V}_g f(z,\zeta)=\inner{f}{\pi(z)\otimes \pi(\zeta)g}_{L^2} \quad \text{ for } z,\zeta \in \Rdd\text{ and } f,g \in L^2(\Rdd),$$ where $\pi(z)\otimes \pi(\zeta)$ is defined by $$\pi(z)\otimes \pi(\zeta)g=M_{(z_\omega,\zeta_\omega)} T_{(z_x,\zeta_x)}g\quad \text{ for }z=(z_x,z_\omega),\zeta=(\zeta_x,\zeta_\omega).$$  Given a submultiplicative, symmetric GRS-weight $v$ on $\Rdd$, we consider the Banach space $M^1_{v\otimes v}(\Rdd)$ of $f \in L^2(\Rdd)$ such that
\begin{equation*}
  \|f\|_{M^1_{v\otimes v}} = \int_{\R^{2d}}\int_{\R^{2d}} |\mathcal{V}_{\phi_0\otimes \phi_0} f (z,\zeta)| v(z) v(\zeta)  \ dz d\zeta<\infty, 
\end{equation*}
where $\phi_0\otimes \phi_0(x,y)=\phi_0(x)\phi_0(y).$  With these definitions it is easy to show that if $\phi,\psi \in M^1_v(\Rd)$, then $\phi\otimes \psi \in M^1_{v\otimes v}(\Rdd)$ with 

 \begin{equation} \label{eq:rankoneprojtensor}
  \|\psi \otimes \phi\|_{M^1_{v\otimes v}} = \|\psi\|_{M^1_v}\|\phi\|_{M^1_v}.
\end{equation}
In fact, $M^1_{v\otimes v}(\Rdd)$ is isomorphic to $M^1_v(\Rd)\hat{\otimes} M^1_v(\Rd)$ \cite[Thm. 5]{Balazs:2019}, where $\hat{\otimes}$ denotes the projective tensor product of Banach spaces. This tensor product construction is covered in detail in \cite{Ryan:2002}, but for our purposes it suffices to note that 
\begin{equation} \label{eq:projtensor1}
  \begin{split}
  M^1_{v\otimes v}(\Rdd)&=M^1_v(\Rd)\hat{\otimes} M^1_v(\Rd)\\
  &=\left\{ \sum_{n\in \N} \phi_n^{(1)} \otimes \phi_n^{(2)}:\sum_{n\in \N} \|\phi_n^{(1)}\|_{M^1_v}\|\phi_n^{(2)}\|_{M^1_v}<\infty  \right\},
  \end{split}
\end{equation} 
with an equivalent norm for $M^1_{v\otimes v}(\Rdd)$ given by 
\begin{equation} \label{eq:projtensor2}
  \|f\|_{M^1_{v\otimes v}}\asymp \inf \left\{ \sum_{n\in \N} \|\phi_n^{(1)}\|_{M^1_v}\|\phi_n^{(2)}\|_{M^1_v} \right\},
\end{equation}
 where the infimum is taken over all sequences $\{\phi_n^{(1)}\}_{n\in \N},\{\phi_n^{(2)}\}_{n\in \N}$ in $M^1_v(\Rd)$ such that $f=\sum_{n\in \N} \phi_n^{(1)} \otimes \phi_n^{(2)}$ and $\sum_{n\in \N} \|\phi_n^{(1)}\|_{M^1_v}\|\phi_n^{(2)}\|_{M^1_v}<\infty $. 

We will be particularly interested in the class of operators $S$ whose kernel $\kernel_S$  belongs to $M^1_{v\otimes v}(\Rdd)$, as studied by several authors \cite{Feichtinger:2018,Feichtinger:1998,Kozek:2005} for $v\equiv 1$. 
 We denote the class of such operators by $\beauty_{v \otimes v}$, and define the norm
\begin{equation*}
  \|S\|_{\beauty_{v\otimes v}}=\|\kernel_S\|_{M^1_{v\otimes v}}.
\end{equation*}
 Since $M^1_{v\otimes v}(\Rdd)\hookrightarrow L^2(\Rdd)$, operators in $\beauty_{v\otimes v}$ define bounded operators on $L^2(\Rd)$ by \eqref{eq:hilbertschmidt}. In fact (see \cite{Grochenig:1996,Grochenig:1999}) we have $\beauty_{v\otimes v}\hookrightarrow \tco\hookrightarrow \bo$, hence
	\begin{equation*}
 \|S\|_{\bo}\leq  \|S\|_{\tco} \lesssim \|S\|_{\beauty_{v\otimes v}} \quad \text{ for } S\in \beauty_{v\otimes v}.
\end{equation*}
Now recall from \eqref{eq:kernelrankone} that the kernel of a rank-one operator $\phi \otimes \psi$ with $\psi,\phi \in M^1_v(\Rd)$  is the function $\phi \otimes \overline{\psi}$.
By \eqref{eq:rankoneprojtensor} we get that 
\begin{equation*} 
  \|\phi\otimes \psi\|_{\beauty_{v\otimes v}}=\|\phi\|_{M^1_v}\|\psi\|_{M^1_v}
\end{equation*}
(we have also used that $\|\overline{\psi}\|_{M^1_v}=\|\psi\|_{M^1_v}$ as $v$ is symmetric).
Equation \eqref{eq:projtensor1} therefore has the following important consequences.  
\begin{proposition}\label{prop:innerkernel}
	Let $S\in \beauty_{v\otimes v}$. 
	
	\begin{enumerate}[(a)]
		\item There exist sequences $\{\phi_n^{(1)}\}_{n\in \N},\{\phi_n^{(2)}\}_{n\in \N}\subset M^1_{v}(\Rd)$ with $$\sum_{n\in \N} \|\phi_n^{(1)}\|_{M^1_{v}} \|\phi_n^{(2)}\|_{M^1_{v}}<\infty$$ such that $S$ can be written as a sum of rank-one operators
	\begin{equation} \label{eq:sumrankone}
  S=\sum_{n\in \N} \phi_n^{(1)}\otimes \phi_n^{(2)}.
\end{equation} 
The decomposition \eqref{eq:sumrankone} converges absolutely in $\beauty_{v\otimes v}$, hence in $\tco$ and $\bo$.
		\item 
			\begin{equation*}
  				\|S\|_{\beauty_{v\otimes v}}\asymp \inf \left\{ \sum_{n\in \N} \|\phi_n^{(1)}\|_{M^1_v}\|\phi_n^{(2)}\|_{M^1_v} \right\},
			\end{equation*}
 with infimum taken over all sequences $\{\phi_n^{(1)}\}_{n\in \N},\{\phi_n^{(2)}\}_{n\in \N}$ as in (a).
 		\item Let $S^*$ denote the Hilbert space adjoint of $S$ when $S$ is viewed as an operator $L^2(\Rd)\to L^2(\Rd)$. Then $S^*\in \beauty_{v\otimes v}$ and $S$ extends to a weak*-to-weak*-continuous operator $S:M^\infty_{1/v}(\Rd)\to M^\infty_{1/v}(\Rd)$ by defining 
 		\begin{equation*} 
				\inner{S\phi}{\psi}_{M^\infty_{1/v},M^1_v}=\inner{\phi}{S^*\psi}_{M^\infty_{1/v},M^1_v} \quad \text{ for } \phi\in M^\infty_{1/v}(\Rd),\psi \in M^1_v(\Rd).
							\end{equation*}
				The decomposition in \eqref{eq:sumrankone} still holds for this extensions of $S$, meaning that
				\begin{equation*}
  						S\psi=\sum_{n\in \N} \inner{\psi}{\phi_n^{(2)}}_{M^\infty_{1/v},M^1_{v}} \phi_n^{(1)} \quad \text{ for } \psi\in M^\infty_{1/v}(\Rd)
        \end{equation*}
        with absolute convergence of the sum in the norm of $M^1_v(\Rdd)$.
		\item The extension of $S$ to $M^\infty_{1/v}(\Rd)$ is bounded from $M^\infty_{1/v}(\Rd)$ into $M^1_v(\Rd)$, and maps weak*-convergent sequences in $M^\infty_{1/v}(\Rd)$ to norm-convergent sequences in $M^1_v(\Rd)$. 
	\end{enumerate}

\end{proposition}  
\begin{proof}
\begin{enumerate}[(a)]
	\item By \eqref{eq:projtensor1}, there exist $\{\phi_n^{(1)}\}_{n\in N},\{\phi_n^{(2)}\}_{n\in N}$ as in the statement with 
		\begin{equation*}
		  \kernel_S (x,y)=\sum_{n\in \N} \phi_n^{(1)}(x)\overline{\phi_n^{(2)}(y)} \quad \text{ for } x,y\in \Rd, 
		\end{equation*}
		with absolute convergence of the sum in the norm of $M^1_{v\otimes v}(\Rdd)$ by \eqref{eq:rankoneprojtensor}.
		 Since the function $\phi_n^{(1)}(x)\overline{\phi_n^{(2)}(y)}$ is the kernel  of the rank-one operator $\phi_n^{(1)}\otimes\phi_n^{(2)}$ by \eqref{eq:kernelrankone}, the decomposition of $\kernel_S$ above and the definition of $\|\cdot \|_{\beauty_{v\otimes v}}$ implies that 
		\begin{equation*}
		  S=\sum_{n\in \N} \phi_n^{(1)}\otimes \phi_n^{(2)},
		\end{equation*} 
		with absolute convergence in the norm of $\beauty_{v\otimes v}$. 
	\item Follows from \eqref{eq:projtensor2} and $\|S\|_{\beauty_{v\otimes v}}=\|\kernel_S\|_{M^1_{v\otimes v}}$.	
	\item It is well-known that the kernel of $S^*$ is $\kernel_{S^*}(x,y)=\overline{\kernel_S(y,x)}$. Since $M^1_{v\otimes v}(\Rdd)$ is closed under this operation -- as follows from \eqref{eq:projtensor1}, for instance -- we get $S^*\in \beauty_{v\otimes v}$. In particular, part (a) applied to $S^*$ implies that $S^*$ is bounded  $M^1_v(\Rd)\to M^1_v(\Rd)$. We may therefore define an extension $\tilde{S}:M^\infty_{1/v}(\Rd)\to M^\infty_{1/v}(\Rd)$ by defining $\tilde{S}$ to be the Banach space adjoint of $S^*$. By definition, this means that
	\begin{equation*} 
  \inner{\tilde{S}\phi}{\psi}_{M^\infty_{1/v},M^1_v}=\inner{\phi}{S^*\psi}_{M^\infty_{1/v},M^1_v}.
\end{equation*}
It is easy to see that $\tilde{S}$ is an extension of $S$: if $\phi\in L^2(\Rd)$, we find that
\begin{align*}
   \inner{\tilde{S}\phi}{\psi}_{M^\infty_{1/v},M^1_v}&= \inner{\phi}{S^*\psi}_{M^\infty_{1/v},M^1_v}  \\
  &=  \inner{\phi}{S^*\psi}_{L^2} \\
  &= \inner{S\phi}{\psi}_{L^2} \\
  &= \inner{S\phi}{\psi}_{M^\infty_{1/v},M^1_v}.
\end{align*}
From now on, we simply denote the extension $\tilde{S}$ by $S$. For the last part, note that $S^*$ has a decomposition
 $
  S^*=\sum_{n\in \N} \phi_n^{(2)}\otimes \phi_n^{(1)}
$
by part (a). By definition, for $\psi\in M^\infty_{1/v}(\Rd)$, we have
\begin{equation*}
  \inner{S\psi}{\phi}_{M^\infty_{1/v},M^1_v}= \inner{\psi}{S^*\phi}_{M^\infty_{1/v},M^1_v}.
\end{equation*}
By the decomposition above, $S^*\phi=\sum_{n=1}^\infty \inner{\phi}{\phi_n^{(1)}}_{L^2} \phi_n^{(2)}$, and as this sum converges absolutely in the norm of $M^1_v(\Rd)$ we find
\begin{align*}
  \inner{S\psi}{\phi}_{M^\infty_{1/v},M^1_v}&= \inner{\psi}{S^*\phi}_{M^\infty_{1/v},M^1_v} \\
  &= \inner{\psi}{\sum_{n=1}^\infty \inner{\phi}{\phi_n^{(1)}}_{L^2} \phi_n^{(2)}}_{M^\infty_{1/v},M^1_v} \\
  &= \sum_{n=1}^\infty \inner{\phi_n^{(1)}}{\phi}_{L^2} \inner{\psi}{\phi_n^{(2)}}_{M^\infty_{1/v},M^1_v} \\
  &= \inner{ \sum_{n=1}^\infty \inner{\psi}{\phi_n^{(2)}}_{M^\infty_{1/v},M^1_v}\phi_n^{(1)} }{\phi}_{M^\infty_{1/v},M^1_v}.
\end{align*}
The absolute convergence in the norm of $M^1_v(\Rdd)$ follows as 
\begin{equation*}
\sum_{n=1}^\infty \left|\inner{\psi}{\phi_n^{(2)}}_{M^\infty_{1/v},M^1_v}\right| \|\phi_n^{(1)}\|_{M^1_v} \\
  \leq \|\psi\|_{M^{\infty}_{1/v}} \sum_{n=1}^\infty  \|\phi_n^{(1)}\|_{M^1_v} \|\phi_n^{(2)}\|_{M^1_v}<\infty.
\end{equation*}
	\item The last inequality above also implies that $S$ is bounded from $M^\infty_{1/v}(\Rd)$ to $M^1_v(\Rd)$, since it shows that 
\begin{equation*}
  \|S\psi\|_{M^1_v} \leq \|\psi\|_{M^{\infty}_{1/v}} \sum_{n=1}^\infty  \|\phi_n^{(1)}\|_{M^1_v} \|\phi_n^{(2)}\|_{M^1_v}.
\end{equation*}
	Finally, let $\{\psi_i\}_{i\in \N}$ be a sequence in $M^\infty_{1/v}(\Rd)$ that converges to $\psi\in M^\infty_{1/v}(\Rd)$ in the weak* topology. Then
		\begin{equation*}
  			S\psi_i=\sum_{n\in \N} \inner{\psi_i}{\phi_n^{(2)}}_{M^\infty_{1/v},M^1_{v}} \phi_n^{(1)}\xrightarrow{i\to \infty} \sum_{n\in \N} \inner{\psi}{\phi_n^{(2)}}_{M^\infty_{1/v},M^1_{v}} \phi_n^{(1)}=S\psi.
		\end{equation*}
		We have used the dominated convergence theorem for Banach spaces \cite[Prop. 1.2.5]{Hytonen:2016} to take the limit inside the sum: as $\{\psi_i\}_{i\in \N}$ is weak*-convergent there exists $0<C<\infty$ such that $\|\psi_i\|_{M^\infty_{1/v}}\leq C$ for any $i$, so $$\left\|\inner{\psi_i}{\phi_n^{(2)}}_{M^\infty_{1/v},M^1_{v}}\phi_n^{(1)}\right\|_{M^1_v}\leq C\|\phi_n^{(2)}\|_{M^1_{v}}\|\phi_n^{(1)}\|_{M^1_{v}}$$
		for any $i$, and $\sum_{n\in \N}\|\phi_n^{(1)}\|_{M^1_{v}}\|\phi_n^{(2)}\|_{M^1_{v}}<\infty$. 

\end{enumerate}
\end{proof}
As a first consequence, we show that $\beauty_{v\otimes v}$ is closed under composition. The proof is similar to that of \cite[Cor. 3.11]{Feichtinger:2018}, where the result is proved for locally compact abelian groups with no weights.
\begin{corollary} \label{cor:composition}
	 $\beauty_{v\otimes v}$ is closed under composition: if $S,T\in \beauty_{v\otimes v}$, then $$\|ST\|_{\beauty_{v\otimes v}}\lesssim \|S\|_{\beauty_{v\otimes v}}\|T\|_{\beauty_{v\otimes v}}.$$
\end{corollary}
\begin{proof}
	Let 
	\begin{align*}
  &S=\sum_{n\in \N} \phi_n^{(1)}\otimes \phi_n^{(2)}, &&T=\sum_{m\in \N} \psi_m^{(1)}\otimes \psi_m^{(2)}
\end{align*}
be decompositions of $S$ and $T$ into rank-one operators as in Proposition \ref{prop:innerkernel}. A simple calculation shows that the composition $ST$ is the operator
\begin{equation*}
  ST=\sum_{m,n\in \N} \inner{\psi_m^{(1)}}{\phi_n^{(2)}}_{L^2}  \phi_n^{(1)}\otimes \psi_m^{(2)}.
  \end{equation*}
This decomposition converges absolutely in $\beauty_{v\otimes v}$, as 
\begin{align*}
  \left\| \inner{\psi_m^{(1)}}{\phi_n^{(2)}}_{L^2}  \phi_n^{(1)}\otimes \psi_m^{(2)}\right\|_{\beauty_{v\otimes v}}&\leq  \left|\inner{\psi_m^{(1)}}{\phi_n^{(2)}}_{L^2} \right| \left\|\phi_n^{(1)}\otimes \psi_m^{(2)}\right\|_{\beauty_{v\otimes v}} \\
  &\leq \|\psi_m^{(1)}\|_{L^2} \|\phi_n^{(2)}\|_{L^2} \|\phi_n^{(1)}\|_{M^1_{v}}\|\psi_m^{(2)}\|_{M^1_{v}},  
\end{align*}
so that  
\begin{equation*}
  \sum_{m,n\in \N} \left\| \inner{\psi_m^{(1)}}{\phi_n^{(2)}}_{L^2}  \phi_n^{(1)}\otimes \psi_m^{(2)}\right\|_{\beauty_{v\otimes v}}
\end{equation*}
is bounded from above by 
\begin{equation*}
 \sum_{m\in \N}  \|\psi_m^{(1)}\|_{L^2} \|\psi_m^{(2)}\|_{M^1_{v}} \sum_{n\in \N} \|\phi_n^{(2)}\|_{L^2} \|\phi_n^{(1)}\|_{M^1_{v}}<\infty.
\end{equation*}
We have used the continuous inclusion (see Proposition \ref{prop:modulationspaces}) $M^1_v(\Rd)\hookrightarrow M^2(\Rd)=L^2(\Rd)$ to obtain $\|\psi_m^{(1)}\|_{L^2}\lesssim \|\psi_m^{(1)}\|_{M^1_v}$ and $\|\phi_n^{(1)}\|_{L^2}\lesssim \|\phi_n^{(1)}\|_{M^1_v}$. The inequality $\|ST\|_{\beauty_{v\otimes v}}\lesssim \|S\|_{\beauty_{v\otimes v}}\|T\|_{\beauty_{v\otimes v}}$ follows from part (b) of Proposition \ref{prop:innerkernel}.
\end{proof}
\begin{remark}
	In \cite[Thm. 7.4.1]{Feichtinger:1998} it is claimed that $\beauty_{v\otimes v}$ for $v\equiv 1$ is even an ideal in $\bo$. This is not true. Consider $S=\psi\otimes \phi_0$ and $T=\phi_0\otimes \phi_0$ where $\psi\in L^2(\Rd)\setminus M^1(\Rd)$. Then $T\in \beauty_{1\otimes 1}$ and $S\in \tco$, and $ST=\psi \otimes \phi_0.$ Yet $ST(\phi_0)=\psi\notin M^1(\Rd)$, so part (d) of Proposition \ref{prop:innerkernel} implies that $ST\notin \beauty_{1\otimes 1}$.
\end{remark}

We next study a continuity property of the Fourier-Wigner transform on $\beauty_{v\otimes v}$.

\begin{proposition} \label{prop:fwbounded}
	The Fourier-Wigner transform is bounded from $\beauty_{v\otimes v}$ to $W(L^1_v)$:
	\begin{equation*}
  \|\F_W(S)\|_{W(L^1_v)}\lesssim \|S\|_{\beauty_{v\otimes v}}.
\end{equation*}
\end{proposition}

\begin{proof}
First consider the rank-one operator $\psi\otimes \phi\in \beauty_{v\otimes v}$, with $\psi,\phi \in M^1_v(\Rd)$. By \eqref{eq:fwrankone} and the proof of \cite[Prop. 12.1.11]{Grochenig:2001}, there exists $C>0$ such that 
\begin{equation*}
  \|\F_W(\psi\otimes \phi)\|_{W(L^1_v)}\leq C \|\psi\|_{M^1_v}\|\phi\|_{M^1_v}.
\end{equation*}
 If we then use Proposition \ref{prop:innerkernel} to write $S\in \beauty_{v\otimes v}$ as
$
  S=\sum_{n\in \N} \phi_n^{(1)}\otimes \phi_n^{(2)},
$ we find  
\begin{equation*}
  \|\F_W(S)\|_{W(L^1_v)} \leq \sum_{n\in \N} \|\F_W(\phi_n^{(1)}\otimes \phi_n^{(2)})\|_{W(L^1_v)}\leq C \sum_{n\in \N} \|\phi_n^{(1)}\|_{M^1_v}\|\phi_n^{(2)}\|_{M^1_v}.
\end{equation*}
By part (b) of Proposition \ref{prop:innerkernel} this implies that $ \|\F_W(S)\|_{W(L^1_v)}\leq C \|S\|_{\beauty_{v\otimes v}}$.     

\end{proof}
\begin{remark}
	If we consider the polynomial weights $v_s(z)=(1+|z|)^s$ for $s\geq 0$ and $z\in \Rdd$, it is known \cite[Prop. 11.3.1]{Grochenig:2001} that the space of Schwartz functions $\mathcal{S}(\Rdd)$ is given by $\mathcal{S}(\Rdd)=\cap_{s=0}^\infty M^1_{v_s\otimes v_s}(\Rdd)$. Therefore the space of operators with kernel in $\mathcal{S}(\Rdd)$ equals $\cap_{s\geq 0}^\infty \beauty_{v_s\otimes v_s}$. Such operators were recently studied in \cite{Keyl:2016}.
\end{remark}

\subsection{The space $\beauty$ and its dual} The largest of the spaces $\beauty_{v\otimes v}$ is the space $\beauty:=\beauty_{1\otimes 1},$ consisting of operators $S$ with kernel $\kernel_S$ in $M^1(\Rdd)$. By definition the map $\kappa:\beauty \to M^1(\Rdd)$ given by $\kappa(S)=k_S$ is an isometric isomorphism of Banach spaces. By \cite[Thm. 3.1.18]{Megginson:1998} the Banach space adjoint $(\kappa^{-1})^*: \beast\to M^\infty(\Rdd)$ is a weak*-to-weak*-continuous isometric isomorphism, and by definition it satisfies 
\begin{equation} \label{eq:dualityaction1}
  \inner{(\kappa^{-1})^*(\tilde{A})}{\kernel_S}_{M^\infty,M^1}=\inner{\tilde{A}}{S}_{\beast,\beauty} \quad \text{ for } \tilde{A}\in \beast,S\in \beauty.
\end{equation}

Hence, to any $\tilde{A}\in \beast$ we obtain a unique element $(\kappa^{-1})^*(\tilde{A})\in M^\infty(\Rdd)$, which by the kernel theorem for modulation spaces induces an operator $A:M^1(\Rd)\to M^\infty(\Rd)$ such that $\kernel_{A}=(\kappa^{-1})^*(\tilde{A})$. We summarize these identifications in a simple diagram, where k.t. refers to the kernel theorem for modulation spaces:
\begin{equation} \label{eq:correspondance}
  \tilde{A}\in \beast\xleftrightarrow{ (\kappa^{-1})^*} (\kappa^{-1})^*(\tilde{A})=\kernel_{A}\in M^\infty(\Rdd) \xleftrightarrow{\text{k.t.}} A\in \mathcal{L}(M^1, M^\infty).
\end{equation} 
Hereafter we will always identify $\beast$ with operators $A:M^1(\Rd)\to M^\infty(\Rd)$, and use the notation $A$ to refer to both the operator $A:M^1(\Rd)\to M^\infty(\Rd)$ and the abstract functional $\tilde{A}$, which are related by \eqref{eq:correspondance}. Since $(\kappa^{-1})^*(\tilde{A})=\kernel_{A}$, \eqref{eq:dualityaction1} becomes
\begin{equation} \label{eq:dualityaction2}
   \inner{A}{S}_{\beast,\beauty}= \inner{\kernel_A}{\kernel_S}_{M^\infty,M^1} \quad \text{ for } A\in \beast,S\in \beauty.
\end{equation}
If $S$ is a rank-one operator $S=\phi \otimes \psi$ for $\phi,\psi\in M^1(\Rd)$, then $\kernel_{S}=\phi\otimes \overline{\psi}$, so the equation above becomes
\begin{equation} \label{eq:dualityaction3}
  \inner{A}{\phi\otimes \psi}_{\beast,\beauty}=\inner{\kernel_A}{\phi\otimes \overline{\psi}}_{M^\infty,M^1} = \inner{A\psi}{\phi}_{M^\infty,M^1},
\end{equation}
which relates the action of $A$ as an abstract linear functional on $\beauty$ to the action of $A$ as an operator from $M^1(\Rd)$ to $M^\infty(\Rd)$. 

\begin{lemma} \label{lem:denseintraceclass}
$\beauty$ is a dense subset of $\tco$ with respect to $\|\cdot\|_{\tco}$.	
\end{lemma}
\begin{proof}
	The rank-one operators span a dense subset of $\tco$ \cite[Thm. 3.11 (e)]{Busch:2016}, hence it suffices to show that any $\psi\otimes \phi\in \tco$ with $\psi,\phi\in L^2(\Rd)$ can be estimated by some $S\in \beauty$. We may safely assume that $\phi\neq 0$, otherwise $\psi\otimes \phi=0\in \beauty$. Let $\epsilon >0$. Since $M^1(\Rd)$ is a dense subset of $L^2(\Rd)$ by \cite[Lem. 4.19]{Jakobsen:2018}, we can find $\xi\neq 0,\eta\in M^1(\Rd)$ with $\|\psi-\xi\|_{L^2}<\frac{\epsilon}{2\|\phi\|_{L^2}}$ and $\|\phi-\eta\|_{L^2}<\frac{\epsilon}{2\|\xi\|_{L^2}}$. Then $\xi\otimes \phi \in \beauty$ and
	\begin{align*}
  \|\psi\otimes \phi-\xi\otimes \eta\|_{\tco}&\leq \|\psi\otimes \phi-\xi\otimes \phi\|_{\tco}+ \|\xi \otimes \phi - \xi \otimes \eta\|_{\tco} \\
  &= \|\psi-\xi\|_{L^2}\|\phi\|_{L^2}+\|\xi\|_{L^2}\|\phi-\eta\|_{L^2}<\epsilon.
\end{align*}
\end{proof}

Now recall that $\bo$ is the dual space of $\tco$ \cite[Thm. 3.13]{Busch:2016}, where $A\in \bo$ acts on $S\in \tco$ by 

\begin{equation} \label{eq:traceclassduality}
  \inner{A}{S}_{\bo,\tco}=\tr(AS^*). 
\end{equation}

 Since the inclusion $\beauty \hookrightarrow \tco$ has dense range, \cite[Thm. 3.1.17]{Megginson:1998} asserts that we get a weak*-to-weak*-continuous inclusion of dual spaces $\bo \hookrightarrow \beast$ satisfying
\begin{equation} \label{eq:dualityaction4}
  \inner{A}{S}_{\beast,\beauty} = \inner{A}{S}_{\bo,\tco}=\tr(AS^*) \quad  \text{ for } A\in \bo,S\in \beauty.
\end{equation}
\begin{remark}
	Readers with little interest in these technical details need only note that we identify $\beast$ with operators $A\in\mathcal{L}(M^1(\Rd),M^\infty(\Rd))$, and that the action of $A$ satisfies \eqref{eq:dualityaction2}, \eqref{eq:dualityaction3} and \eqref{eq:dualityaction4}. 
\end{remark}

 The next result is due to Feichtinger and Kozek \cite{Feichtinger:1998}; in their terminology the result says that $\F_W$ and the Weyl transform are \textit{Gelfand triple isomorphisms.} Recall that $\HS$ are the Hilbert-Schmidt operators on $L^2(\Rd).$
\begin{proposition} \label{prop:weak*}
	The Weyl transform $S\longleftrightarrow a_s$ and Fourier-Wigner transform $S\longleftrightarrow \F_W(S)$ are isomorphisms $\beauty \longleftrightarrow M^1(\Rdd)$, unitary maps $\HS \longleftrightarrow L^2(\Rdd)$ and weak*-to-weak*-continuous isomorphisms $\beast \longleftrightarrow M^\infty(\Rdd)$.
\end{proposition}
An appropriate framework for such statements is the theory of (Banach) Gelfand triples \cite{Feichtinger:1998,Feichtinger:2009,Cordero:2008}. In particular, that approach gives the duality bracket identity
	\begin{equation} \label{eq:gelftripoperators}
  \inner{S}{T}_{\beast,\beauty}= \inner{\weyl_S}{\weyl_T}_{M^\infty,M^1},
\end{equation}
where $\weyl_S$ and $\weyl_T$ are the Weyl symbols of $S$ and $T$, see \cite[Cor. 5]{Cordero:2008}. 	

\begin{remark}
We will often consider weak*-convergence of sequences in $\beast$. To get a better grasp of this notion of convergence, note that if a sequence $\{A_n\}_{n\in \N}\subset \beast$ converges in the weak* topology to $A\in \beast$ then \eqref{eq:dualityaction3} gives for $\psi,\phi\in M^1(\Rd)$
\begin{equation*}
  \inner{A_n \psi}{\phi}_{M^\infty,M^1}\to \inner{A\psi}{\phi}_{M^\infty,M^1}.
\end{equation*}
Hence: if $A_n\to A$ in the weak* topology of $\beast$, then $A_n\psi\to A\psi$ in the weak* topology of $M^\infty(\Rd)$ for any $\psi\in M^1(\Rd)$.
\end{remark}

\section{Gabor g-frames} \label{sec:ggf}
Gabor frames, or more generally multi-window Gabor frames, have a richer structure than general frames. Since any frame is also a g-frame, we can ask whether Gabor frames belong to a certain class of g-frames, and whether this class contains other g-frames that share the rich structure of Gabor frames. This is the motivation for the following definition.

\begin{definition}
	Let $\Lambda$ be a lattice in $\Rdd$ and $S\in \bo$. We say that $S$ generates a \textit{Gabor g-frame} with respect to $\Lambda$ if $\{\alpha_\lambda (S)\}_{\lambda \in \Lambda}$ is a g-frame for $L^2(\Rd)$, i.e. if there exist positive constants $A,B>0$ such that 
\begin{equation} \label{eq:gaborgframes}
	A\|\psi\|_{L^2}^2 \leq \sum_{\lambda \in \Lambda} \|\alpha_{\lambda}(S)\psi\|_{L^2}^2 \leq B \|\psi\|_{L^2}^2 \text{\ \ \ \  for any } \psi \in L^2(\Rd).
\end{equation}
\end{definition} 
\begin{remark}[Cohen's class] \label{rem:cohen}

This definition may also be rephrased in terms of Cohen's class of time-frequency distributions\cite{Cohen:1966}.  In the notation from \cite{Luef:2018b} an operator $T\in \bo$ defines a Cohen's class distribution $Q_T$ by $$Q_T(\psi)(z)=\inner{T\pi(z)^*\psi}{\pi(z)^*\psi}_{L^2} \quad \text{ for } z\in \Rdd, \psi \in L^2(\Rd).$$ 
It is straightforward to show that 
	 $$\|\alpha_{z}(S)\psi\|_{L^2}^2= Q_{S^*S}(\psi)(z),$$ hence 
	 \eqref{eq:gaborgframes} may be rephrased as
	\begin{equation*} 
	A\|\psi\|_{L^2}^2 \leq \sum_{\lambda \in \Lambda} Q_{S^*S}(\psi)(\lambda) \leq B \|\psi\|_{L^2}^2 \text{\ \ \ \  for any } \psi \in L^2(\Rd).
\end{equation*}
We will soon see that \eqref{eq:gaborgframes} forces $S$ to be a Hilbert Schmidt operator, which implies by \cite[Thm. 7.6]{Luef:2018b} that $Q_{S^*S}$ is a positive Cohen's class distribution satisfying 
\begin{equation} \label{eq:continuouscohen}
  \int_{\Rdd} Q_{S^*S}(\psi)(z) \ dz = \int_{\Rdd} \|\alpha_{z}(S)\psi\|_{L^2}^2 \ dz = \|S\|_{\HS}^2 \|\psi\|_{L^2}^2,
\end{equation}
 as recently studied in \cite{Luef:2018b,Luef:2018a}. This equality is a continuous version of \eqref{eq:gaborgframes}, similar to how Moyal's identity is a continuous version \footnote{In fact, Moyal's identity says that the system $\{\pi(z)\varphi\}_{z\in \Rdd}$ is a \textit{tight continuous frame} for $L^2(\Rd)$ for any $0\neq \varphi \in L^2(\Rd)$.  See \cite{Christensen:2016} for continuous frames. Similarly, \eqref{eq:continuouscohen} says that $\{\alpha_z(S)\}_{z\in \Rdd}$ is a \textit{tight continuous g-frame} as introduced in \cite{Abdollahpour:2008}.} of the Gabor frame inequalities \eqref{eq:intro:gaborframeop}.  The simplest example of a Cohen's class distribution of the form $Q_{S^*S}$ is the spectrogram $Q_{S^*S}(z)=|V_{\phi}\psi(z)|^2$ for some $\phi \in L^2(\Rd)$, which corresponds to the rank-one operator $S=\frac{1}{\|\phi\|_{L^2}^2}\phi\otimes \phi$. By inserting $\|\alpha_z(S)\psi\|_{L^2}^2=Q_{S^*S}(\psi)(z)=|V_\phi \psi(z)|^2$, \eqref{eq:gaborgframes} becomes the condition for $\phi$ to generate a Gabor frame. We return to this special case in Example \ref{example:multiwindow}. 
\end{remark}

\subsection{The Gabor g-frame operator}
By the general theory of g-frames, the g-frame operator associated to a Gabor g-frame generated by $S$ over a lattice $\Lambda$ is the operator
\begin{equation}\label{eq:gframeop}
  \frameop_S=\sum_{\lambda\in \Lambda} (\alpha_\lambda (S))^* (\alpha_\lambda (S))=\sum_{\lambda\in \Lambda}\alpha_{\lambda}(S^*S),
\end{equation}
where the last equality uses \eqref{eq:translateproduct}. Furthermore, $\frameop_S$ satisfies
\begin{equation*}
  \inner{\frameop_S \psi}{\psi}_{L^2} =\sum_{\lambda \in \Lambda} \|\alpha_{\lambda}(S)\psi\|_{L^2}^2\quad  \text { for } \psi \in L^2(\Rd),
\end{equation*}
and $\frameop_S$ is positive, bounded and invertible on $L^2(\Rd)$ with $A\leq \|\frameop_S\|_{\bo}\leq B$ and $\frac{1}{B}\leq \|\frameop_S^{-1}\|\leq \frac{1}{A}$. Since we think of $\alpha_\lambda (S^*S)$ as the translation of $S^*S$ by $\lambda\in \Lambda$, the g-frame operator $\frameop_S$ corresponds to the \textit{periodization} of $S^*S$ over $\Lambda$.\\

\subsection{Analysis and synthesis operators}
Let $\ell^2(\Lambda;L^2(\Rd))$ be the Hilbert space of sequences $\{\psi_\lambda\}_{\lambda \in \Lambda}\subset L^2(\Rd)$ such that 
\begin{equation*}
  \|\{\psi_\lambda\}\|_{\ell^2(\Lambda;L^2)}:=\left(\sum_{\lambda \in \Lambda} \|\psi_\lambda\|_{L^2}^2\right)^{1/2}<\infty,
\end{equation*}
with inner product
\begin{equation*}
  \inner{\{\psi_\lambda\}}{\{\phi_\lambda\}}_{\ell^2(\Lambda;L^2)}=\sum_{\lambda \in \Lambda} \inner{\psi_\lambda}{\phi_\lambda}_{L^2} .
\end{equation*}
For $S\in \bo$ we define the \textit{analysis operator} $C_S$ by
\begin{equation*}
  C_S(\psi)=\left\{ \alpha_{\lambda}(S)\psi \right\}_{\lambda \in \Lambda} \quad \text{ for } \psi \in L^2(\Rd)
\end{equation*}
and the \textit{synthesis operator} $D_S$ by
\begin{equation*}
  D_S(\{\psi_\lambda\}):= \sum_{\lambda \in \Lambda} \alpha_\lambda(S^*)\psi_\lambda \quad \text{ for } \{\psi_\lambda\}_{\lambda\in \Lambda}\in \ell^2(\Lambda;L^2).
\end{equation*}
The upper bound in \eqref{eq:gaborgframes} is precisely the statement that $C_S:L^2(\Rd)\to \ell^2(\Lambda;L^2)$ is a bounded operator with operator norm $\leq \sqrt{B}$. It is not difficult to show that $D_S$ is the Hilbert space adjoint of $C_S$, which implies that $D_S$ is bounded whenever $C_S$ is, with the same operator norm as $C_S$. It follows from the definitions that 
\begin{equation*}
  \frameop_{S}=D_SC_S.
\end{equation*}
\subsubsection{Dual g-frames} \label{sec:dual}
If $S$ generates a Gabor g-frame over $\Lambda$, then the theory of g-frames \cite{Sun:2006} says that the \textit{canonical dual g-frame} is 
\begin{equation*} 
  \{\alpha_\lambda (S)\frameop_{S}^{-1}\}_{\lambda \in \Lambda}.
\end{equation*}
It is clear from \eqref{eq:gframeop} that $\alpha_\lambda (\frameop_S)=\frameop_S$ for any $\lambda \in \Lambda$, and it is then easy to check that we also have $\alpha_\lambda (\frameop_S^{-1})=\frameop_S^{-1}$. The canonical dual g-frame is therefore 
\begin{equation*}
  \{\alpha_\lambda (S)\frameop_{S}^{-1}\}_{\lambda \in \Lambda}=\{\alpha_\lambda (S)\alpha_\lambda(\frameop_{S}^{-1})\}_{\lambda \in \Lambda}=\{\alpha_\lambda (S\frameop_{S}^{-1})\}_{\lambda \in \Lambda}.
\end{equation*}
Hence the canonical dual g-frame is also a Gabor g-frame, generated by $S\frameop_{S}^{-1}.$
We get the reconstruction formulas
\begin{align*} 
  \psi &= \frameop_{S}\frameop_S^{-1}\psi=\sum_{\lambda \in \Lambda} \alpha_\lambda (S^*S)\frameop_S^{-1}\psi=\sum_{\lambda \in \Lambda} \alpha_\lambda (S^*) \alpha_\lambda(S\frameop_S^{-1})\psi=D_{S}C_{S\frameop_S^{-1}}\psi, \\
  \psi&=\frameop_{S}^{-1}\frameop_{S}\psi=\frameop_S^{-1} \sum_{\lambda\in \Lambda} \alpha_\lambda(S^*S) \psi =\sum_{\lambda\in \Lambda} \alpha_\lambda(\frameop_S^{-1}S^*)\alpha_\lambda(S)\psi=D_{S\frameop_S^{-1}}C_S\psi.
\end{align*}
In the very last of these equalities we have used that $\frameop_S^{-1}$ is a positive (hence self-adjoint) operator, so $(S\frameop_S^{-1})^*=\frameop_S^{-1}S^*$. Inspired by these formulas and the theory of dual windows for Gabor frames \cite{Grochenig:2001,Feichtinger:1998a}, we say that two operators $S,T\in \bo$ generate \textit{dual Gabor g-frames} if $S$ and $T$ generate Gabor g-frames and $D_SC_T$ is the identity operator on $L^2(\Rd)$, i.e.
 \begin{equation} \label{eq:dualgaborgframes}
  \sum_{\lambda \in \Lambda} \alpha_\lambda (S^*)\alpha_\lambda(T)\psi=\sum_{\lambda \in \Lambda} \alpha_\lambda (S^*T)\psi=\psi \quad \text{ for any } \psi \in L^2(\Rd).
\end{equation}
If $D_S$ and $C_T$ are bounded operators (i.e. $S$ and $T$ satisfy the upper g-frame bound in \eqref{eq:gaborgframes}), then \eqref{eq:dualgaborgframes} implies that both $S$ and $T$ generate Gabor g-frames. This follows from the general theory of g-frames, see \cite[p. 441]{Sun:2006}: the lower bound in \eqref{eq:gaborgframes} for $T$ follows from $$\|\psi\|_{L^2}^2=\|D_SC_T\psi\|^2_{L^2}\lesssim \|C_T\psi\|_{\ell^2(\Lambda;L^2)}^2=\sum_{\lambda \in \Lambda} \|\alpha_\lambda (T)\psi\|_{L^2}^2,$$ and the lower bound for $S$ is similar. We state this as a proposition for later reference.
\begin{proposition} \label{prop:dualframes}
	Assume that $S,T\in \bo$ satisfy \eqref{eq:dualgaborgframes} and the upper bound in  \eqref{eq:gaborgframes}. Then $S$ and $T$ generate Gabor g-frames.  
\end{proposition}

\subsection{Two examples}
 We will now show that the Gabor g-frames include multi-window Gabor frames as a special case. 
 \begin{example}[Multi-window Gabor frames] \label{example:multiwindow}
 Consider a set of $N<\infty$ functions $\{\phi_n\}_{n=1}^N\subset L^2(\Rd)$. We seek an operator $S$ such that the multi-window Gabor system generated by $\{\phi_n\}_{n=1}^N$ is captured by the system $\{\alpha_\lambda(S)\}_{\lambda \in \Lambda}$. To achieve this, let $\{\xi_n\}_{n=1}^N$ be any orthonormal set in $L^2(\Rd)$, and consider the operator 
 \begin{equation*}
  S=\sum_{n=1}^N \xi_n \otimes \phi_n.
\end{equation*}
We start by writing out the condition \eqref{eq:gaborgframes} for $S$ to generate a Gabor g-frame. For $\psi \in L^2(\Rd)$, we easily find using \eqref{eq:shiftrankone} that 
\begin{equation*}
  \alpha_\lambda(S)\psi=\sum_{n=1}^N V_{\phi_n}\psi(\lambda) \pi(\lambda)\xi_n.
\end{equation*}
 By the orthonormality of $\{\xi_n\}_{n=1}^N$ and Pythagoras' theorem for inner product spaces, this implies that $\|\alpha_\lambda(S)\psi\|_{L^2}^2=\sum_{n=1}^N|V_{\phi_n}\psi(\lambda)|^2$. Inserting this into \eqref{eq:gaborgframes}, we see that $S$ generates a Gabor g-frame if and only if
\begin{equation*}
  A\|\psi\|^2_{L^2} \leq \sum_{\lambda \in \Lambda} \sum_{n=1}^N |V_{\phi_n}\psi(\lambda)|^2 \leq B \|\psi\|_{L^2}^2 \quad \psi \in L^2(\Rd)
\end{equation*}
for some $A,B>0$, which is precisely the condition that $\{\phi_n\}_{n=1}^N$ generate a multi-window Gabor frame.

We then note that $S^*=\sum_{n=1}^N \phi_n \otimes \xi_n$, and $S^*S=\sum_{n=1}^N \phi_n \otimes \phi_n$ by the orthonormality of $\{\xi_n\}_{n=1}^N$. Denote by $C_{MW}$ and $\frameop_{MW}$ the analysis and frame operator associated with the multi-window Gabor system generated by $\{\phi_n\}_{n=1}^N$. For $\psi\in L^2(\Rd)$, we find that 
\begin{align*}
  &C_S(\psi)= \left\{\sum_{n=1}^N V_{\phi_n} \psi(\lambda) \pi(\lambda)\xi_n\right\}_{\lambda \in \Lambda}, && C_{MW}(\psi)=\{V_{\phi_n}\psi(\lambda)\}_{n\in \Z_n,\lambda \in \Lambda}, \\
   &\frameop_S(\psi)= \sum_{\lambda \in \Lambda} \sum_{n=1}^N V_{\phi_n} \psi(\lambda) \pi(\lambda)\phi_n, && \frameop_{MW}(\psi)= \sum_{\lambda \in \Lambda} \sum_{n=1}^N V_{\phi_n} \psi(\lambda) \pi(\lambda)\phi_n.
\end{align*}
We see that the frame operators $\frameop_S$ and $\frameop_{MW}$ are equal. Since $\{\pi(\lambda)\xi_n\}_{n\in \N}$ is orthonormal for each $\lambda \in \Lambda$, we also see that $C_{MW}(\psi)$ and $C_{S}(\psi)$ carry exactly the same information: if we know $C_S(\psi)$, i.e. we know $$\sum_{n=1}^N V_{\phi_n} \psi (\lambda) \pi(\lambda)\xi_n$$ for each $\lambda \in \Lambda$, we can find $C_{MW}(\psi)$ by $$V_{\phi_m}\psi(\lambda)=\inner{\sum_{n=1}^N V_{\phi_n} \psi (\lambda) \pi(\lambda)\xi_n}{\pi(\lambda)\xi_m}_{L^2}.$$ Hence multi-window Gabor frames are Gabor g-frames.
 \end{example}
A less trivial example was considered in \cite{Dorfler:2006,Dorfler:2011}. Section \ref{sec:eqnorms} will be dedicated to showing that the results from \cite{Dorfler:2011} hold for more general Gabor g-frames, and not just for the following example. The fact that the results of \cite{Dorfler:2006} is an example of g-frames was noted already by Sun \cite{Sun:2006}.
\begin{example}[Localization operators] \label{example:locops}
  Let $0\neq \varphi \in M^1_v(\Rd)$, $\Lambda$ a lattice and $h\in L^1_v(\Rdd)$ a non-negative function. Here $h\in L^1_v(\Rdd)$ means that $\|h\|_{L^1_v}:=\int_{\Rdd} h(z)v(z)\ dz<\infty$. Assume further that $$A'\leq \sum_{\lambda \in \Lambda} h(z-\lambda)\leq B' \quad \text{ for all } z\in \Rdd$$ for some $A',B'>0$. Then the localization operator $A_h^\varphi$ generates a Gabor g-frame over $\Lambda$ \cite{Sun:2006,Dorfler:2011,Dorfler:2006}. The key to connecting the summability condition on $h$ to the Gabor g-frame condition for $A_h^\varphi$ is equation \eqref{eq:translatelocop}. We will return to this example in Section \ref{sec:locops}.
\end{example}
\subsection{A trace class condition}
In the definition of Gabor g-frames, we only assumed that $S$ was a bounded linear operator on $L^2(\Rd)$. We will now show that $S$ must be a Hilbert Schmidt operator. The following lemma is essentially the same as \cite[Lem. 3.1]{Balazs:2008}. 
\begin{lemma} \label{lem:frametrace}
	Let $T\in \bo$ be a positive operator. If $\{\xi_n\}_{n\in \N}$ is an orthonormal basis for $L^2(\Rd)$ and $\{\eta_j\}_{j\in \N}$ is a Parseval frame, then 
	\begin{equation*}
  \tr(T):=\sum_{n\in \N} \inner{T \xi_n}{\xi_n}_{L^2} =\sum_{j\in \N} \inner{T \eta_j}{\eta_j}_{L^2} .
\end{equation*}
\end{lemma}
\begin{proof}
	Using the square root of the positive operator $T$, we have that $$\inner{T\eta_j}{\eta_j}_{L^2} =\inner{T^{1/2}\eta_j}{T^{1/2}\eta_j}_{L^2} = \|T^{1/2}\eta_j\|_{L^2}^2.$$ Hence by Parseval's identity
	\begin{align*}
  \sum_{j\in \N} \inner{T \eta_j}{\eta_j}_{L^2}  &= \sum_{j\in \N} \|T^{1/2}\eta_j\|_{L^2}^2 \\
  &= \sum_{j\in \N} \sum_{n\in \N} \left|\inner{T^{1/2}\eta_j}{\xi_n}_{L^2} \right|^2 \\
  &=  \sum_{n\in \N} \sum_{j\in \N} \left|\inner{\eta_j}{T^{1/2}\xi_n}_{L^2} \right|^2 \\
  &= \sum_{n\in \N} \|T^{1/2}\xi_n\|_{L^2}^2 
  = \sum_{n\in \N} \inner{T\xi_n}{\xi_n}_{L^2} .
\end{align*}
\end{proof}

\begin{proposition}
	Let $\Lambda$ be any lattice and assume that $\{\alpha_\lambda (S)\}_{\lambda \in \Lambda}$ satisfies the upper g-frame bound in \eqref{eq:gaborgframes}. Then $S^*S$ is a trace class operator. Equivalently,  $S$ is a Hilbert Schmidt operator. 
\end{proposition}

\begin{proof}
	The upper g-frame bound implies that $\sum_{\lambda\in \Lambda} \|\alpha_\lambda(S) \psi\|_{L^2}^2<\infty$ for any $\psi \in L^2(\Rd)$. There exist $\{\varphi_n\}_{n=1}^N\subset L^2(\Rd)$ that generate a Parseval multi-window Gabor frame over $\Lambda$ \cite{Luef:2009}, i.e. $\{\pi(\lambda)\varphi_n\}_{n=1,...,N,\lambda \in \Lambda}$ is a Parseval frame. Then Lemma \ref{lem:frametrace} says that
	\begin{align*}
  \tr(S^*S)&=\sum_{n=1}^N\sum_{\lambda \in \Lambda} \inner{S^*S\pi(\lambda)\varphi_n}{\pi(\lambda)\varphi_n}_{L^2}  \\
  &=\sum_{n=1}^N\sum_{\lambda \in \Lambda} \inner{S\pi(\lambda)\varphi_n}{S\pi(\lambda)\varphi_n}_{L^2}  \\
  &= \sum_{n=1}^N\sum_{\lambda \in \Lambda} \|S\pi(\lambda)\varphi_n\|_{L^2}^2. 
\end{align*}
By first using that $\pi(\lambda)^*=e^{-2\pi i \lambda_x \cdot \lambda_\omega}\pi(-\lambda)$ for $\lambda=(\lambda_x,\lambda_\omega)$, and then that $\pi(-\lambda)$ is a unitary operator, we see that
$$\|S\pi(\lambda)\varphi_n\|_{L^2}^2=\|S\pi(-\lambda)^*\varphi_n\|_{L^2}^2=\|\pi(-\lambda)S\pi(-\lambda)^*\varphi_n\|_{L^2}^2.$$
Hence 
\begin{equation*}
  \tr(S^*S)=\sum_{n=1}^N\sum_{\lambda \in \Lambda} \|\alpha_{-\lambda} (S) \varphi_n\|_{L^2}^2=\sum_{n=1}^N\sum_{\lambda \in \Lambda} \|\alpha_{\lambda} (S) \varphi_n\|_{L^2}^2<\infty.
\end{equation*}
so $S^*S$ is a positive trace class operator, and $S$ is a Hilbert Schmidt operator.
\end{proof}

\subsection{Periodization of operators and $\beauty$}
To prepare for the next section on Fourier series of operators, we now consider the periodization of operators. The key to proving these results is \cite[Thm. 8.2]{Luef:2017}, which states that for $S\in \beauty$ and $T\in \tco$, the function $z\mapsto \tr(\alpha_z(S)T)\in M^1(\Rdd)$ with 
\begin{equation} \label{eq:qha}
  \|\tr(\alpha_z(S)T)\|_{M^1}\lesssim \|S\|_\beauty \|T\|_{\tco}
\end{equation}
and similarly for $S\in \tco$ and $T\in \beauty$ 
\begin{equation} \label{eq:qha2}
  \|\tr(\alpha_z(S)T)\|_{M^1}\lesssim \|S\|_\tco \|T\|_{\beauty}.
\end{equation}
\begin{proposition}[Operator periodization]\label{prop:feichtingerwindow}
	 The periodization map given by $S\mapsto \sum_{\lambda \in \Lambda} \alpha_\lambda (S)$ is a well-defined and bounded map $\beauty\to \bo$:
	\begin{equation*}
  \left\|\sum_{\lambda \in \Lambda} \alpha_\lambda (S) \right\|_{\bo}\lesssim \|S\|_{\beauty},
\end{equation*}
and a well-defined and bounded map $\tco\to \beast$:
\begin{equation*}
    \left\|\sum_{\lambda \in \Lambda} \alpha_\lambda (S) \right\|_{\beast}\lesssim \|S\|_{\tco}.
\end{equation*}
The sum $\sum_{\lambda \in \Lambda} \alpha_\lambda (S)$ converges in the weak* topology of $\bo$ when $S\in \beauty$, and in the weak* topology of $\beast$ when $S\in \tco$.
\end{proposition}

\begin{proof}
Let $S\in \beauty$. Since $\bo$ is the dual space of $\tco$ \cite[Thm. 3.13]{Busch:2016}, we define $\sum_{\lambda\in \Lambda} \alpha_\lambda(S)\in \bo$ by duality, by defining its action as an antilinear functional:
\begin{equation*} 
  \inner{\sum_{\lambda\in \Lambda} \alpha_\lambda(S)}{T}_{\bo,\tco}:=\sum_{\lambda\in \Lambda}  \inner{\alpha_\lambda(S)}{T}_{\bo,\tco} \quad \text{ for } T\in \tco.
\end{equation*}
To see that this defines a \textit{bounded} antilinear functional on $\tco$, we estimate that
\begin{align*}
  \left|\sum_{\lambda\in \Lambda}  \inner{\alpha_\lambda(S)}{T}_{\bo,\tco} \right| &\leq \sum_{\lambda\in \Lambda}   \left| \inner{\alpha_\lambda(S)}{T}_{\bo,\tco} \right| \\
  &= \sum_{\lambda\in \Lambda}   \left| \tr(\alpha_\lambda(S)T^*) \right| \quad \text{ by \eqref{eq:traceclassduality}} \\
  &\lesssim  \| \tr(\alpha_z(S)T^*)\|_{M^1} \quad \text{ by Lemma } \ref{lem:wienersampling} \\
  &\lesssim \|S\|_{\beauty} \|T\|_{\tco} \quad \text{ by } \eqref{eq:qha}. 
\end{align*}
It is clear that the partial sums converge to this element $\sum_{\lambda \in \Lambda} \alpha_\lambda (S)$ in the weak* topology of $\bo$: For any finite subset $J\subset \Lambda$ we get 
 \begin{equation*}
  \inner{\sum_{\lambda \in \Lambda} \alpha_\lambda (S)-\sum_{\lambda \in J} \alpha_\lambda (S)}{T}_{\bo,\tco}=\sum_{\lambda \in \Lambda \setminus J}  \inner{\alpha_\lambda (S)}{T}_{\bo,\tco},
\end{equation*}
and we showed above that the sum $\sum_{\lambda \in \Lambda}  \inner{\alpha_\lambda(S)}{T}_{\bo,\tco}$ converges absolutely. 
Then let $S\in \tco$. We define $\sum_{\lambda \in \Lambda}\alpha_\lambda(S)\in \beast$ by duality: 
\begin{equation*}  
  \inner{\sum_{\lambda\in \Lambda} \alpha_\lambda(S)}{T}_{\beast,\beauty}:=\sum_{\lambda\in \Lambda}  \inner{\alpha_\lambda(S)}{T}_{\beast,\beauty} \quad \text{ for } T\in \beauty.
\end{equation*}
The estimate showing that this defines a bounded antilinear functional on $\beauty$ with $\left|\sum_{\lambda\in \Lambda} \inner{\alpha_\lambda(S)}{T}_{\beast,\beauty}\right|\lesssim \|S\|_{\tco} \|T\|_{\beauty}$ is the same as above using \eqref{eq:qha2}, but note that we need to write $\inner{\alpha_\lambda(S)}{T}_{\beast,\beauty}=\tr(\alpha_\lambda(S)T^*)$ to use \eqref{eq:qha2} -- this is true by \eqref{eq:dualityaction4}. 
\end{proof}

\begin{corollary} \label{cor:beautyupperbound}
	If $S^*S\in \beauty$ then $\{\alpha_\lambda(S)\}_{\lambda\in \Lambda}$ satisfies  the upper g-frame bound 
	\begin{equation*}
  \sum_{\lambda \in \Lambda} \|\alpha_\lambda (S) \psi\|_{L^2}^2 \lesssim \|S^*S\|_{\beauty} \|\psi\|_{L^2}^2 \quad \text{ for all } \psi \in L^2(\Rd).
\end{equation*}
In particular, this is true if $S\in \beauty$.
\end{corollary}
\begin{proof}
We observed in the proof above (now with $S^*S$ instead of $S$) that 
\begin{equation} \label{eq:proofupperbound}
	\sum_{\lambda\in \Lambda}   \left| \inner{\alpha_\lambda(S^*S)}{T}_{\bo,\tco} \right|\lesssim \|S^*S\|_{\beauty} \|T\|_{\tco}.
\end{equation}
 If $T=\psi \otimes \psi$, it is simple to show that $\inner{\alpha_\lambda(S^*S)}{T}_{\bo,\tco} =\inner{\alpha_\lambda(S^*S)\psi}{\psi}_{L^2}$ and $\|T\|_{\tco}=\|\psi\|_{L^2}^2$. Therefore equation \eqref{eq:proofupperbound} says that 
 \begin{equation*}
	\sum_{\lambda\in \Lambda}   \left| \inner{\alpha_\lambda(S^*S)\psi}{\psi}_{L^2} \right|\lesssim \|S^*S\|_{\beauty} \|\psi\|_{L^2}^2.
\end{equation*}
 As we have seen, $\inner{\alpha_{\lambda} (S^*S)\psi}{\psi}_{L^2} =\|\alpha_\lambda (S) \psi\|_{L^2}^2$, which completes the proof of the first part. If $S\in \beauty$, it follows from Proposition \ref{prop:innerkernel} and Corollary \ref{cor:composition} that $S^*S\in \beauty$.
\end{proof}
The fact that we only need $S^*S\in \beauty$ is useful in light of our treatment of multi-window Gabor frames in Example \ref{example:multiwindow}. To a system $\{\phi_n\}_{n=1}^N\subset M^1(\Rd)$ we associated the operator $$S=\sum_{n=1}^N \xi_n \otimes \phi_n,$$ where $\{\xi_n\}_{n=1}^N$ is an arbitrary orthonormal set in $L^2(\Rd)$. Hence we do not necessarily have $S\in \beauty$, yet $S^*S=\sum_{n=1}^N \phi_n\otimes \phi_n \in \beauty.$ A version of this corollary for Gabor frames is well-known \cite[Thm. 12.2.3]{Grochenig:2001}. 

\section{Fourier series of operators: the Janssen representation} \label{sec:janssen}
A key insight of Werner's paper \cite{Werner:1984} is that the Fourier-Wigner transform in many respects behaves as a Fourier transform for operators. Given a lattice $\Lambda \subset \Rdd$, this leads to a natural question: if an operator is in some sense $\Lambda$-periodic, can we find a Fourier series expansion of the operator? In fact, $\Lambda$-periodic operators were studied in \cite{Feichtinger:1998}, where an operator $S$ was said to be $\Lambda$-periodic if
\begin{equation*}
  \alpha_\lambda(S)=S \quad \text{ for any } \lambda \in \Lambda.
\end{equation*}
An important tool in \cite{Feichtinger:1998} is the \textit{adjoint lattice} $\Lambda^\circ$ of $\Lambda$, defined by
\begin{align*}
  \Lambda^\circ&=\{\lambda^\circ \in \Rdd : \pi(\lambda^\circ)\pi(\lambda)= \pi(\lambda)\pi(\lambda^\circ) \text{ for any } \lambda \in \Lambda\} \\
  &= \{\lambda^\circ \in \Rdd : e^{2\pi i \sigma(\lambda^\circ,\lambda)}=1 \text{ for any } \lambda \in \Lambda\}, 
\end{align*}
where $\sigma$ is the standard symplectic form.
It is shown in \cite{Feichtinger:1998} that $\Lambda^\circ$ is a lattice, and $|\Lambda^\circ|=\frac{1}{|\Lambda|}.$ One can interpret $\Lambda^\circ$ using abstract harmonic analysis. Identify the dual group $\widehat{\R^{2d}}$ with $\R^{2d}$ by the bijection $\R^{2d} \ni z\mapsto \chi_z\in \widehat{\R^{2d}}$, where $\chi_z$ is the \textit{symplectic} character $\chi_{z}(z')=e^{2\pi i \sigma(z,z')}$. With this identification, we see that
\begin{equation*}
  \Lambda^\circ= \{\lambda^\circ \in \Rdd : \chi_{\lambda^\circ}(\lambda)=1 \text{ for any } \lambda \in \Lambda\}
\end{equation*}
Hence $\Lambda^\circ$ is the annihilator of $\Lambda$, and $\Lambda^\circ$ can therefore be identified with the dual group of $\Rdd/\Lambda$ \cite[Prop. 3.6.1]{Deitmar:2014}. By abstract harmonic analysis, this implies that any well-behaved $\Lambda$-periodic \textit{function} $f$ on $\Rdd$ can be expanded in a \textit{symplectic Fourier series} 
\begin{equation*} 
f(z)=    \sum_{\lambda^\circ \in \Lambda^\circ} c_{\lambda^\circ}e^{2\pi i \sigma(\lambda^\circ,z)}, 
\end{equation*}
and we will refer to $\{c_{\lambda^\circ}\}_{\lambda^\circ \in \Lambda^\circ}$ as the \textit{symplectic Fourier coefficients} of $f$.

\begin{remark} 
	The main results of this section, namely Theorems \ref{theorem:janssen}, \ref{theorem:fourierseriesdual} and \ref{theorem:wexlerraz}, are due to Feichtinger and Kozek \cite{Feichtinger:1998}. The spirit of our approach is also the same as in \cite{Feichtinger:1998} -- we express operators as linear combinations of time-frequency shifts by applying methods from abstract harmonic analysis to their symbol with respect to some pseudodifferential calculus. Since the results form a natural and important part of the theory of Gabor g-frames, we choose to include detailed proofs. Our proofs differ slightly from those in \cite{Feichtinger:1998} by using the Weyl symbol (rather than Kohn-Nirenberg symbol), which makes it particularly transparent that the Janssen representation is a Fourier series of operators (see Lemma \ref{lem:weylsymboltf}). This fits well with our interpretation of $\F_W$ as a Fourier transform. We also extend the results of \cite{Feichtinger:1998} to trace class operators. 

\end{remark}

As our function and operator spaces we will use $M^1(\Rdd)$ and $\beauty$ along with their duals. In the following lemma $A(\Rdd/\Lambda)$ denotes the $\Lambda$-periodic functions $f:\Rdd \to \C$ with symplectic Fourier coefficients $\{c_{\lambda^\circ}\}_{\lambda^\circ \in \Lambda^\circ}$ in $\ell^1(\Lambda^\circ)$, with norm $$\|f\|_{A(\Rdd/\Lambda)}:=\|\{c_{\lambda^\circ}\}\|_{\ell^1(\Lambda^\circ)}.$$ $A'(\Rdd/\Lambda)$  denotes its dual space of distributions with symplectic Fourier coefficients in $\ell^\infty(\Lambda^\circ)$. The proofs of the two parts of the next lemma can be found in \cite[Thm. 7]{Feichtinger:1981} and \cite[Prop. 13]{Keville:2003}, respectively.

\begin{lemma} \label{lem:periodization}
	Let $\Lambda$ be a lattice and $P_\Lambda$ be the periodization operator 
	\begin{equation*}
  P_\Lambda f=\sum_{\lambda \in \Lambda} T_\lambda (f) \quad  \text{ for } f\in M^1(\Rdd).
\end{equation*}
\begin{enumerate}[(a)]
	\item $P_\Lambda$ is bounded and surjective from $M^1(\Rdd)$ onto $A(\Rdd/\Lambda)$.
	\item The range of the Banach space adjoint operator $P_\Lambda^*:A'(\Rdd/\Lambda)\to M^\infty(\Rdd)$ is the set of $\Lambda$-periodic elements of $M^\infty(\Rdd)$.
\end{enumerate}
\end{lemma}
\begin{remark}
\begin{enumerate}[(a)]
	\item A distribution $f\in M^\infty(\Rdd)$ is $\Lambda$-periodic if $T_\lambda (f)=f$ for any $\lambda \in \Lambda$, where $T_\lambda (f)$ is defined by $$\inner{T_\lambda (f)}{g}_{M^\infty,M^1}:=\inner{f}{T_{-\lambda}(g)}_{M^\infty,M^1}$$ for $g\in M^1(\Rdd).$
	\item If $q:\Rdd\to \Rdd/\Lambda$ denotes the quotient map, then a simple calculation using Weil's formula \cite[(6.2.11)]{Grochenig:1998} shows that $P_\Lambda^*(f)=\frac{1}{|\Lambda|}\cdot f\circ q$ for $f\in A(\Rdd/\Lambda)$. 
\end{enumerate}
	
\end{remark}

Since $P_\Lambda f$ has absolutely summable symplectic Fourier coefficients when $f \in M^1(\Rdd)$ by Lemma \ref{lem:periodization}, we can use Poisson's summation formula to find its symplectic Fourier coefficients, see \cite[Example 5.11]{Jakobsen:2018} or \cite[Thm. 3.6.3]{Deitmar:2014} for a proof.

\begin{proposition}[Poisson summation formula]
	Let $f \in M^1(\Rdd)$. The symplectic Fourier coefficients of $P_\Lambda f$ are $\{\frac{1}{|\Lambda|}\F_\sigma(f)(\lambda^\circ)\}_{\lambda^\circ \in \Lambda^\circ}$, i.e.
	\begin{equation*}
  P_\Lambda f (z)=\frac{1}{|\Lambda|} \sum_{\lambda^\circ \in \Lambda^\circ} \F_\sigma(f)(\lambda^\circ)e^{2\pi i \sigma (\lambda^\circ,z)}.
\end{equation*}

\end{proposition}

 To use this to obtain Fourier series of operators, we need the following simple lemma \cite[Prop. 198]{deGosson:2011}.
\begin{lemma} \label{lem:weylsymboltf}
	For any $z=(x,\omega)\in \Rdd$, the Weyl symbol of $e^{- \pi i x \cdot \omega}\pi(z)$  is the function $z'\mapsto e^{2\pi i \sigma (z,z')} $.
\end{lemma}

We will now consider Fourier series of operators arising as periodizations of operators in $\beauty$, in other words a Poisson summation formula for operators. The second part of the result extends Janssen's representation of multi-window Gabor frame operators to Gabor g-frame operators. As mentioned, this result is due to \cite{Feichtinger:1998} who used it to prove the Janssen representation for multi-window Gabor frames. In this and following statements, we use the notation $\lambda^\circ = (\lambda^\circ_x,\lambda^\circ_\omega)$ to denote the elements of $\Lambda^\circ$. 

\begin{theorem}[Janssen's representation of Gabor g-frame operators] \label{theorem:janssen}
	Let $S\in \beauty$ and $\Lambda$ a lattice. Then 
	\begin{equation*}
  \sum_{\lambda \in \Lambda} \alpha_\lambda(S)=\frac{1}{|\Lambda|} \sum_{\lambda^\circ \in \Lambda^\circ} \F_W(S)(\lambda^\circ)e^{-\pi i \lambda^\circ_x\cdot \lambda^\circ_\omega} \pi(\lambda^\circ).
\end{equation*}
In particular,
\begin{equation*}
  \frameop_S=\frac{1}{|\Lambda|} \sum_{\lambda^\circ \in \Lambda^\circ} \F_W(S^*S)(\lambda^\circ)e^{-\pi i \lambda^\circ_x\cdot \lambda^\circ_\omega} \pi(\lambda^\circ).
\end{equation*}
Moreover, if $S\in \beauty_{v\otimes v}$, then $\{\F_W(S)(\lambda^\circ)\}_{\lambda^\circ \in \Lambda^\circ}\in \ell^1_{v}(\Lambda^\circ).$ 
\end{theorem}
\begin{proof}
	Recall that $\alpha_\lambda$ corresponds to a translation of the Weyl symbol by \eqref{eq:translateweyl}. Since the map sending operators in $\beast$ to their Weyl symbols in $M^\infty(\Rdd)$ is weak*-to-weak*-continuous by Proposition \ref{prop:weak*} and $\sum_{\lambda \in \Lambda} \alpha_\lambda(S)$ converges in the weak* topology of $\beast$ by Proposition \ref{prop:feichtingerwindow}, the Weyl symbol $f$ of $\sum_{\lambda \in \Lambda} \alpha_\lambda(S)$ is
	\begin{equation*}
  f =\sum_{\lambda \in \Lambda} T_\lambda (a_{S}) \in M^\infty(\Rdd),
\end{equation*}
where $\weyl_{S}$ is the Weyl symbol of $S$. Hence $f=P_\Lambda \weyl_{S}$. By the Poisson summation formula the symplectic Fourier series of $f$ is given by
\begin{align*}
 f(z)&=\frac{1}{|\Lambda|}\sum_{\lambda^\circ \in \Lambda^\circ} \F_\sigma(\weyl_S)(\lambda^\circ) e^{2\pi i \sigma(\lambda^\circ,z)} \\
  &= \frac{1}{|\Lambda|}\sum_{\lambda^\circ \in \Lambda^\circ} \F_W(S)(\lambda^\circ)  e^{2\pi i \sigma(\lambda^\circ,z)} \quad \text{ by } \eqref{eq:fwsymp}.
\end{align*}
By Proposition \ref{prop:weak*}, $\F_W(S)\in M^1(\Rdd)$, so $\{\F_W(S)(\lambda^\circ)\}_{\lambda^\circ \in \Lambda^\circ}\in \ell^1(\Lambda^\circ)$ by Lemma \ref{lem:wienersampling} -- hence the sum above converges absolutely in the norm of $M^\infty(\Rdd)$. 
Taking the Weyl transform of this using Lemma \ref{lem:weylsymboltf}, we see that 
\begin{equation*}
   \sum_{\lambda \in \Lambda} \alpha_\lambda(S) =\frac{1}{|\Lambda|}\sum_{\lambda^\circ\in \Lambda^\circ} \F_W(S)(\lambda^\circ) e^{-\pi i \lambda^\circ_x\cdot \lambda^\circ_\omega} \pi(\lambda^\circ).
\end{equation*}
For the last part, note that if $S\in \beauty_{v\otimes v}$, then $\F_W(S)\in W(L^1_v)$ by Proposition \ref{prop:fwbounded}, and the result follows from Lemma \ref{lem:wienersampling}. 
\end{proof}

\begin{example}[Multi-window Gabor frames] 
For $\{\phi_n\}_{n=1}^N\subset M^1(\Rd)$, we saw in Example \ref{example:multiwindow} that the frame operator of the multi-window Gabor system generated by $\{\phi_n\}_{n=1}^N$ equals $\frameop_S$ for $S=\sum_{n=1}^N \xi_n \otimes \phi_n,$ where $\{\xi_n\}_{n=1}^N$ is any orthonormal set in $L^2(\Rd)$. Then $$S^*S=\sum_{n=1}^N \phi_n \otimes \phi_n \in \beauty,$$ so by \eqref{eq:fwrankone}
$$\F_W(S^*S)(\lambda^\circ)=\sum_{n=1}^N \F_W(\phi_n\otimes \phi_n)(\lambda^\circ)=\sum_{n=1}^Ne^{\pi i \lambda^\circ_x\cdot \lambda^\circ_\omega} V_{\phi_n}\phi_n(\lambda^\circ).$$
Therefore Theorem \ref{theorem:janssen} gives that
\begin{equation*}
  \frameop_S=\frac{1}{|\Lambda|} \sum_{\lambda^\circ \in \Lambda^\circ}\sum_{n=1}^N V_{\phi_n}\phi_n(\lambda^\circ) \pi(\lambda^\circ),
\end{equation*}
which is the  Janssen representation for multi-window Gabor frames \cite{Dorfler:2011,Janssen:1995}.
\end{example}
We can also prove that any periodic operator in $\beast$ has a Fourier series expansion. By considering Weyl symbols, this is essentially the fact that any $\Lambda$-periodic distribution $f \in M^\infty(\Rdd)$ can be expanded in a symplectic Fourier series, which follows from the second part of Lemma \ref{lem:periodization}. The result is due to \cite{Feichtinger:1998}.

\begin{theorem} \label{theorem:fourierseriesdual}
	Let $S\in \beast$ be a $\Lambda$-periodic operator. Then there exists a unique sequence $\{c_{\lambda^\circ}\}_{\lambda^\circ \in \Lambda^\circ}\in \ell^\infty(\Lambda^\circ)$ such that 
	\begin{equation} \label{eq:fourierbeasts}
  S=\frac{1}{|\Lambda|}\sum_{\lambda^\circ \in \Lambda^\circ} c_{\lambda^\circ}e^{-\pi i \lambda_x^\circ\cdot \lambda_\omega^\circ} \pi(\lambda^\circ),
\end{equation}
with weak* convergence in $\beast$. Furthermore, the map $$\{c_{\lambda^\circ}\}_{\lambda^\circ \in \Lambda^\circ}\mapsto \sum_{\lambda^\circ \in \Lambda^\circ}c_{\lambda^\circ} e^{-\pi i \lambda^\circ_x \cdot \lambda^\circ_\omega} \pi(\lambda^\circ)$$ is weak*-to-weak*-continuous from $\ell^\infty(\Lambda^\circ)$ to $\beast$.
\end{theorem}
\begin{proof}
We first show that series of the form $\frac{1}{|\Lambda|}\sum_{\lambda^\circ \in \Lambda^\circ} c_{\lambda^\circ}e^{-\pi i \lambda_x^\circ \cdot  \lambda_\omega^\circ} \pi(\lambda^\circ)$ converge in the weak* topology of $\beast$ when $\{c_{\lambda^\circ}\}\in \ell^\infty(\Lambda^\circ)$. For $\{c_{\lambda^\circ}\}\in \ell^\infty(\Lambda^\circ)$, we define an antilinear functional on $\beauty$ by 
\begin{equation*}
 \inner{ \frac{1}{|\Lambda|}\sum_{\lambda^\circ \in \Lambda^\circ} c_{\lambda^\circ}e^{-\pi i \lambda_x^\circ \cdot  \lambda_\omega^\circ} \pi(\lambda^\circ)}{T}_{\beast, \beauty}:= \frac{1}{|\Lambda|}\sum_{\lambda^\circ \in \Lambda^\circ}c_{\lambda^\circ} \inner{e^{-\pi i \lambda_x^\circ \cdot  \lambda_\omega^\circ}\pi(\lambda^\circ)}{T}_{\beast, \beauty}.
\end{equation*}
To see that this is a \textit{bounded} functional, let $a_T$ be the Weyl symbol of $T$. Then 
\begin{align*}
 \left| \frac{1}{|\Lambda|}\sum_{\lambda^\circ \in \Lambda^\circ}c_{\lambda^\circ}e^{-\pi i \lambda_x^\circ \cdot  \lambda_\omega^\circ} \inner{\pi(\lambda^\circ)}{T}_{\beast, \beauty}\right|&\lesssim \sum_{\lambda^\circ \in \Lambda^\circ}|c_{\lambda^\circ}| \left|\inner{e^{-\pi i \lambda_x^\circ \cdot  \lambda_\omega^\circ} \pi(\lambda^\circ)}{T}_{\beast, \beauty}\right| \\
 &= \sum_{\lambda^\circ \in \Lambda^\circ}|c_{\lambda^\circ}| \left|\tr(e^{-\pi i \lambda_x^\circ \cdot  \lambda_\omega^\circ} \pi(\lambda^\circ) T^*)\right| 
\end{align*}
where the last step uses \eqref{eq:traceclassduality}.
By the definition of $\F_W$ and equation \eqref{eq:fwofadjoint}, $$tr(e^{-\pi i \lambda_x^\circ \cdot  \lambda_\omega^\circ} \pi(\lambda^\circ) T^*)=\F_W(T^*)(-\lambda^\circ)=\overline{“\mathcal{F}_W(T)(\lambda^\circ)”}.$$ We may therefore continue our estimate by 
\begin{align*}
  \sum_{\lambda^\circ \in \Lambda^\circ}|c_{\lambda^\circ}| \left|\tr(e^{-\pi i \lambda_x^\circ \cdot  \lambda_\omega^\circ} \pi(\lambda^\circ) T^*)\right|&= \sum_{\lambda^\circ \in \Lambda^\circ}|c_{\lambda^\circ}| |\F_W(T)(\lambda^\circ)| \\
 &\lesssim \|\{c_{\lambda^\circ}\}\|_{\ell^\infty(\Lambda^\circ)} \|\F_W(T)\|_{M^1} \quad \text{ by Lem. \ref{lem:wienersampling}}\\
 &\lesssim  \|\{c_{\lambda^\circ}\}\|_{\ell^\infty(\Lambda^\circ)} \|T\|_{\beauty} \quad \text{ by Prop. \ref{prop:weak*}}.
\end{align*}
Hence $\frac{1}{|\Lambda|}\sum_{\lambda^\circ \in \Lambda^\circ} c_{\lambda^\circ}e^{-\pi i \lambda_x^\circ \cdot  \lambda_\omega^\circ} \pi(\lambda^\circ) \in \beast$. The same calculation without absolute values shows that 
\begin{equation} \label{eq:fourierseriesproof}
  \inner{\frac{1}{|\Lambda|}\sum_{\lambda^\circ \in \Lambda^\circ}c_{\lambda^\circ}e^{-\pi i \lambda_x^\circ \cdot  \lambda_\omega^\circ} \pi(\lambda^\circ)}{T}_{\beast, \beauty}=\frac{1}{|\Lambda|}\sum_{\lambda^\circ \in \Lambda^\circ} c_{\lambda^\circ}\overline{\F_W(T)(\lambda^\circ)},
\end{equation}
which implies that the map sending $\{c_{\lambda^\circ}\}$ to this functional is in fact the Banach space adjoint of the map $\beauty \to \ell^1(\Lambda^\circ)$ given by $T\mapsto \{\frac{1}{|\Lambda|} \F_W(T)(\lambda^\circ)\}.$ In particular, the weak*-to-weak* continuity of the map $$\{c_{\lambda^\circ}\}\mapsto \sum_{\lambda^\circ \in \Lambda^\circ}c_{\lambda^\circ} e^{-\pi i \lambda^\circ_x \cdot \lambda^\circ_\omega} \pi(\lambda^\circ)$$  follows, as does the weak* convergence of the sum. 

The uniqueness also follows: the map $\beauty \to \ell^1(\Lambda^\circ)$ defined by $T\mapsto \{\frac{1}{|\Lambda|} \F_W(T)(\lambda^\circ)\}$ is surjective by \cite[Thm. 7 C)]{Feichtinger:1981}
 hence its Banach space adjoint is injective. 
We then turn to finding $\{c_{\lambda^\circ}\}_{\lambda^\circ \in \Lambda^\circ}$ such that \eqref{eq:fourierbeasts} holds. Since $S$ is a $\Lambda$-periodic operator in $\beast$, its Weyl symbol $\weyl_S$ is a $\Lambda$-periodic distribution in $M^\infty(\Rdd)$. By Lemma \ref{lem:periodization} there exists $\tilde{f}\in A'(\Rdd/\Lambda)$ such that $P_\Lambda^* \tilde{f}=\weyl_S$, and we pick $\{c_{\lambda^\circ}\}_{\lambda^\circ \in \Lambda^\circ}$ to be the symplectic Fourier coefficients of $\tilde{f}$. For any $T \in \beauty$ we have from \eqref{eq:gelftripoperators} that
	\begin{align*}
  \inner{S}{T}_{\beast,\beauty}&=\inner{\weyl_S}{\weyl_T}_{M^\infty,M^1} \\
  &= \inner{P_\Lambda^*\tilde{f}}{\weyl_T}_{M^\infty,M^1} \\
  &= \inner{\tilde{f}}{P_\Lambda \weyl_T}_{A'(\Rdd/\Lambda),A(\Rdd/\Lambda)} \\
  &= \frac{1}{|\Lambda|}  \inner{\{c_\lambda^\circ\}}{\{\F_\sigma(\weyl_T)(\lambda^\circ)\}}_{\ell^\infty(\Lambda^\circ),\ell^1(\Lambda^\circ)}.
\end{align*}
In the last equality we have used the Poisson summation formula to get that $\{\frac{1}{|\Lambda|}\F_\sigma(\weyl_T)(\lambda^\circ)\}_{\lambda^\circ \in \Lambda^\circ}$ are the symplectic Fourier coefficients of $P_\Lambda \weyl_T$. By comparing this to \eqref{eq:fourierseriesproof} and using $\F_W(T)=\F_\sigma(a_T)$ by \eqref{eq:fwsymp}, we have proved \eqref{eq:fourierbeasts}. 
\end{proof}
\begin{remark}
\begin{enumerate}[(a)]
	\item The uniqueness part of the previous theorem amounts to a well-known fact: if $\sum_{\lambda^\circ \in \Lambda^\circ}c_{\lambda^\circ} \pi(\lambda^\circ)=0$ for $c=\{c_{\lambda^\circ}\}_{\lambda^\circ\in \Lambda^\circ}\in \ell^\infty(\Lambda^\circ)$, then $c=0.$ Earlier proofs of this fact range from the rather complicated \cite{Rieffel:1988} to the pleasantly elementary \cite{Grochenig:2007}. Our proof is similar to that in \cite{Grochenig:2007}, and comes with a simple interpretation: the Fourier coefficients of periodic operators are unique. 
	\item If $S\in \beast$ is $\Lambda$-periodic and its Weyl symbol $\weyl_S$ belongs to the space $A(\Rdd/\Lambda)$ (i.e. its symplectic Fourier coefficients are absolutely summable), then there exists some $P_S\in \beauty$ such that $S=\sum_{\lambda \in \Lambda} \alpha_\lambda(P_S)$. This is \cite[Thm.. 7.7.6]{Feichtinger:1998}, and follows from applying the surjectivity in part (a) of Lemma \ref{lem:periodization} to $\weyl_S$.
\end{enumerate}
	\end{remark}
\subsection{Poisson summation formula for trace class operators}
When $S\in \tco$ the periodization $\sum_{\lambda \in \Lambda}\alpha_\lambda(S)$ converges in $\beast$ by Proposition \ref{prop:feichtingerwindow}, and by Theorem \ref{theorem:fourierseriesdual} there exists $\{c_{\lambda^\circ}\}\in \ell^\infty(\Lambda^\circ)$ such that $$\sum_{\lambda \in \Lambda}\alpha_\lambda(S)=\sum_{\lambda^\circ \in \Lambda^\circ} c_{\lambda^\circ}e^{-\pi i \lambda_x^\circ\cdot  \lambda_\omega^\circ} \pi(\lambda^\circ).$$ If $S\in \beauty$, we know from Theorem \ref{theorem:janssen} that $c_{\lambda^\circ}$ is given by the samples of $\F_W(S)$. However, even if $S\in \tco\setminus \beauty$, we know from the Riemann-Lebesgue lemma \eqref{eq:riemannlebesgue} that $\F_W(S)\in C_0(\Rdd)$. Hence the samples of $\F_W(S)$ are still well-defined, and we will use a continuity argument to show that $c_{\lambda^\circ}=\F_W(S)(\lambda^\circ)$ also when $S\in \tco \setminus \beauty$. 

\begin{theorem}[Poisson summation formula for trace class operators] \label{prop:traceclassjanssen}
	Let $S\in \tco$. Then 
	\begin{equation*}
  \sum_{\lambda \in \Lambda}\alpha_\lambda(S)=\frac{1}{|\Lambda|}\sum_{\lambda \in \Lambda} \F_W(S)(\lambda^\circ)e^{-\pi i \lambda^\circ_x \cdot  \lambda^\circ_\omega} \pi(\lambda^\circ),
\end{equation*}
with weak* convergence of both sums in $\beast.$
\end{theorem}

\begin{proof}
	 Let $\{S_n\}_{n\in \N}\subset \beauty$ be a sequence converging to $S$ in the norm of $\tco$ using Lemma \ref{lem:denseintraceclass}. By Theorem \ref{theorem:janssen}, we have for each $n\in \N$ that
	\begin{equation} \label{eq:approximatefourierseries}
  \sum_{\lambda \in \Lambda} \alpha_\lambda(S_n)=\frac{1}{|\Lambda|} \sum_{\lambda^\circ \in \Lambda^\circ} \F_W(S_n)(\lambda^\circ) e^{-\pi i \lambda_x^\circ \cdot \lambda_\omega^\circ} \pi(\lambda^\circ).
\end{equation}
By Proposition \ref{prop:feichtingerwindow}, the left hand side of \eqref{eq:approximatefourierseries} converges to $\sum_{\lambda \in \Lambda} \alpha_\lambda(S)$ in $\beast$ as $n\to \infty$. Then note that $$\|\F_W(S)\vert_{\Lambda^\circ}-\F_W(S_n)\vert_{\Lambda^\circ} \|_{\ell^\infty(\Lambda^\circ)}=\|\F_W(S-S_n)\vert_{\Lambda^\circ}\|_{\ell^\infty(\Lambda^\circ)} \leq \|S-S_n\|_{\tco}$$ by \eqref{eq:riemannlebesgue}, hence the samples $\F_W(S_n)\vert_{\Lambda^\circ}$ converge to $\F_W(S)\vert_{\Lambda^\circ}$ in $\ell^\infty(\Lambda^\circ)$ as $n\to \infty$. Combining this with the continuity statement in Theorem \ref{theorem:fourierseriesdual}, we see that the right hand side of \eqref{eq:approximatefourierseries} converges in the weak* topology of $\beast$ to  $\frac{1}{|\Lambda|} \sum_{\lambda^\circ \in \Lambda^\circ} \F_W(S)(\lambda^\circ) e^{-\pi i \lambda_x^\circ \cdot\lambda_\omega^\circ} \pi(\lambda^\circ)$ as $n\to \infty.$ As the limits of both sides of \eqref{eq:approximatefourierseries} must be equal, the result follows.
\end{proof}

\subsection{The twisted Wiener's lemma}
The results in the previous section supplement the theory of the Fourier transform of operators, as introduced by Werner in \cite{Werner:1984}, by showing that periodic operators have a Fourier series expansion. A classic result for Fourier series of functions is Wiener's lemma: if a periodic function is invertible and has an absolutely convergent Fourier series, then its inverse has an absolutely convergent Fourier series. The same holds for operators, by a result due to Gr\"ochenig and Leinert \cite{Grochenig:2004}. Recall that $v$ is a submultiplicative, symmetric GRS-weight -- the GRS condition is crucial for this result.

\begin{theorem} \label{theorem:leinert}
	Assume that $S=\sum_{\lambda^\circ \in \Lambda^\circ} c_{\lambda^\circ} \pi(\lambda^\circ)$ for some sequence $\{c_{\lambda^\circ}\}_{\lambda^\circ \in \Lambda^\circ}\in \ell^1_{v}(\Lambda^\circ)$ and that $S$ is invertible on $L^2(\Rd)$. Then $$S^{-1}=\sum_{\lambda^\circ} a_{\lambda^\circ} \pi(\lambda^\circ)$$ for some sequence $\{a_{\lambda^\circ}\}_{\lambda^\circ \in \Lambda^\circ}\in \ell^1_{v}(\Lambda^\circ)$.
\end{theorem}

This has consequences for Gabor g-frames generated by an operator $S\in \beauty_{v\otimes v}$.

\begin{corollary} \label{cor:dualgenerators0}
Assume that $S\in \beauty_{v\otimes v}$ generates a Gabor g-frame over a lattice $\Lambda$. Then
$\frameop^{-1}_S=\sum_{\lambda^\circ\in \Lambda^\circ} a_{\lambda^\circ} \pi(\lambda^\circ)$ for a sequence $\{a_{\lambda^\circ}\}_{\lambda^\circ \in \Lambda^\circ}\in \ell^1_{v}(\Lambda^\circ).$
\end{corollary}
\begin{proof}
$\frameop_S$ is invertible on $L^2(\Rd)$ as $S$ generates a Gabor g-frame. By the Janssen representation in Theorem \ref{theorem:janssen} we can apply Theorem \ref{theorem:leinert} to $\frameop_S$. 
 \end{proof}

\subsection{Wexler-Raz and some conditions for Gabor g-frames} 
Recall that two operators $S,T\in \HS$ generate \textit{dual Gabor g-frames} if $S$ and $T$ generate Gabor g-frames and
 \begin{equation*}
  \sum_{\lambda \in \Lambda} \alpha_\lambda (S^*T)\psi=\psi \quad \text{ for any } \psi \in L^2(\Rd).
\end{equation*}
 A characterization of dual Gabor g-frames is given by a version of the Wexler-Raz biorthogonality conditions from \cite{Feichtinger:1998}. We extend the result in \cite{Feichtinger:1998} to Hilbert Schmidt operators.
\begin{theorem}[Wexler-Raz biorthogonality]  \label{theorem:wexlerraz}
	Let $S,T\in \HS$ such that $S$ and $T$ satisfy the upper g-frame bound in \eqref{eq:gaborgframes}. Then
	\begin{equation} \label{eq:firstwexler}
  \sum_{\lambda \in \Lambda} \alpha_{\lambda} (S^*T) \psi = \psi \quad \text{ for any } \psi \in L^2(\Rd)
\end{equation}
	  if and only if 
	  \begin{equation} \label{eq:secondwexler}
  \F_W(S^*T)(\lambda^\circ)=|\Lambda| \delta_{\lambda^\circ,0} \quad \text{ for } \lambda^\circ \in \Lambda^\circ.
\end{equation}
\end{theorem}
\begin{proof}
	Our assumption on $S$ and $T$ ensures that $D_SC_T\psi = \sum_{\lambda} \alpha_\lambda (S^*T)\psi$ defines a bounded operator on $L^2(\Rd)$. Since $S,T\in \HS$, we have $S^*T\in \tco$ and by Proposition \ref{prop:traceclassjanssen}
	\begin{equation*}
    \sum_{\lambda\in \Lambda} \alpha_\lambda(S^*T) = \frac{1}{|\Lambda|}  \sum_{\lambda^\circ \in \Lambda^\circ} \F_W(S^*T)(\lambda^\circ)e^{-i\pi \lambda^\circ_x \cdot \lambda^\circ_\omega} \pi ( \lambda^\circ).
	\end{equation*}
	Equation \eqref{eq:firstwexler} states that the left hand side is the identity operator $\pi(0)$, and the uniqueness part of Theorem \ref{theorem:fourierseriesdual} implies that this is true if and only if \eqref{eq:secondwexler} holds. \\
\end{proof}

Note that under the assumptions of Theorem \ref{theorem:wexlerraz}, both $S^*$ and $T$ generate Gabor g-frames by Proposition \ref{prop:dualframes}. As first noted in \cite{Feichtinger:1998}, the theorem reproduces the familiar Wexler-Raz biorthogonality conditions for Gabor frames. 

\begin{example}  
Consider two sets of $N$ functions $\{\phi_n\}_{n=1}^N, \{\psi_n\}_{n=1}^N \subset L^2(\Rd)$. As in Example \ref{example:multiwindow}, we associate an operator to each of these systems:
\begin{align*}
  &S=\sum_{n=1}^N \xi_n \otimes \phi_n, &&T=\sum_{n=1}^N \xi_n \otimes \psi_n,
\end{align*}
where $\{\xi_n\}_{n=1}^N$ is an orthonormal system in $L^2(\Rd)$. Assume that the multi-window Gabor systems generated by $\{\phi_n\}_{n=1}^N$ and $\{\psi_n\}_{n=1}^N$ are Bessel systems, i.e.  $$\sum_{n=1}^N \sum_{\lambda\in \Lambda} |V_{\phi_n}\psi(\lambda)|^2\lesssim \|\psi\|_{L^2}^2 \quad \text{ for any } \psi \in L^2(\Rd),$$ and the same inequality for $\psi_n$. It is a simple exercise to show that this condition implies that $S$ and $T$ satisfy the upper g-frame bound, so Theorem \ref{theorem:wexlerraz} applies. 

Note that $S^*T=\sum_{n=1}^N \phi_n\otimes \psi_n$, and $\F_W(S^*T)(z)=e^{\pi i x \cdot \omega}\sum_{n=1}^N  V_{\psi_n}\phi_n(z)$ by \eqref{eq:fwrankone}. We also find using \eqref{eq:shiftrankone} that
\begin{equation*}
  \sum_{\lambda \in \Lambda} \alpha_\lambda (S^*T) \eta= \sum_{\lambda \in \Lambda}V_{\psi_n}\eta(\lambda) \pi(\lambda)\phi_n \quad \text{ for } \eta \in L^2(\Rd). 
\end{equation*}
Hence Theorem \ref{theorem:wexlerraz} says that 
\begin{equation*}
  \eta= \sum_{\lambda \in \Lambda}V_{\psi_n}\eta(\lambda) \pi(\lambda)\phi_n \quad \text{ for } \eta \in L^2(\Rd)
\end{equation*}
if and only if  $$\sum_{n=1}^N  V_{\psi_n}\phi_n(\lambda^\circ)=|\Lambda| \delta_{\lambda^\circ,0} \quad \text{ for } \lambda^\circ \in \Lambda^\circ.$$ This is the usual version of the Wexler-Raz biorthogonality conditions for multi-window Gabor frames.
\end{example}
We note some simple consequences of Theorem \ref{theorem:wexlerraz}.
\begin{corollary}
	\begin{enumerate}[(a)]
		\item Let $S\in \beauty$. If there exists some $T\in \beauty$ such that $\F_W(S^*T)(0)\neq 0$ and $\F_W(S^*T)(\lambda^\circ)=0$ for $\lambda^\circ\neq 0$, then $S$ generates a Gabor g-frame.
		\item Let $S\in \beauty$. If there exist $\phi,\psi \in M^1(\Rd)$ such that $$V_{\phi} (S^*\psi)(\lambda^\circ)=|\Lambda| \delta_{\lambda^\circ,0}\quad \text{ for } \lambda^\circ \in \Lambda^\circ,$$ then $S$ generates a Gabor g-frame.
		\item If $S\in \beauty$ satisfies $\Lambda^\circ\cap\{z'-z'':z',z''\in \text{supp}(\F_W(S))\} =\{0\}$, then $S$ generates a tight Gabor g-frame.    
	\end{enumerate}
\end{corollary}

\begin{proof}
\begin{enumerate}[(a)]
	\item Define $\tilde{T}=\frac{|\Lambda|}{\F_W(S^*T)(0)}T$. Then $S,\tilde{T}\in \beauty$, so $S,\tilde{T}$ satisfy the upper g-frame bound in \eqref{eq:gaborgframes} by Corollary \ref{cor:beautyupperbound}. Hence Theorem \ref{theorem:wexlerraz} applies to give that $S,\tilde{T}$ generate dual Gabor g-frames, and the result follows from Proposition \ref{prop:dualframes}. 
	\item Let $T=\psi \otimes \phi.$ Then  $S^*T=(S^*\psi)\otimes \phi$. Since $\F_W( (S^*\psi)\otimes \phi)(x,\omega)=e^{\pi i x\cdot \omega}V_\phi(S^*\psi)(x,\omega)$ by \eqref{eq:fwrankone}, the result follows from part (a).
	\item It is well-known (see \cite{Luef:2017,Feichtinger:1998}) that $$\F_W(S^*S)(z)= \int_{\Rdd} \F_W(S^*)(z-z')\F_W(S)(z') e^{\pi i \sigma(z,z')} \ dz',$$ where the right hand side is the so-called twisted convolution of $\F_W(S^*)$ with $\F_W(S)$. \eqref{eq:fwofadjoint} we get $$\F_W(S^*S)(z)= \int_{\Rdd} \overline{\F_W(S)(z'-z)}\F_W(S)(z') e^{\pi i \sigma(z,z')} \ dz'.$$ One easily deduces that a necessary condition for $\F_W(S^*S)(z)$ to be non-zero is that $z=z'-z''$, where both $z',z''\in \text{supp}(\F_W(S))$, hence the condition in the statement ensures that $\F_W(S^*S)(\lambda^\circ)=0$ for $\lambda^\circ \neq 0$. In addition, $\F_W(S^*S)(0)=\tr(S^*S)=\|S\|_{\HS}^2>0$. Therefore $\tilde{S}=\frac{\sqrt{|\Lambda|}}{\|S\|_\HS}S$ satisfies $$\sum_{\lambda \in \Lambda} \alpha_{\lambda}(\tilde{S}^*\tilde{S})\psi=\psi$$ for any $\psi \in L^2(\Rd)$ by Theorem \ref{theorem:janssen}, which implies that $$\sum_{\lambda \in \Lambda} \alpha_{\lambda}(S^*S)\psi=\frac{\|S\|_{\HS}^2}{|\Lambda|}\psi.$$ 
\end{enumerate}	
\end{proof}
\begin{remark}
\begin{enumerate}[(a)]
	\item The condition in part (c) above can be satisfied if $S$ is an \textit{underspread} operator (as defined by Kozek \cite{Kozek:1997thesis,Kozek:1997,Kozek:1998}), with $\text{supp}(\F_W(S))\subset B_R(0)$ for some small $R>0$, where $B_R(0)\subset \Rdd$ is the ball of radius $R$ centered at $0$. In this case $\{z'-z'':z',z''\in \text{supp}(\F_W(S))\}\subset B_{2R}(0)$, so by picking sufficiently small $R$ the condition in the corollary can be satisfied. Such $S$ may easily be constructed, for instance by picking a smooth bump function $f\in M^1(\Rdd)$ supported in $B_R(0)$ -- since $\F_W$ is bijective from $\beauty$ to $M^1(\Rdd)$, there exists some $S\in \beauty$ with $\F_W(S)=f.$ By a result of Janssen \cite{Janssen:1998a} this simple construction will never work for Gabor frames: there is no rank-one operator $S=\psi \otimes \phi$ such that $\F_W(S)(x,\omega)=e^{\pi i x \cdot \omega}V_\phi \psi(x,\omega)$ has compact support.	
	\item If $\Lambda$ is a separable lattice $\Lambda=\alpha \Z^d \times \beta \Z^d$ for $\alpha,\beta \in \R$, then $\Lambda^\circ=\frac{1}{\beta} \Z^d \times \frac{1}{\alpha} \Z^d$. It follows from the Janssen representation that if $\F_W(S^*S)(\frac{m}{\beta},\frac{n}{\alpha})=0$ whenever $0\neq m\in \Z^d$, then the g-frame operator is simply the multiplication operator $$\psi(t) \mapsto \left(\frac{1}{\alpha \beta} \sum_{n\in \Z^d} \F_W(S^*S)\left(0,\frac{n}{\alpha}\right) e^{2\pi i n\cdot t/\alpha }\right) \psi(t).$$ If $S$ is a rank-one operator $\phi \otimes \phi$, this can be achieved by picking compactly supported $\phi$ -- this leads to the \textit{painless nonorthogonal expansions} of \cite{Daubechies:1986}.
\end{enumerate}

\end{remark}

The Wexler-Raz conditions sometimes allow us to deduce that $S$ and $T$ generate dual Gabor g-frames, or, when $S=T$, that $S$ generates a tight Gabor g-frame. The Janssen representation also implies the following test for deciding when $S\in \beauty$ generates a (not necessarily tight) Gabor g-frame. 

\begin{proposition} \label{prop:janssentest}
	Let $S\in \beauty$, and assume that $\sum_{0\neq \lambda^\circ \in \Lambda^\circ} |\F_W(S^*S)(\lambda^\circ)|<\|S\|_{\HS}^2$. Then $S$ generates a Gabor g-frame. 
\end{proposition}
\begin{proof}
	By the Janssen representation and the fact that $\F_W(S^*S)(0)=\tr(S^*S)=\|S\|^2_{\HS}>0$,
	\begin{align*}
  \frameop_S&=\frac{1}{|\Lambda|} \sum_{\lambda^\circ \in \Lambda^\circ} \F_W(S^*S)(\lambda^\circ)e^{-\pi i \lambda^\circ_x\cdot \lambda^\circ_\omega} \pi(\lambda^\circ) \\
  &= \frac{\|S\|^2_{\HS}}{|\Lambda|}\underbrace{\left(I + \sum_{0\neq \lambda^\circ \in \Lambda^\circ} \frac{\F_W(S^*S)(\lambda^\circ)}{\|S\|^2_{\HS}}e^{-\pi i \lambda^\circ_x \cdot \lambda^\circ_\omega} \pi(\lambda^\circ) \right)}_{:=A},
\end{align*}
so $\frameop_S$ has a bounded inverse on $L^2(\Rd)$ if and only if $A$ has a bounded inverse. As
\begin{equation*}
  \|A-I\|_{\bo}\leq \sum_{0\neq \lambda^\circ \in \Lambda^\circ} \frac{|\F_W(S^*S)(\lambda^\circ)|}{\|S\|_{\HS}^2}<1,
\end{equation*}
by assumption, the Neumann theorem \cite[Thm. A.3]{Grochenig:2001} implies that $A$ has a bounded inverse on $L^2(\Rd)$.
\end{proof}

\begin{remark}
  When $S=\phi\otimes \phi$ for some $\phi \in M^1(\Rd)$, the proposition above becomes a well-known result for Gabor frames. To our knowledge the first appearance of this special case in the literature is \cite[Thm. 4.1.1]{Tschurtschenthaler:2000}.
\end{remark}

\begin{corollary}
	Let $0\neq S\in \beauty$ and $\Lambda$ a lattice. There exists $N\in \N$ such that $S$ generates a Gabor g-frame over the lattice $\frac{1}{N}\Lambda$.
\end{corollary}
\begin{proof}
	Since $\sum_{\lambda^\circ \in \Lambda} |\F_W(S^*S)(\lambda^\circ)|<\infty$ by Theorem \ref{theorem:janssen}, there exists $K\in \N$ with 
	\begin{equation*}
  \sum_{|\lambda^\circ|>K}  |\F_W(S^*S)(\lambda^\circ)|<\|S\|_{\HS}^2.
\end{equation*}
Let $N\in \N$ be the smallest integer such that $|\lambda^\circ|> K/N$ for any $0\neq \lambda^\circ \in \Lambda^\circ$, and consider the lattice $\Gamma=\frac{1}{N}\Lambda$. Then $\Gamma^\circ=N \Lambda^\circ\subset \Lambda^\circ$. 
By definition, the non-zero elements $\gamma^\circ \in \Gamma^\circ$ are all of the form $\gamma^\circ = N \lambda^\circ$. In particular, they satisfy $|\gamma^\circ|>K$ and $\gamma^\circ \in \Lambda^\circ$. Therefore 
  \begin{equation*}
  \sum_{0\neq \gamma^\circ \in \Gamma^\circ} |\F_W(S^*S)(\gamma^\circ)|\leq
  \sum_{|\lambda^\circ|>K}  |\F_W(S^*S)(\lambda^\circ)|<\|S\|_{\HS}^2,
\end{equation*}
hence $S$ generates a Gabor g-frame with respect to $\Gamma=\frac{1}{N}\Lambda$ by Proposition \ref{prop:janssentest}. 
\end{proof}

\section{Gabor g-frames and modulation spaces} \label{sec:eqnorms}
It is a well-known fact that if a function $\phi\in M^1_v(\Rd)$ generates a Gabor frame, then the $\ell^p_m(\Lambda)$-norm of the coefficients $\{V_{\phi}\psi(\lambda)\}_{\lambda \in \Lambda}$ is an equivalent norm to $\|\psi\|_{M^p_m}.$ To extend this result to Gabor g-frames, we will need to introduce some appropriate Banach spaces. Once this is done, our proofs will mainly proceed by reducing the statement for Gabor g-frames to the statement for Gabor frames, which may be found in the standard reference \cite{Grochenig:2001}.

 For $p\in [1,\infty]$ and a $v$-moderate weight $m$ we define the space $\ell^p_m(\Lambda;L^2)$ to be the Banach space of sequences $\{\psi_\lambda\}_{\lambda \in \Lambda}\subset L^2(\Rd)$ such that
\begin{equation*}
  \|\{\psi_\lambda\}\|_{\ell^p_m(\Lambda;L^2)}:=\left(\sum_{\lambda \in \Lambda} \|\psi_\lambda\|^p_{L^2}m(\lambda)^p \right)^{1/p}<\infty.
\end{equation*}
For $p=\infty$ the sum is replaced by a supremum in the usual way. For $m\equiv 1$ we write $\ell^p_m(\Lambda;L^2)=\ell^p(\Lambda;L^2).$ The dual space of $\ell^p_m(\Lambda;L^2)$ for $p<\infty$ is $\ell^{p'}_{1/m}(\Lambda;L^2)$ with
\begin{equation} \label{eq:sequenceduality}
  \inner{\{\phi_\lambda\}}{\{\psi_\lambda\}}_{\ell^{p'}_{1/m}(\Lambda;L^2),\ell^p_m(\Lambda;L^2)}=\sum_{\lambda \in \Lambda} \inner{\phi_\lambda}{\psi_\lambda}_{L^2}
\end{equation}
for $\{\phi_\lambda\}_{\lambda \in \Lambda}\in \ell^{p'}_{1/m}(\Lambda;L^2),\{\psi_\lambda\}_{\lambda \in \Lambda}\in \ell^p_m(\Lambda;L^2).$ It is clear from the definitions that finite sequences $\{\psi_\lambda\}_{\lambda\in \Lambda}$ (meaning that $\psi_\lambda\neq 0$ for finitely many $\lambda$) are dense in $\ell^p_m(\Lambda;L^2)$ for $p<\infty$ and weak*-dense in $\ell^\infty_{m}(\Lambda;L^2)$.
\begin{remark}
	The norm $\|\{\psi_\lambda\}\|_{\ell^p_m(\Lambda;L^2)}$ equals $\|\{m(\lambda)\cdot\psi_\lambda\}\|_{L^p(\Lambda,L^2)}$, where $L^p(\Lambda,L^2)$ is a vector-valued $L^p$-space with $\Lambda$ equipped with counting measure. Since $m(\lambda)>0$ for any $\lambda \in \Lambda$, we may immediately translate results from the theory of vector-valued $L^p$-spaces, see Chapter 1 of \cite{Hytonen:2016}, into statements about $\ell^p_m(\Lambda;L^2)$. In particular, they are Banach spaces and the duality \eqref{eq:sequenceduality} follows from \cite[Prop. 1.3.3]{Hytonen:2016}. 
\end{remark}

We have already met the space $\ell^2(\Lambda;L^2)$, and seen that $C_S$ is bounded from $L^2(\Rd)$ into $\ell^2(\Lambda;L^2)$ when $S$ generates a Gabor g-frame. The next result shows that this result can be generalized to other $p$ and $m$ when $S\in \beauty_{v\otimes v}$.

\begin{theorem} \label{theorem:upperbound}
	If $S\in \beauty_{v\otimes v}$ and $p\in [1,\infty]$, then the analysis operator $C_S$ is bounded from $M^p_m(\Rd)$ to $\ell^p_m(\Lambda;L^2)$ with operator norm $\|C_S\|_{M^p_m\to \ell^p_m(\Lambda;L^2)}\lesssim \|S\|_{\beauty_{v\otimes v}}$
	where the implicit constant is independent of $p$ and $m$. 
\end{theorem}
\begin{proof} 
Let 
\begin{equation*}
  S=\sum_{n\in \N} \phi_n^{(1)}\otimes \phi_n^{(2)},
\end{equation*} 
be a decomposition as in part (a) of Proposition \ref{prop:innerkernel}. Then 
\begin{align*}
  \|\alpha_\lambda (S) \psi\|_{L^2}&=\left\|\left(\sum_{n \in \mathbb{N}} \pi(\lambda) \phi_n^{(1)}\otimes \pi(\lambda)\phi_n^{(2)}\right) \psi  \right\|_{L^2}\\ 
  &\leq \sum_{n \in \mathbb{N}}   |V_{\phi_n^{(2)}}\psi (\lambda)|\cdot  \|\pi(\lambda)\phi_n^{(1)}\|_{L^2} \\
  &=\sum_{n \in \mathbb{N}}   |V_{\phi_n^{(2)}}\psi (\lambda)|\cdot  \|\phi_n^{(1)}\|_{L^2} \lesssim \sum_{n \in \mathbb{N}}  |V_{\phi_n^{(2)}}\psi (\lambda)|\cdot  \|\phi_n^{(1)}\|_{M^1_v},
\end{align*}
where the last inequality uses $M_{v}^1(\Rd)\hookrightarrow L^2(\Rd)$.
Then assume that $p<\infty$, and use the inequality above and the triangle inequality for $\ell^{p}_{m}(\Lambda)$ to get
\begin{multline*}
  \left( \sum_{\lambda \in \Lambda}\|\alpha_\lambda (S) \psi\|_{L^2}^p m(\lambda)^p\right)^{1/p}  \\
  \begin{aligned}
    &\lesssim  \left(\sum_{\lambda \in \Lambda} \left( \sum_{n \in \mathbb{N}}   |V_{\phi_n^{(2)}}\psi (\lambda)|\cdot  \|\phi_n^{(1)}\|_{M^1_v}\right)^{p}m(\lambda)^p\right)^{1/p}\\
    &=\left\|\sum_{n \in \mathbb{N}} \left\{  |V_{\phi_n^{(2)}}\psi (\lambda)|\cdot \|\phi_n^{(1)}\|_{M^1_v}\right\}_{\lambda \in \Lambda} \right\|_{\ell^p_{m}(\Lambda)} \\
    &\leq \sum_{n \in \mathbb{N}}  \left\| \left\{ |V_{\phi_n^{(2)}}\psi (\lambda)|\cdot \|\phi_n^{(1)}\|_{M^1_v}\right\}_{\lambda \in \Lambda}\right\|_{\ell^p_{m}(\Lambda)}  \\
    &= \sum_{n \in \mathbb{N}}  \|\phi_n^{(1)}\|_{M^1_v} \left\| \left\{ |V_{\phi_n^{(2)}}\psi (\lambda)|\right\}_{\lambda \in \Lambda}\right\|_{\ell^p_{m}(\Lambda)} 
    \\
    &\lesssim  \|\psi\|_{M_{m}^{p}} \sum_{n \in \mathbb{N}}  \|\phi_n^{(1)}\|_{M^1_v} \|\phi_n^{(2)}\|_{M_{v}^1} \quad \text{ by Lemma \ref{lem:lpstft}}.
  \end{aligned}
\end{multline*}
The norm inequality $\|C_S\|_{op}\lesssim \|S\|_{\beauty_{v\otimes v}}$ then follows from part (b) of Proposition \ref{prop:innerkernel}.
For $p=\infty$, we use Lemma \ref{lem:lpstft} to find that for any $\lambda\in \Lambda$
\begin{align*}
  \|\alpha_\lambda (S) \psi \|_{L^2}\cdot  m(\lambda)&\lesssim  \sum_{n \in \mathbb{N}} |V_{\phi_n^{(2)}}\psi(\lambda)| \cdot m(\lambda) \cdot \|\phi_n^{(1)}\|_{M^1_v} \\
  &\leq \|\psi\|_{M^{\infty}_{m}} \sum_{n \in \mathbb{N}}  \|\phi_n^{(2)}\|_{M_{v}^1}  \|\phi_n^{(1)}\|_{M^1_v}  .
\end{align*}
\end{proof}

\begin{theorem}  \label{theorem:synthesis}
If $S\in \beauty_{v\otimes v}$ and $p\in [1,\infty]$, then the synthesis operator $D_S$ is bounded from $\ell^p_m(\Lambda;L^2)$ to $M^p_m(\Rd)$,
	with operator norm $\|D_S\|_{\ell^p_m(\Lambda;L^2)\to M^p_m}\lesssim \|S\|_{\beauty_{v\otimes v}}$ independent of $p$ and $m$. For $\{\psi_\lambda\}_{\lambda \in \Lambda}\in \ell^p_m(\Lambda;L^2)$, the expansion
\begin{equation*}
  D_S(\{\psi_\lambda\})=\sum_{\lambda \in \Lambda} \alpha_{\lambda}(S^*) \psi_\lambda
\end{equation*}
converges unconditionally in $M^p_m(\Rd)$ for $p<\infty$ and in the weak* topology of $M^\infty_{1/v}(\Rd)$ for $p=\infty$.	
\end{theorem}
\begin{proof}
	First assume that $p<\infty$, and let $\{\psi_\lambda\}_{\lambda \in \Lambda}$ be a finite sequence. Using Proposition \ref{prop:innerkernel} we write $
  S=\sum_{n\in \N} \phi_n^{(1)}\otimes \phi_n^{(2)}$.
 Then one finds using \eqref{eq:shiftrankone} that 
 \begin{align*}
  D_S(\{\psi_\lambda\})&=\sum_{\lambda \in \Lambda} \sum_{n\in \N} V_{\phi_n^{(1)}}\psi_\lambda(\lambda) \pi(\lambda)\phi_n^{(2)}\\
  &= \sum_{n\in \N} \sum_{\lambda \in \Lambda}  V_{\phi_n^{(1)}}\psi_\lambda(\lambda) \pi(\lambda)\phi_n^{(2)}.
\end{align*}
Interchanging the order of summation is allowed as the finiteness of the sum over $\lambda$ implies  absolute convergence in $M^p_m(\Rd)$: by parts (c) and (e) of Proposition \ref{prop:modulationspaces}
\begin{equation*} 
  \|\pi(\lambda)\phi_n^{(2)}\|_{M^p_m}\lesssim v(\lambda) \|\phi_n^{(2)}\|_{M^1_v},
\end{equation*}
and by Cauchy-Schwarz and $M^1_v(\Rd)\hookrightarrow L^2(\Rd)$
\begin{equation} \label{eq:cauchyschwarz}
  |V_{\phi_{n}^{(1)}}\psi_\lambda(\lambda)|=\left|\inner{\psi_\lambda}{\pi(\lambda)\phi_n^{(1)}}_{L^2}\right|\lesssim \|\psi_\lambda\|_{L^2} \|\phi_n^{(1)}\|_{M^1_v}.
\end{equation}
 Hence the absolute convergence follows by  
 \begin{multline*}
  \sum_{n\in \N} \sum_{\lambda \in \Lambda} | V_{\phi_n^{(1)}}\psi_\lambda(\lambda)|\cdot \| \pi(\lambda)\phi_n^{(2)} \|_{M^p_m}  \\
  \begin{aligned}
    &\lesssim  \sum_{n\in \N} \sum_{\lambda \in \Lambda} \|\psi_\lambda\|_{L^2} \|\phi_n^{(1)}\|_{M^1_v}\cdot v(\lambda)\cdot  \|\phi_n^{(2)} \|_{M^1_v} \\
    &= \left(\sum_{n\in \N} \|\phi_n^{(1)}\|_{M^1_v} \|\phi_n^{(2)}\|_{M^1_v} \right)\left(\sum_{\lambda \in \Lambda} \|\psi_\lambda\|_{L^2} v(\lambda)\right) \\
    &<\infty.
  \end{aligned}
\end{multline*}
Now apply the $M^p_m$-norm to our expression for $D_S(\{\psi_\lambda\})$. When passing to the second line, we use \cite[Thm. 12.2.4]{Grochenig:2001}, which is the Gabor frame version of the statement we are proving, and the implicit constant is independent of $p$ and $m$. 
\begin{align*}
   \|D_S(\{\psi_\lambda\})\|_{M^p_m}&\leq \sum_{n\in \N} \left\| \sum_{\lambda \in \Lambda}  V_{\phi_n^{(1)}}\psi_\lambda(\lambda) \pi(\lambda)\phi_n^{(2)} \right\|_{M^p_m} \\
   &\lesssim \sum_{n\in \N}\|\phi_n^{(2)}\|_{M^1_v} \|\{V_{\phi_{n}^{(1)}}\psi_\lambda\}\|_{\ell^{p}_m(\Lambda)} \\
  &\leq \sum_{n\in \N}\|\phi_n^{(1)}\|_{M^1_v}\|\phi_n^{(2)}\|_{M^1_v} \|\{\|\psi_\lambda\|_{L^2}\}\|_{\ell^{p}_m(\Lambda)} \quad \text{ by \eqref{eq:cauchyschwarz}} \\
  &=\|\{\psi_\lambda\}\|_{\ell^p_m(\Lambda;L^2)}\sum_{n\in \N}\|\phi_n^{(2)}\|_{M^1_v}\|\phi_n^{(1)}\|_{M^1_v}.
\end{align*}
Since finite sequences are dense in $\ell^p_m(\Lambda;L^2)$, this shows that $D_S$ extends to a bounded operator $\ell^p_m(\Lambda;L^2)\to M^p_m(\Rd)$ and $\|D_S\|_{\ell^p_m(\Lambda;L^2)\to M^p_m}\lesssim \|S\|_{\beauty_{v\otimes v}}$ follows from part (b) of Proposition \ref{prop:innerkernel}. The same proof works for $p=\infty$ when replacing the sum with a supremum.
For the unconditional convergence for $p<\infty$, let $J\subset \Lambda$ be a finite subset and let $\{\psi_\lambda\}_{\lambda\in \Lambda}\in \ell^p_m(\Lambda;L^2)$. Then 
\begin{align*}
  \|D_S(\{\psi_\lambda\})-\sum_{\lambda \in J} \alpha_\lambda (S^*)\psi_\lambda \|_{M^p_m(\Rd)}^p &=\|D_S(\{\psi_\lambda\}_{\lambda\in \Lambda}-\{\psi_\lambda\}_{\lambda\in J})\|_{M^p_m}^p\\
  &\lesssim \|\{\psi_\lambda\}_{\lambda\in \Lambda}-\{\psi_\lambda\}_{\lambda\in J}\|_{\ell^p_m(\Lambda;L^2)}^p \\
  &= \sum_{\lambda \in \Lambda\setminus J} \|\psi_\lambda\|_{L^2}^p m(\lambda)^p.
\end{align*}
As the sum $\sum_{\lambda \in \Lambda} \|\psi_\lambda\|_{L^2}^p m(\lambda)^p$ converges by assumption, the estimate above shows that for any $\epsilon>0$ we can find a finite subset $J_\epsilon \subset \Lambda$ such that  $\|D_S(\{\psi_\lambda\})-\sum_{\lambda \in J} \alpha_\lambda (S^*)\psi_\lambda \|_{M^p_m}^p<\epsilon$ whenever $J_\epsilon\subset J$. It follows that $\sum_{\lambda \in \Lambda} \alpha_\lambda (S^*)\psi_\lambda$ converges to $D_S(\{\psi_\lambda\})$ in the sense that the net of partial sums converges, which implies unconditional convergence \cite[Prop. 5.3.1]{Grochenig:2001}. \\
If $p=\infty,$ let $\phi\in M^1_v(\Rd)$. Then
\begin{multline*}
  \sum_{\lambda \in \Lambda} |\inner{ \alpha_\lambda(S^*)\psi_\lambda}{\phi}_{M^\infty_{1/v},M^1_v}|  \\
  \begin{aligned}
    &=\sum_{\lambda \in \Lambda} |\inner{ \psi_\lambda}{\alpha_\lambda(S)\phi}_{L^2}| \quad \text{by Prop. \ref{prop:innerkernel} (c) } \\
    &\leq \sum_{\lambda \in \Lambda} \|\psi_\lambda\|_{L^2} \frac{1}{v(\lambda)}  \|\alpha_\lambda(S)\phi\|_{L^2}v(\lambda) \quad \text{ by Cauchy-Schwarz} \\
 &\leq \|\{\psi_\lambda\}\|_{\ell^\infty_{1/v}(\Lambda;L^2)}\|C_S(\phi)\|_{\ell^1_{v}(\Lambda,L^2)} \\
 &\lesssim \|\{\psi_\lambda\}\|_{\ell^\infty_{1/v}(\Lambda;L^2)}  \|S\|_{\beauty_{v\otimes v}} \|\phi\|_{M^1_v} \quad \text{ by Theorem \ref{theorem:upperbound}} .
  \end{aligned}
\end{multline*}
Hence the sum $\sum_{\lambda \in \Lambda} \inner{ \alpha_\lambda(S^*)\psi_\lambda}{\phi}_{M^\infty_{1/v},M^1_v}$ converges absolutely for $\phi\in M^1_v(\Rd)$.
\end{proof}

When $p<\infty$, $\{\psi_\lambda\}_{\lambda \in \Lambda}\in \ell^p_m(\Lambda;L^2)$ and $\phi\in M^{p'}_{1/m}(\Rd)$, one finds that
	\begin{align*}
   \inner{\phi}{D_S(\{\psi_\lambda\})}_{M^{p'}_{1/m},M^p_m}&= \inner{\phi}{\sum_{\lambda \in \Lambda} \alpha_{\lambda}(S^*)\psi_\lambda}_{M^{p'}_{1/m},M^p_m}\\
  &= \sum_{\lambda \in \Lambda}\inner{\phi}{ \alpha_{\lambda}(S^*)\psi_\lambda}_{M^{p'}_{1/m},M^p_m} \\
  &= \sum_{\lambda \in \Lambda} \inner{\alpha_{\lambda}(S)\phi}{ \psi_\lambda}_{L^2} \quad \text{ by Prop. \ref{prop:innerkernel} (c)} \\
  &= \inner{C_S(\phi)}{\{\psi_\lambda\}_{\lambda \in \Lambda}}_{\ell^{p'}_{1/m}(\Lambda;L^2),\ell^p_m(\Lambda;L^2)}.
\end{align*}
In the same way, when $\{\psi_\lambda\}_{\lambda \in \Lambda}\in \ell^{p'}_{1/m}(\Lambda;L^2)$ and $\phi\in M^{p}_m(\Rd)$, one shows that $$\inner{D_S(\{\psi_\lambda\})}{\phi}_{M^{p'}_{1/m},M^p_m}=\inner{\{\psi_\lambda\}}{C_S(\phi)}_{\ell^{p'}_{1/m}(\Lambda;L^2),\ell^p_m(\Lambda;L^2)}.$$

These calculations and the fact that Banach space adjoints are weak*-to-weak*-continuous imply the following result.
\begin{corollary} \label{cor:weakstar}
	Let $p<\infty$. The analysis operator 
	\begin{equation*}
  C_S:M^{p'}_{1/m}(\Rd)\to \ell^{p'}_{1/m}(\Lambda;L^2)
\end{equation*}
is the Banach space adjoint of the synthesis operator
\begin{equation*}
  D_S: \ell^p_m(\Lambda;L^2)\to M^p_m(\Rd).
\end{equation*}
Similarly, the synthesis operator $D_S: \ell^{p'}_{1/m}(\Lambda;L^2)\to M^{p'}_{1/m}(\Rd)$ is the Banach space adjoint of the analysis operator $C_S:M^p_m(\Rd)\to \ell^p_m(\Lambda;L^2)$.
In particular, both $C_S:M^{p'}_{1/m}(\Rd)\to \ell^{p'}_{1/m}(\Lambda;L^2)$ and $D_S: \ell^{p'}_{1/m}(\Lambda;L^2)\to M^{p'}_{1/m}(\Rd)$ are weak*-to-weak*-continuous.
\end{corollary}

Using the Janssen representation, we deduced in Corollary \ref{cor:dualgenerators0} that if $S\in \beauty_{v\otimes v}$ generates a Gabor g-frame, then $\frameop_S^{-1}$ has a representation 
\begin{equation*}
  \frameop_S^{-1}  = \frac{1}{|\Lambda|} \sum_{\lambda^\circ \in \Lambda^\circ} c_{\lambda^\circ} \pi(\lambda^\circ)
\end{equation*}
for some sequence $\{c_{\lambda^\circ}\}\in \ell^1_v(\Lambda^\circ)$. Since $\pi(\lambda^\circ)$ is bounded on any modulation space $M^p_m(\Rd)$ by Proposition \ref{prop:modulationspaces}, we find that $\frameop_S^{-1}$ extends to a bounded operator on any modulation space by
\begin{align*}
  \|\frameop_S^{-1}\psi\|_{M^{p}_m}&\leq \frac{1}{|\Lambda|} \sum_{\lambda^\circ \in \Lambda^\circ} |c_{\lambda^\circ}| \|\pi(\lambda^\circ)\psi\|_{M^p_m} \\
  &\lesssim  \sum_{\lambda^\circ \in \Lambda^\circ} |c_{\lambda^\circ}|v(\lambda^\circ) \|\psi\|_{M^p_m} =  \|\psi\|_{M^p_m}\|\{c_{\lambda^\circ}\}\|_{\ell^1_v(\Lambda^\circ)}.
\end{align*}

Then recall that the canonical dual Gabor g-frame is generated by the operator $S\frameop_S^{-1}$. The next result shows that  $S\frameop_S^{-1}$ also satisfies the assumptions of Theorems \ref{theorem:upperbound} and \ref{theorem:synthesis}. 

\begin{proposition} \label{prop:dualgframeM}
	 If $S\in \beauty_{v\otimes v}$ generates a Gabor g-frame, then $S\frameop_S^{-1}\in \beauty_{v\otimes v}.$
\end{proposition}

\begin{proof}
	Let
	\begin{equation*}
  S=\sum_{n\in \N} \phi_n^{(1)}\otimes \phi_n^{(2)},
\end{equation*} 
be a decomposition of $S$ from Proposition \ref{prop:innerkernel}. For $\psi \in L^2(\Rd)$, this implies that 
\begin{align*}
  S\frameop_S^{-1}\psi&= \sum_{n\in \N} \inner{\frameop_S^{-1}\psi}{\phi_n^{(2)}}_{L^2} \phi_n^{(1)} \\
  &= \sum_{n\in \N} \inner{\psi}{\frameop_S^{-1}\phi_n^{(2)}}_{L^2} \phi_n^{(1)},
\end{align*}
where we have used that $\frameop_S^{-1}$ is positive and therefore self-adjoint. Hence
\begin{equation*}
  S\frameop_S^{-1}=\sum_{n\in \N} \phi_n^{(1)}\otimes (\frameop_S^{-1}\phi_n^{(2)}),
\end{equation*}
and this decomposition converges absolutely in $\beauty_{v\otimes v}$ since 
\begin{align*}
  \sum_{n\in \N}\| \phi_n^{(1)}\otimes (\frameop_S^{-1}\phi_n^{(2)})\|_{\beauty_{v\otimes v}}&=\sum_{n\in \N} \|\phi_n^{(1)}\|_{M^1_{v}} \|\frameop_S^{-1}\phi_n^{(2)}\|_{M^1_{v}} \\
  &\lesssim  \sum_{n\in \N} \|\phi_n^{(1)}\|_{M^1_{v}} \|\phi_n^{(2)}\|_{M^1_{v}} <\infty
\end{align*}
by the aforementioned boundedness of $\frameop_S^{-1}:M^1_v(\Rd)\to M^1_v(\Rd)$.
\end{proof}

\begin{corollary} \label{cor:gaborexpansions}
	Assume that $S\in \beauty_{v\otimes v}$ generates a Gabor g-frame. For any $\psi\in M^p_m(\Rd)$, the expansions
	\begin{align*}
  \psi&= D_SC_{S\frameop_S^{-1}}\psi=\sum_{\lambda\in \Lambda} \alpha_\lambda(S^*)\alpha_\lambda(S\frameop_S^{-1}) \psi=\sum_{\lambda \in \Lambda} \alpha_\lambda(S^*S\frameop_S^{-1})\psi, \\
  \psi&=D_{S\frameop_S^{-1}}C_S\psi =\sum_{\lambda\in \Lambda} \alpha_\lambda((S\frameop_S^{-1})^*)\alpha_\lambda(S) \psi=\sum_{\lambda \in \Lambda} \alpha_\lambda(\frameop_S^{-1}S^*S)\psi
\end{align*}
converge unconditionally in $M^p_m(\Rd)$ for $p<\infty$ and in the weak* topology of $M^\infty_{1/v}(\Rd)$ for $p=\infty$.
\end{corollary}
\begin{proof}
	We prove the result for $D_SC_{S\frameop_S^{-1}},$ the same proof works for $D_{S\frameop_S^{-1}}C_S$. 
	From the previous proposition, we know that $S,S\frameop^{-1}_S\in \beauty_{v\otimes v}$. In particular we know from Theorem \ref{theorem:synthesis} that $D_S$ is bounded from $\ell^p_m(\Lambda;L^2)$ to $M^p_m(\Rd)$, and that $C_{S\frameop^{-1}_S}$ is bounded from $M^p_m(\Rd)$ to $\ell^p_m(\Lambda;L^2)$. Hence $D_SC_{S\frameop_S^{-1}}$ is bounded on $M^p_m(\Rd)$. If $p<\infty$, then the expansions in the statement converge unconditionally by Theorem \ref{theorem:synthesis}. We know that $D_SC_{S\frameop_S^{-1}}$ is the identity operator on $L^2(\Rd)$ from Section \ref{sec:dual}, and as $M^1_v(\Rd)\subset L^2(\Rd)$ is dense in $M^p_m(\Rd)$ by Proposition \ref{prop:modulationspaces} it follows that $D_SC_{\frameop_S^{-1}}$ is the identity operator on $M^p_m(\Rd)$, so the expansions converge to $\psi$. 
	
	For $p=\infty$ the last part of the argument must be slightly modified: $M^1_v(\Rd)$ is only weak*-dense in $M^\infty_m(\Rd)$, so to conclude that $D_SC_{S\frameop_S^{-1}}$ is the identity operator on $M^\infty_m(\Rd)$ we need to use that $D_SC_{S\frameop_S^{-1}}$ is weak*-to-weak*-continuous on $M^\infty_m(\Rd)$ by Corollary \ref{cor:weakstar}. 
\end{proof}

We are now ready to prove one of our main results, namely that Gabor g-frames generated by $S\in \beauty_{v\otimes v}$ define equivalent norms for modulation spaces. By picking $S$ as in Examples \ref{example:multiwindow} and \ref{example:locops}, we recover results for Gabor frames \cite{Feichtinger:1989,Grochenig:2001,Feichtinger:1997} and localization operators \cite{Dorfler:2006,Dorfler:2011}.

\begin{corollary} \label{cor:equivnorms}
Assume that $S\in \beauty_{v\otimes v}$ generates a Gabor g-frame. There exist constants $C,D$ depending on $v$ and $\Lambda$ such that for any $1\leq p \leq \infty$ and $v$-moderate weight $m$ we have
\begin{equation*}
  C \|\psi\|_{M^p_m}\leq  \left( \sum_{\lambda \in \Lambda} \|\alpha_\lambda(S)\|_{L^2}^pm(\lambda)^p\right)^{1/p} \leq  D \|\psi\|_{M^p_m},
\end{equation*}
and $\psi \in M^\infty_{1/v}(\Rd)$ belongs to $M^p_m(\Rd)$ if and only if $$\sum_{\lambda \in \Lambda} \|\alpha_\lambda(S)\|_{L^2}^pm(\lambda)^p<\infty.$$ For $p=\infty$ the sum is replaced by a supremum in the usual way.
\end{corollary}
\begin{proof}
	 By Theorem \ref{theorem:upperbound} we have
	\begin{equation*}
  \left( \sum_{\lambda \in \Lambda} \|\alpha_\lambda(S)\|_{L^2}^pm(\lambda)^p\right)^{1/p} = \|C_S\psi\|_{\ell^p_m(\Lambda;L^2)} \lesssim \|\psi\|_{M^p_m}
\end{equation*}
 as $C_S$ is bounded. On the other hand, Corollary \ref{cor:gaborexpansions} says that
 \begin{equation*}
  \|\psi\|_{M^p_m} =  \|D_{S\frameop_S^{-1}} C_S\psi\|_{M^p_m}  \lesssim \|C_S\psi\|_{\ell^p_m(\Lambda;L^2)}=\left( \sum_{\lambda \in \Lambda} \|\alpha_\lambda(S)\|_{L^2}^pm(\lambda)^p\right)^{1/p},
\end{equation*}
where we have used that $D_{S\frameop_S^{-1}}:\ell^p_m(\Lambda;L^2)\to M^p_m(\Rd)$ is bounded by Proposition \ref{prop:dualgframeM} and Theorem \ref{theorem:synthesis}. 

Finally, if $\sum_{\lambda \in \Lambda} \|\alpha_\lambda(S)\|_{L^2}^pm(\lambda)^p<\infty$, then $C_S(\psi)\in \ell^p_m(\Lambda;L^2)$. As $D_{S\frameop_S^{-1}}$ is bounded  $\ell^p_m(\Lambda;L^2)\to M^p_m(\Rd)$, its follows from $\psi = D_{S\frameop_S^{-1}}C_S \psi$ that $\psi\in M^p_m(\Rd)$.
\end{proof}

\begin{remark}

	In this section we have assumed $S\in \beauty_{v\otimes v}$, but the result also holds for operators $S\in \tco$ that can be written $$S=\sum_{n\in \N} \phi_n^{(1)}\otimes \phi_n^{(2)}$$ where $\sum_{n\in \N}  \|\phi_n^{(2)}\|_{M^1_v}<\infty$ and $\{\phi_n^{(1)}\}_{n\in \N}$ is orthonormal in $L^2(\Rd)$. The proofs of Theorems \ref{theorem:upperbound} and \ref{theorem:synthesis} still work, with upper bound $\sum_{n\in \N}  \|\phi_n^{(2)}\|_{M^1_v}$ for the operator norms of $C_S$ and $D_S$ (in the original proofs we use $\|\phi_n^{(1)}\|_{L^2}\lesssim \|\phi_n^{(1)}\|_{M^1_v}$, using $\|\phi_n^{(1)}\|_{L^2}=1$ instead leads to this modified result).  Since $S^*S=\sum_{n\in \N} \phi_n^{(2)}\otimes \phi_n^{(2)}\in \beauty$, we can still use the Janssen representation to get that $\frameop_S^{-1}$ is bounded on $M^1_v(\Rd)$, and the proof of Proposition \ref{prop:dualgframeM} shows that $$S\frameop_S^{-1}=\sum_{n\in \N}  \phi_n^{(1)}\otimes \frameop_S^{-1}\phi_n^{(2)},$$ hence $S\frameop_S^{-1}$ is of the same form. The proofs of the corollaries above still work without change. In particular, this shows that our treatment of multi-window Gabor frames in Example \ref{example:multiwindow} is compatible with the theory of this section. 
\end{remark}

\subsection{Alternative characterization of Gabor g-frames and multi-window Gabor frames of eigenfunctions}  \label{sec:multiwindow}  
The norm equivalences in Corollary \ref{cor:equivnorms} were proved for localization operators in \cite{Dorfler:2006,Dorfler:2011}. This section is mainly a reinterpretation and slight extension of the results in \cite{Dorfler:2011} in terms of Gabor g-frames -- the main result is Theorem \ref{theorem:alternativecharacterization}, which shows that a surprising characterization of Gabor frames from \cite{Grochenig:2007} holds for Gabor g-frames.
We first need to understand the singular value decomposition of operators in $\beauty_{v\otimes v}$. The following is due to \cite{Dorfler:2011} when $S$ is a localization operator, and our proof is a slight modification of their proof to allow general $S\in \beauty_{v\otimes v}$. 
\begin{lemma} \label{lem:svd}
	Assume that $S\in \beauty_{v\otimes v}$. There exist $N_0\in \mathbb{N}\cup \{\infty\}$, orthonormal systems $\left\{ \xi_n \right\}_{n=1}^{N_0},$ $\left\{ \varphi_n \right\}_{n=1}^{N_0}$ in $L^2(\Rd)$ and a sequence $\left\{ s_n \right\}_{n=1}^{N_0}\in \ell^1$ of positive numbers   with 
	\begin{equation} \label{eq:svd}
  S=\sum_{n=1}^{N_0} s_n \xi_n \otimes \varphi_n
\end{equation}
as an operator on $L^2(\Rd)$. Furthermore, $\varphi_n,\xi_n\in M^1_{v}(\Rd)$, and for $\lambda \in \Lambda$ the expansion
\begin{equation} \label{eq:svddual}
  \alpha_\lambda(S)\psi = \sum_{n=1}^{N_0} s_n \inner{\psi}{\pi(\lambda)\varphi_n}_{M^\infty_{1/v},M^1_{v}} \pi(\lambda)\xi_n
\end{equation}holds even for $\psi \in M^\infty_{1/v}(\Rd)$, with convergence of the sum in $L^2(\Rd)$.
\end{lemma}

\begin{proof}
	The existence of $\left\{ \xi_n \right\}_{n=1}^{N_0}$, $\left\{ \varphi_n \right\}_{n=1}^{N_0}$ and $\left\{ s_n \right\}_{n=1}^{N_0}$ with these properties is the singular value decomposition from Section \ref{sec:tchs}.
	To see that $\xi_n\in M^1_{v}(\Rd)$, note that Proposition \ref{prop:innerkernel} says that $S:M^\infty_{1/v}(\Rd)\to M^1_{v}(\Rd)$. From \eqref{eq:svd} one obtains that $M^1_{v}(\Rd) \ni S\varphi_n=s_n \xi_n$, which forces $\xi_n \in M^1_{v}(\Rd)$ when $s_n\neq0$. Since $S^*\in \beauty_{v\otimes v}$ by Proposition \ref{prop:innerkernel}, the same argument as above gives that $M^1_{v}(\Rd) \ni S^*\xi_n=s_n \varphi_n$, so $\varphi_n\in M^1_{v}(\Rd)$.
	
	We prove the expansion \eqref{eq:svddual} for $\lambda= 0$, without loss of generality. If $\psi\in M^\infty_{1/v}(\Rd)$, we know from Proposition \ref{prop:innerkernel} that $S\psi \in M^1_{v}(\Rd)\subset L^2(\Rd)$. Thus we may find $\gamma\in L^2(\Rd)$ such that 	\begin{equation*}
  S\psi=\sum_{n=1}^{N_0} \inner{S\psi}{\xi_n}_{L^2} \xi_n + \gamma,
\end{equation*}
where $\gamma \perp \xi_n$ for each $n\leq N_0$. The sum converges in $L^2(\Rd)$ as $S\psi \in L^2(\Rd)$ and the set $\left\{ \xi_n \right\}_{n=1}^{N_0}$ is orthonormal. By Proposition \ref{prop:innerkernel}, we get
\begin{equation*}
  \inner{S\psi}{\xi_n}_{L^2}=\inner{S\psi}{\xi_n}_{M^\infty_{1/v},M^1_{v}}=\inner{\psi}{S^*\xi_n}_{M^\infty_{1/v},M^1_{v}}=s_n\inner{\psi}{ \varphi_n}_{M^\infty_{1/v},M^1_{v}},
\end{equation*}
 hence we have shown 
\begin{equation*}
  S\psi = \sum_{n=1}^{N_0} s_n\inner{\psi}{ \varphi_n}_{M^\infty_{1/v},M^1_{v}} \xi_n + \gamma,
\end{equation*}
and it simply remains to show that $\gamma=0$. Note that $\|\gamma\|_{L^2}^2=\inner{S\psi}{\gamma}_{L^2}$. As is shown in the proof of \cite[Cor. 7]{Dorfler:2011}, we can pick a sequence $\left\{ \psi_i \right\}_{i\in \N}$ in $L^2(\Rd)$ that converges to $\psi$ in the weak* topology of $M^\infty_{1/v}(\Rd)$. Then $\inner{S\psi_i}{\gamma}_{L^2}=0$, since \eqref{eq:svd} shows that $S\psi_i$ can be expanded in terms of the $\xi_n$, and $\gamma$ is orthogonal to each $\xi_n$. However, $S$ maps  weak*-convergent sequences in $M^\infty_{1/v}(\Rd)$ into norm convergent sequences in $M^1_{v}(\Rd)$, hence $S\psi_i\to S\psi$ in $L^2(\Rd)$ and
\begin{equation*}
  0=\inner{S\psi_i}{\gamma}_{L^2}\to \inner{S\psi}{\gamma}_{L^2}=\|\gamma\|_{L^2},
\end{equation*}
which completes the proof. 
\end{proof}

\begin{remark} 
The singular value decomposition in Lemma \ref{lem:svd} should be compared to the decomposition from Proposition \ref{prop:innerkernel}. There is one clear advantage to the singular value decomposition
$
  S=\sum_{n=1}^{N_0} s_n \xi_n \otimes \varphi_n,
$
namely that the systems $\left\{ \xi_n \right\}_{n=1}^{N_0}$ and $\left\{ \varphi_n \right\}_{n=1}^{N_0}$ are orthonormal. The disadvantage of the singular value decomposition is that, unlike the decomposition from Proposition \ref{prop:innerkernel}, it does not necessarily converge absolutely in the norm of $\beauty_{v\otimes v}$. In other words, we cannot guarantee that $\sum_{n=1}^{N_0} s_n \|\xi_n\|_{M^1_v}\|\varphi_n\|_{M^1_v}<\infty$. This was recently proved in \cite{Balan:2018}, solving a problem first posed by Hans Feichtinger. 
\end{remark}

 The following result is used in the proof of \cite[Lem. 9]{Dorfler:2011} for localization operators $S$. Our proof is a slight modification of the proof in \cite{Dorfler:2011} to allow general $S\in \beauty$.

\begin{proposition} \label{prop:MWframe}
	Assume that $S\in \beauty$ and let $\{\varphi_n\}_{n=1}^{N_0}$ be as in Lemma \ref{lem:svd}. If $C_S:M^\infty(\Rd)\to \ell^{\infty}(\Lambda;L^2)$ is injective,  then there is some $N\leq N_0$ such that $\{\varphi_n\}_{n=1}^N\subset M^1_v(\Rd)$ generate a multi-window Gabor frame.
\end{proposition}

\begin{proof}
   Assume that, for any $N\leq N_0$, $\{\varphi_n\}_{n=1}^N$ does not generate a multi-window Gabor frame. Consider the set
    \begin{equation*}
    \mathcal{W}_N=\{\eta \in M^\infty(\Rd): \inner{\eta}{\pi(\lambda) \varphi_n}_{M^\infty,M^1}=0 \text{ for any } \lambda \in \Lambda, n=1,...,N. \}
  \end{equation*}
  By \cite[Lem. 3]{Dorfler:2011}, $\mathcal{W}_N$ is a non-trivial subspace of $M^\infty(\Rd)$, and by  \cite[Lem. 10]{Dorfler:2011}, the intersection of all $\mathcal{W}_N$ for $N\leq N_0$ is a non-trivial subspace of $M^\infty(\Rd)$. Let $\eta$ be a non-zero element from this intersection, meaning that 
  \begin{equation*}
    \inner{\eta}{\pi(\lambda)\varphi_n }_{M^\infty,M^1}=0 \text{ for any } \lambda \in \Lambda, n\leq N_0.
  \end{equation*}
  By \eqref{eq:svddual}, we have that $\alpha_\lambda (S) \eta=\sum_{n=1}^{N_0} s_n \inner{\eta}{\pi(\lambda)\varphi_n}_{M^\infty,M^1} \pi(\lambda) \xi_n=0$ for any $\lambda \in \Lambda$, since $\inner{\eta}{\pi(\lambda)\varphi_n }_{M^\infty,M^1}=0$ for $n\leq N_0$. This means that $C_S\eta=0$.  Thus $\eta=0$, which contradicts our assumption. Hence there is an $N\leq N_0$ such that $\{\varphi_n\}_{n=1}^N$ generates a multi-window Gabor frame.   \end{proof}

For Gabor frames, the following theorem is one of the main results of \cite{Grochenig:2007}, and the reader who has consulted the proof of Proposition \ref{prop:MWframe} may have noted that the Gabor frame-version of the statement is the key to the proof of that proposition. 
\begin{theorem} \label{theorem:alternativecharacterization}
	Let $S\in \beauty.$ $S$ generates a Gabor g-frame if and only if $C_S:M^\infty(\Rd)\to \ell^{\infty}(\Lambda;L^2)$ is injective.
\end{theorem}
\begin{proof}
	If $S$ generates a Gabor g-frame, $D_{S\frameop^{-1}_S}C_S$ is the identity operator on $M^\infty(\Rd)$ by Corollary \ref{cor:gaborexpansions}, hence $C_S$ is injective. Then assume that $C_S$ is injective. Since $S\in \beauty$, Corollary \ref{cor:beautyupperbound} says that the upper g-frame bound in \eqref{eq:gaborgframes} is satisfied. For the lower bound, Lemma \ref{lem:svd} and Proposition \ref{prop:MWframe} say that $S=\sum_{n=1}^{N_0} s_n \xi_n \otimes \varphi_n$, where $\{\varphi_n\}_{n=1}^N$ generate a multi-window Gabor frame for some $N\leq N_0$. Note that $$\|\alpha_{\lambda}(S)\psi\|_{L^2}^2=\inner{\alpha_\lambda(S)\psi}{\alpha_\lambda(S)\psi}_{L^2}=\inner{\alpha_\lambda(S^*S)\psi}{\psi}_{L^2}.$$
	By the decomposition $S=\sum_{n=1}^{N_0} s_n \xi_n \otimes \varphi_n$ and orthonormality of $\{\xi_n\}_{n=1}^{N_0}$, we get
	\begin{equation*}
  \alpha_\lambda(S^*S)=\sum_{n=1}^{N_0}s_{n}^2 \pi(\lambda)\varphi_n \otimes \pi(\lambda)\varphi_n,
\end{equation*}
hence 
 \begin{align*}
  \sum_{\lambda \in \Lambda} \|\alpha_\lambda(S)\psi\|_{L^2}^2 &= \sum_{\lambda \in \Lambda} \inner{\sum_{n=1}^{N_0}s_{n}^2 V_{\varphi_n}\psi(\lambda)\pi(\lambda)\varphi_n}{\psi}_{L^2} \\
  &= \sum_{\lambda \in \Lambda} \sum_{n=1}^{N_0}s_{n}^2 |V_{\varphi_n}\psi(\lambda)|^2 \\
  &\geq \sum_{\lambda \in \Lambda} \sum_{n=1}^N s_{n}^2 |V_{\varphi_n}\psi(\lambda)|^2 \\
  &\gtrsim  \|\psi\|_2^2,
\end{align*}
since  $\{\varphi_n\}_{n=1}^N$ generate a multi-window Gabor frame and $s_n>0$ for $n\leq N$.
	\end{proof}

\subsection{Localization operators and time-frequency partitions} \label{sec:locops}
In \cite{Dorfler:2006,Dorfler:2011}, the methods from the previous section were used to prove the norm equivalence in Corollary \ref{cor:equivnorms} for the localization operators $A_h^\varphi$ in Example \ref{example:locops}, i.e. assuming $0\neq \varphi\in M^1_v(\Rd)$ and $h\in L^1_v(\Rdd)$ a non-negative function satisfying $$A'\leq \sum_{\lambda \in \Lambda} h(z-\lambda)\leq B' \quad \text{ for all } z\in \Rdd$$ for some $A',B'>0$. Their proof consists of applying Proposition \ref{prop:MWframe} to obtain multi-window Gabor frames of eigenfunctions of localization operators to reduce the statement to the fact that multi-window Gabor frames give equivalent norms for $M^p_m(\Rd)$. Since inserting $p=2$ and $m\equiv 1$ in Corollary \ref{cor:equivnorms} gives the Gabor g-frame inequality, this means in particular that these localization operators generate Gabor g-frames.  

\begin{remark}
	Obtaining multi-window Gabor frames consisting of eigenfunctions of localization operators is itself an interesting result. D\"orfler and Romero \cite{Dorfler:2014} use techniques from \cite{Romero:2012} to obtain frames consisting of eigenfunctions of localization operators in more general settings. If $S=A_{\chi_\Omega}^\varphi$, then $\alpha_\lambda (S) = A_{\chi_{\Omega+\lambda}}^\varphi$. In this sense, applying $\alpha_\lambda$ corresponds to covering $\Rdd$ by shifts of $\Omega$, and the results of \cite{Dorfler:2014} consider much more general coverings of $\Rdd$ when $S$ is a localization operator.
\end{remark}

In order to apply the machinery of Section \ref{sec:eqnorms} to localization operators $A_h^\varphi$, we need to show that $A_h^\varphi\in \beauty_{v\otimes v}$. The next proposition shows that this is true if we assume the stronger condition $h\in L^1_{v^2}(\Rdd)$.
\begin{proposition} \label{prop:locops}
	Let $\varphi \in M^1_v(\Rd)$ and  $h\in L^1_{v^2}(\Rdd)$. Then $A_h^\varphi\in \beauty_{v\otimes v}$. 
\end{proposition}
\begin{proof} 
	It is a straightforward calculation to check that the kernel of $A^\varphi_{h}$ is
\begin{equation*}
  \kernel_{A_h^\varphi}(x,\omega)=\int_{\Rdd} h(t,\xi) (M_\xi T_t \varphi)(x) (M_{-\xi} T_t \overline{\varphi})(\omega) \ dt d\xi.
\end{equation*}
For each $t$ and $\xi$, the function $$M_\xi T_t \varphi(x) M_{-\xi} T_t \overline{\varphi}(\omega)=\left(\pi(t,\xi)\varphi \otimes  \pi(t,-\xi) \overline{\varphi}\right)(x,\omega)$$
belongs to $M^1_{v\otimes v}$ by \eqref{eq:rankoneprojtensor} and part (e) of Proposition \ref{prop:modulationspaces}, with 
\begin{equation*}
 \left\| \pi(t,\xi)\varphi \otimes  \pi(t,-\xi) \overline{\varphi}\right\|_{M^1_{v\otimes v}} =  \|\pi(t,\xi)\varphi\|_{M^1_v}\| \pi(t,-\xi) \overline{\varphi}\|_{M^1_v} \leq v(t,\xi)^2 \|\varphi\|^2_{M^1_v},
\end{equation*}
where we have used that $v$ is symmetric in each coordinate. Hence 
\begin{equation*}
  \int_{\Rdd}\left\| h(t,\xi) \left(\pi(t,\xi)\varphi \otimes  \pi(t,-\xi) \overline{\varphi}\right)\right)\|_{M^1_{v\otimes v}} \ dt d\xi 
\end{equation*}
is bounded from above by 
\begin{equation*}
 \|\varphi\|^2_{M^1_v}  \int_{\Rdd} |h (t,\xi)| v(t,\xi)^2 \ dt d\xi,
\end{equation*}
 and this last integral converges by assumption. It follows that the integral $$\int_{\Rdd} h(t,\xi) \left(\pi(t,\xi)\varphi \otimes  \pi(t,-\xi) \overline{\varphi}\right) \ dt d\xi$$ is a convergent Bochner integral in $M^1_{v\otimes v}(\Rdd)$, thus $\kernel_{A_h^\varphi}\in M^1_{v\otimes v}(\Rdd)$.
\end{proof}
The setting $S=A^\varphi_h$ allows us to interpret many objects and results for Gabor g-frames in a natural way, in particular when $h=\chi_\Omega\in L^1_{v^2}(\Rdd)$ is the characteristic function of some compact $\Omega \subset \Rdd$. Since one has the well-known inversion formula 
\begin{equation*}
  \psi=\int_{\Rdd} V_{\varphi}\psi(z) \pi(z)\varphi \ dz \quad \text{ whenever } \|\varphi\|_{L^2}=1,
\end{equation*}
one interprets 
\begin{equation*}
 A^\varphi_{\chi_\Omega} \psi=\int_{\Omega} V_{\varphi}\psi(z) \pi(z)\varphi \ dz
\end{equation*}
as the part of $\psi$ that "lives in $\Omega$ in the time-frequency plane" \cite{Cordero:2003}. For brevity, we call $A_{\chi_\Omega}^\varphi \psi$ the $\Omega$-component of $\psi$.
 Since $\alpha_\lambda(A^\varphi_{\chi_\Omega})=A^\varphi_{T_\lambda (\chi_\Omega)}$, we see that $\alpha_\lambda(A^\varphi_{\chi_\Omega})\psi$ is the $\lambda+\Omega$-component of $\psi$, where $\lambda+\Omega=\{\lambda + z : z\in \Omega\}$. The corresponding analysis operator $$C_{A_{\chi_\Omega}^\varphi}(\psi)=\left\{A^\varphi_{T_\lambda (\chi_\Omega)}\psi\right\}_{\lambda \in \Lambda}$$ therefore analyzes $\psi$ by considering its $\lambda+\Omega$-components as $\lambda$ varies over $\Lambda$. 
 
 When $A^\varphi_{\chi_\Omega}$ actually generates a Gabor g-frame, Corollary \ref{cor:equivnorms} says that summability conditions on the $L^2$-norm of the $\lambda+\Omega$-components of $\psi$ precisely captures the modulation space norms of $\psi$, as first proved by \cite{Dorfler:2006,Dorfler:2011}. Furthermore, Corollary \ref{cor:gaborexpansions} shows us \textit{how} $\psi$ may be reconstructed from its $\lambda+\Omega$-components. By that result, there exists some $R:=A_{\chi_\Omega}^\varphi \frameop_{A^\varphi_{\chi_\Omega}}^{-1}\in \beauty_{v\otimes v}$ such that 
 \begin{equation} \label{eq:inversion}
  \psi=\sum_{\lambda \in \Lambda} \alpha_{\lambda}(R) \left(A^\varphi_{T_\lambda (\chi_\Omega)}\psi\right), 
\end{equation}
with unconditional convergence in whatever modulation space $M^p_m(\Rd)$, $p<\infty$, that $\psi$ belongs to.
By Remark \ref{rem:cohen}, there is also a Cohen's class distribution associated with $A^\varphi_{\chi_\Omega}$, namely
\begin{equation*}
  Q_{\left(A^\varphi_{\chi_\Omega}\right)^2}(\psi)(z)=\|A_{T_z (\chi_\Omega)}^\varphi \psi \|_{L^2}^2.
\end{equation*}
This Cohen's class distributions has an obvious interpretation: $\|A_{T_z (\chi_\Omega)}^\varphi \psi \|_{L^2}^2$ measures the size of the $z+\Omega$-component of $\psi$. By \eqref{eq:continuouscohen} one has the equality

\begin{equation*}
  \int_{\Rdd} \|A_{T_z (\chi_\Omega)}^\varphi \psi \|_{L^2}^2 \ dz = \|A_{\chi_\Omega}^\varphi\|_{\HS}^2 \|\psi\|_{L^2}^2.
\end{equation*}
This is a continuous version of the Gabor g-frame inequality \eqref{eq:gaborgframes} for localization operators, in the same way that Moyal's identity is the continuous version of the Gabor frame inequalities. 

 It should be remarked that one usually associates a different Cohen's class distribution (independently of $\Omega$) with localization operators $A_{\chi_\Omega}^\varphi$, namely the spectrogram $|V_\varphi \psi(z)|^2$ \cite[Example 8.1]{Luef:2018b}. 
\begin{remark}
\begin{enumerate}[(a)]
	\item Let us clarify the relation between our results and those of \cite{Dorfler:2011}. As mentioned, Corollary \ref{cor:equivnorms} was proved in \cite{Dorfler:2011} for localization operators $A_h^\varphi$ satisfying the conditions in Example \ref{example:locops}, without the notion of Gabor g-frames. The statements in Section \ref{sec:multiwindow} may all be deduced from proofs in \cite{Dorfler:2011}, and we have merely reinterpreted them as natural statements about Gabor g-frames. Proposition \ref{prop:locops} says that if we assume $h\in L^1_{v^2}(\Rdd)$ -- a stronger condition than $h\in L^1_v(\Rdd)$ as assumed in \cite{Dorfler:2011} -- then $A_h^\varphi$ satisfies the assumptions for the other results in Section \ref{sec:eqnorms}. In particular, we get the inversion formula \eqref{eq:inversion}.
	\item The discussion above generalizes without change to other Gabor g-frames\newline $\{\alpha_\lambda S\}_{\lambda \in \Lambda}$, but the natural interpretation of $\|\alpha_\lambda(S)\|_{L^2}^2$ above does not necessarily hold when $S$ is not a localization operator.
\end{enumerate}
\end{remark}
\section{Singular value decomposition and multi-window Gabor frames} \label{sec:singval}
From the very first paper published on g-frames \cite{Sun:2006}, it has been known that g-frames correspond to ordinary frames when a basis is chosen for the Hilbert spaces involved: if $\{A_i\}_{i\in I}\subset \bo$ and $\{\xi_n\}_{n\in \N}$ is an orthonormal basis of $L^2(\Rd)$, then $\{A_i\}_{i\in I}$ is g-frame if and only if $\{A_i^*\xi_n\}_{i\in I,n\in \N}$ is a frame for $L^2(\Rd)$\cite[Thm. 3.1]{Sun:2006}. Gabor g-frames must therefore be related to frames in $L^2(\Rd)$, and we will now make this connection explicit. By the singular value decomposition, any $S\in \HS$ may be expanded as 
\begin{equation*}
  S=\sum_{n\in \N} \xi_n \otimes \varphi_n,
\end{equation*}
where $\{\xi_n\}_{n\in \N}$ is an orthonormal basis for $L^2(\Rd)$ and $\sum_{n\in \N} \|\varphi_n\|_{L^2}^2<\infty$.
 For $\psi \in L^2(\Rd)$ we find using \eqref{eq:shiftrankone} that
\begin{align*}
  \|\alpha_\lambda(S)\psi\|_{L^2}^2&=\inner{\sum_{n\in \N}  V_{\varphi_n}\psi(\lambda)\pi(\lambda)\xi_n}{\sum_{m\in \N}  V_{\varphi_m}\psi(\lambda)\pi(\lambda)\xi_m}_{L^2}\\
  &= \sum_{m,n\in \N} V_{\varphi_n}\psi(\lambda)\overline{V_{\varphi_m}\psi(\lambda)}\inner{\pi(\lambda)\xi_n}{\pi(\lambda)\xi_m}_{L^2} \\
  &= \sum_{n\in \N}  |V_{\varphi_n}\psi(\lambda)|^2.
  \end{align*}
By comparing this with the definition \eqref{eq:gaborgframes} of a Gabor g-frame, we see that $S$ generates a Gabor g-frame if and only if there exist $A,B>0$ such that
\begin{equation*}
 	A\|\psi\|_{L^2}^2 \leq \sum_{\lambda \in \Lambda} \sum_{n\in \N} |V_{\varphi_n}\psi(\lambda)|^2 \leq B \|\psi\|_{L^2}^2\quad  \text{ for any } \psi \in L^2(\Rd),
\end{equation*}
 in other words, if and only if the functions $\left\{ \varphi_n \right\}_{n\in \N}$ generate a multi-window Gabor frame with countably many windows. Combining this with Proposition \ref{prop:MWframe}, we obtain the following result on multi-window Gabor frames with countably many generators. 

 \begin{theorem}
 	Assume that $\left\{ \varphi_n \right\}_{n\in \N}\subset M^1(\Rd)$ such that $\sum_{n\in \N}\|\varphi_n\|_{M^1}< \infty$. If $\left\{ \varphi_n \right\}_{n\in \N}$ generates a multi-window Gabor frame for $L^2(\Rd)$, i.e. there exist $A,B>0$ such that
\begin{equation} \label{eq:countMW}
 	A\|\psi\|_{L^2}^2 \leq \sum_{\lambda \in \Lambda} \sum_{n\in \N} |V_{\varphi_n}\psi(\lambda)|^2 \leq B \|\psi\|_{L^2}^2 \text{\ \ \ \  for any } \psi \in L^2(\Rd),
 	\end{equation}
 	then there exists $N\in \N$ such that $\left\{ \varphi_n \right\}_{n=1}^N$ generates a multi-window Gabor frame for $L^2(\Rd)$.
 \end{theorem}
 \begin{proof}
 	Let  $\left\{ \xi_n \right\}_{n\in \N}$ be an orthonormal basis for $L^2(\Rd)$ such that $\|\xi_n\|_{M^1}\leq C$ for some $C>0$ -- for instance a Wilson basis \cite[Prop. 12.3.8]{Grochenig:2001}. Then let $$S=\sum_{n\in \N} \xi_n \otimes \varphi_n.$$ By our assumptions $\sum_{n\in \N}\|\varphi_n\|_{M^1}< \infty$ and $\|\xi_n\|_{M^1}\leq C$, this sum converges absolutely in $\beauty$. Hence $S\in \beauty$. By the arguments preceding this theorem, \eqref{eq:countMW} ensures that $S$ generates a Gabor g-frame. Hence Theorem \ref{theorem:alternativecharacterization} and Proposition \ref{prop:MWframe} give\footnote{Proposition \ref{prop:MWframe} assumes that $\varphi_n$ come from the singular value decomposition, but this is not used in the proof besides using Lemma \ref{lem:svd} to ensure that the decomposition into rank-one operators converges.} the existence of $N\in \N$ such that $\left\{ \varphi_n \right\}_{n=1}^N$ generates a multi-window Gabor frame for $L^2(\Rd)$.
 \end{proof}
 \begin{remark}
 	The fact that Gabor g-frames correspond to multi-window Gabor frames with countably many generators, suggests that the duality theory of Gabor g-frames (in the sense of Ron-Shen duality, see \cite{Grochenig:2001}) is covered by the approach in \cite{Jakobsen:2018a}, where multi-window Gabor frames with countably many generators are considered.
 \end{remark}

\subsection*{Acknowledgements}
We wish to thank the author's supervisor Franz Luef for his insightful feedback on various drafts of the paper. We also thank Mads S. Jakobsen for helpful discussions. Finally, we thank Hans Feichtinger for his suggestion that the methods developed in the author's master's thesis might be applied to the theory of periodic operators.



\normalsize


\begin{thebibliography}{[HD82]}




\normalsize
\baselineskip=17pt



\bibitem{Abdollahpour:2008}
  M.~R. Abdollahpour and M.~H. Faroughi.
  \newblock Continuous {$G$}-frames in {H}ilbert spaces.
  \newblock {\em Southeast Asian Bull. Math.}, 32(1):1--19, 2008.
       
\bibitem{Balan:2018}
R.~Balan, K.~A. Okoudjou, and A.~Poria.
\newblock On a problem by {H}ans {F}eichtinger.
\newblock {\em Oper. Matrices}, 12(3):881--891, 2018.

\bibitem{Balazs:2008}
{P}. {B}alazs.
\newblock {H}ilbert-{S}chmidt operators and frames - classification, best
  approximation by multipliers and algoritheorems.
\newblock {\em {I}nt. {J}. {W}avelets {M}ultiresolut. {I}nf. {P}rocess.},
  6(2):315 -- 330, 2008.

\bibitem{Balazs:2014inverse} 
P.~{B}alazs and D.~T. {S}toeva. 
\newblock {R}epresentation of the inverse of a frame multiplier. 
\newblock {\em J. Math. Anal. Appl.}, 422(2):981 -- 994, 2015. 
  

\bibitem{Balazs:2019} 
P.~{B}alazs, K.~{G}r{\"o}chenig, and M.~{S}peckbacher. 
\newblock {K}ernel theorems in coorbit theory. 
\newblock {\em Trans. Amer. Math. Soc.}, 6(11):346--364, 2019. 

\bibitem{Bello:1963} 
P.~A. {B}ello. 
\newblock {C}haracterization of {R}andomly {T}ime-{V}ariant {L}inear 
{C}hannels. 
\newblock {\em IEEE Trans. Comm.}, 11:360--393, 1963. 

\bibitem{Bochner:1942} 
S.~{B}ochner and R.~{P}hillips. 
\newblock {A}bsolutely convergent {F}ourier expansions for non-commutative 
normed rings. 
\newblock {\em Ann. of Math. (2)}, 43:409--418, 1942. 

\bibitem{Boggiatto:2004} 
P.~{B}oggiatto, E.~{C}ordero, and K.~{G}r{\"o}chenig. 
\newblock {G}eneralized anti-{W}ick operators with symbols in distributional 
{S}obolev spaces. 
\newblock {\em {I}ntegr. {E}qu. {O}per. {T}heory}, 48(4):427--442, 2004. 

\bibitem{Busch:2016}
{P}. {B}usch, {P}. {L}ahti, {J}.-{P}. {P}ellonp\"{a}\"{a}, and {K}. {Y}linen.
\newblock {\em {Q}uantum {M}easurement}.
\newblock {T}heoretical and {M}athematical {P}hysics. {S}pringer, 2016.

%
%
%
%
%

\bibitem{Christensen:2016} 
O.~{C}hristensen. 
\newblock {\em {A}n {I}ntroduction to {F}rames and {R}iesz {B}ases}. 
\newblock {A}pplied and {N}umerical {H}armonic {A}nalysis. {B}irkh{\"a}user 
{B}asel, {S}econd edition, 2016. 

\bibitem{Cohen:1966}
L. Cohen.
\newblock Generalized phase-space distribution functions.
\newblock {\em J. Math. Phys.}, 7(5):781--786, 1966.

\bibitem{Cordero:2003}
{E}. {C}ordero and {K}. {G}r{\"o}chenig.
\newblock {T}ime-frequency analysis of localization operators.
\newblock {\em {J}. {F}unct. {A}nal.}, 205(1):107--131, 2003.

\bibitem{Cordero:2008} 
E.~{C}ordero, H.~G. {F}eichtinger, and F.~{L}uef. 
\newblock {B}anach {G}elfand triples for {G}abor analysis. 
\newblock In {\em {P}seudo-differential {O}perators}, volume 1949 of {\em 
{L}ecture {N}otes in {M}athematics}, pages 1--33. {S}pringer, {B}erlin, 2008. 

\bibitem{Daubechies:1986} 
I.~{D}aubechies, A.~{G}rossmann, and Y.~{M}eyer. 
\newblock {P}ainless nonorthogonal expansions. 
\newblock {\em J. Math. Phys.}, 27(5):1271--1283, 1986. 

\bibitem{Daubechies:1988} 
I.~{D}aubechies. 
\newblock {T}ime-frequency localization operators: a geometric phase space 
approach. 
\newblock {\em IEEE Trans. Inform. Theory}, 34(4):605--612, 1988. 

\bibitem{Daubechies:1995} 
I.~{D}aubechies, H.~J. {L}andau, and Z.~{L}andau. 
\newblock {G}abor time-frequency lattices and the {W}exler-{R}az identity. 
\newblock {\em J. Fourier Anal. Appl.}, 1(4):437--478, 1995.

\bibitem{deLeeuw:1973} 
K.~de {L}eeuw. 
\newblock {F}ourier series of operators and an extension of the {{F}}. and 
{{M}}. {{R}}iesz theorem. 
\newblock {\em Bull. Amer. Math. Soc.}, 79:342--344, 1973. 

\bibitem{deLeeuw:1975} 
K.~de {L}eeuw. 
\newblock {A}n harmonic analysis for operators. {I}: {F}ormal properties. 
\newblock {\em {I}ll. {J}. {M}ath.}, 19:593--606, 1975. 

\bibitem{deLeeuw:1977} 
K.~de {L}eeuw. 
\newblock {A}n harmonic analysis for operators. {I}{I}: {O}perators on 
{H}ilbert space and analytic operators. 
\newblock {\em {I}ll. {J}. {M}ath.}, 21:164--175, 1977.

\bibitem{deGosson:2011} 
M.~de {G}osson. 
\newblock {\em {S}ymplectic {M}ethods in {H}armonic {A}nalysis and in 
{M}athematical {P}hysics}, volume~7 of {\em {P}seudo-{D}ifferential 
{O}perators. {T}heory and {A}pplications}. 
\newblock {B}irkh{\"a}user/{S}pringer {B}asel {A}{G}, {B}asel, 2011. 

\bibitem{Deitmar:2014}
A.~Deitmar and S.~Echterhoff.
\newblock {\em Principles of harmonic analysis}.
\newblock Universitext. Springer, Cham, second edition, 2014.

\bibitem{Dorfler:2006}
{M}. {D}{\"o}rfler, {H}.~{G}. {F}eichtinger, and {K}. {G}r{\"o}chenig.
\newblock {T}ime-frequency partitions for the {G}elfand triple $({S}_0,
  {L}^2,{{S}_0}')$.
\newblock {\em {M}ath. {S}cand.}, 98(1):81--96, 2006.

\bibitem{Dorfler:2011}
{M}. {D}{\"o}rfler and {K}. {G}r{\"o}chenig.
\newblock {T}ime-frequency partitions and characterizations of modulations
  spaces with localization operators.
\newblock {\em {J}. {F}unct. {A}nal.}, 260(7):1903 -- 1924, 2011.

\bibitem{Dorfler:2014}
M.~D{\"o}rfler and J.~L. Romero.
\newblock Frames adapted to a phase-space cover.
\newblock {\em Constr. Approx.}, 39(3):445--484, 2014.

\bibitem{Feichtinger:1981}
{H}.~{G}. {F}eichtinger.
\newblock {O}n a new {S}egal algebra.
\newblock {\em {M}onatsh. {M}ath.}, 92:269--289, 1981.

\bibitem{Feichtinger:1989} 
H.~G. {F}eichtinger and K.~{G}r{\"o}chenig. 
\newblock {B}anach spaces related to integrable group representations and their 
atomic decompositions, {I}. 
\newblock {\em J. Funct. Anal.}, 86(2):307--340, 1989. 


\bibitem{Feichtinger:1997}
{H}.~{G}. {F}eichtinger and {K}. {G}r{\"o}chenig.
\newblock {G}abor frames and time-frequency analysis of distributions.
\newblock {\em {J}. {F}unct. {A}nal.}, 146(2):464--495, 1997.


\bibitem{Feichtinger:1998}
{H}.~{G}. {F}eichtinger and {W}. {K}ozek.
\newblock {Q}uantization of {T}{F} lattice-invariant operators on elementary
  {L}{C}{A} groups.
\newblock In {H}ans~{G}. {F}eichtinger and {T}. {S}trohmer, editors, {\em
  {G}abor analysis and algoritheorems}, {A}ppl. {N}umer. {H}armon. {A}nal., pages
  233--266. {B}irkh{\"a}user {B}oston, 1998.

\bibitem{Feichtinger:1998a}
{H}.~{G}. {F}eichtinger and {G}. {Z}immermann.
\newblock {A} {B}anach space of test functions for {G}abor analysis.
\newblock In {H}ans~{G}. {F}eichtinger and {T}. {S}trohmer, editors, {\em
  {G}abor {A}nalysis and {A}lgoritheorems: {T}heory and {A}pplications}, {A}pplied
  and {N}umerical {H}armonic {A}nalysis, pages 123--170. {G}abor;{N}u{H}{A}{G},
  {B}irkh{\"a}user {B}oston, 1998.

\bibitem{Feichtinger:2009} 
H.~G. {F}eichtinger. 
\newblock {B}anach {G}elfand triples for applications in physics and 
engineering. 
\newblock volume 1146 of {\em {A}{I}{P} {C}onf. {P}roc.}, pages 189--228. 
{A}mer. {I}nst. {P}hys., 2009. 

\bibitem{Feichtinger:2018}
{H}.~{G}. {F}eichtinger and {M}.~{S}. {J}akobsen.
\newblock {T}he inner kernel theorem for a certain {S}egal algebra.
\newblock {\em arXiv:1806.06307}, 2018.


\bibitem{Folland:1989}
{G}.~{B}. {F}olland.
\newblock {\em {H}armonic {A}nalysis in {P}hase {S}pace}.
\newblock {P}rinceton {U}niversity {P}ress, 1989.

\bibitem{Folland:2016} 
G.~{F}olland. 
\newblock {\em {A} {C}ourse in {A}bstract {H}armonic {A}nalysis}. 
\newblock {T}extbooks in {M}athematics. {C}{R}{C} {P}ress, {B}oca {R}aton, 
{S}econd edition, 2016. 

\bibitem{Grochenig:1996}
{K}. {G}r{\"o}chenig.
\newblock {A}n uncertainty principle related to the {P}oisson summation
  formula.
\newblock {\em {S}tudia {M}ath.}, 121(1):87--104, 1996.

\bibitem{Grochenig:1998} 
K.~{G}r{\"o}chenig. 
\newblock {A}spects of {G}abor analysis on locally compact abelian groups. 
\newblock In H.~G. {F}eichtinger and T.~{S}trohmer, editors, {\em {G}abor 
{A}nalysis and {A}lgoritheorems: {T}heory and {A}pplications}, pages 211--231. 
{B}irkh{\"a}user {B}oston, {B}oston, {M}{A}, 1998. 

\bibitem{Grochenig:1999}
{K}. {G}r{\"o}chenig and {C}. {H}eil.
\newblock {M}odulation spaces and pseudodifferential operators.
\newblock {\em {I}ntegr. {E}qu. {O}per. {T}heory}, 34(4):439--457, 1999.

\bibitem{Grochenig:2001} 
K.~{G}r{\"o}chenig. 
\newblock {\em {F}oundations of {T}ime-{F}requency {A}nalysis}. 
\newblock {A}ppl. {N}umer. {H}armon. {A}nal. {B}irkh{\"a}user, {B}oston, 
{M}{A}, 2001. 

\bibitem{Grochenig:2004}
{K}. {G}r{\"o}chenig and {M}. {L}einert.
\newblock {W}iener's lemma for twisted convolution and {G}abor frames.
\newblock {\em {J}. {A}mer. {M}ath. {S}oc.}, 17:1--18, 2004.

\bibitem{Grochenig:2007} 
K.~{G}r{\"o}chenig. 
\newblock {G}abor frames without inequalities. 
\newblock {\em {I}nt. {M}ath. {R}es. {N}ot. {I}{M}{R}{N}}, (23):{A}rt. {I}{D} 
rnm111, 21, 2007. 

\bibitem{Grochenig:2007w} 
K.~{G}r{\"o}chenig. 
\newblock {W}eight functions in time-frequency analysis. 
\newblock In L.~{R}odino and et~al., editors, {\em {P}seudodifferential 
{O}perators: {P}artial {D}ifferential {E}quations and {T}ime-{F}requency 
{A}nalysis}, volume~52 of {\em {F}ields {I}nst. {C}ommun.}, pages 343--366. 
{A}mer. {M}ath. {S}oc., {P}rovidence, {R}{I}, 2007. 

\bibitem{Heil:2011} 
C.~{H}eil. 
\newblock {\em {A} {B}asis {T}heory {P}rimer. {E}xpanded ed.} 
\newblock {A}pplied and {N}umerical {H}armonic {A}nalysis. {B}asel: 
{B}irkh{\"a}user, 2011. 

\bibitem{Hytonen:2016} 
T.~{H}yt{\"o}nen, J.~van {N}eerven, M.~{V}eraar, and L.~{W}eis. 
\newblock {\em {A}nalysis in {B}anach spaces. {V}olume {I}. {M}artingales and 
{L}ittlewood-{P}aley theory.}, volume~63. 
\newblock {C}ham: {S}pringer, 2016. 

\bibitem{Jakobsen:2018}
M.~S. Jakobsen.
\newblock On a (no longer) new Segal algebra: A review of the Feichtinger
  algebra.
\newblock {\em J. Fourier Anal. Appl.}, 24(6):1579--1660, 2018. 

\bibitem{Jakobsen:2018a} 
M.~S. {J}akobsen and F.~{L}uef. 
\newblock {D}uality of {G}abor frames and {H}eisenberg modules. 
\newblock {\em arXiv:1806.05616}, 2018.

\bibitem{Janssen:1995} 
A.~J. E.~M. {J}anssen. 
\newblock {D}uality and biorthogonality for {W}eyl-{H}eisenberg frames. 
\newblock {\em J. Fourier Anal. Appl.}, 1(4):403--436, 1995. 

\bibitem{Janssen:1998a} 
A.~J. E.~M. {J}anssen. 
\newblock {P}roof of a conjecture on the supports of {W}igner distributions. 
\newblock {\em J. Fourier Anal. Appl.}, 4(6):723--726, 1998. 

\bibitem{Keyl:2016} 
M.~{K}eyl, J.~{K}iukas, and R.~F. {W}erner. 
\newblock {S}chwartz operators. 
\newblock {\em Rev. Math. Phys.}, 28(3):1630001, 60, 2016. 

\bibitem{Kozek:1992} 
W.~{K}ozek. 
\newblock {O}n the generalized {W}eyl correspondence and its application to 
time-frequency analysis of linear time-varying systems. 
\newblock In {\em {I}{E}{E}{E} {I}nt. {S}ymp. on {T}ime--{F}requency and 
{T}ime--{S}cale {A}nalysis}, pages 167--170, {V}ictoria, {C}anada, 
1992. 

\bibitem{Kozek:1997thesis} 
W.~{K}ozek. 
\newblock {\em {M}atched {W}eyl-{H}eisenberg {E}xpansions of {N}onstationary 
{E}nvironments}. 
\newblock PhD thesis, {U}niversity of {T}echnology {V}ienna, 1997. 

\bibitem{Kozek:1997} 
W.~{K}ozek. 
\newblock {O}n the transfer function calculus for underspread {L}{T}{V} 
channels. 
\newblock {\em IEEE Trans. Signal Process.}, 45(1):219--223,  1997. 

\bibitem{Kozek:1998} 
W.~{K}ozek. 
\newblock {A}daptation of {W}eyl-{H}eisenberg frames to underspread 
environments. 
\newblock In H.~G. {F}eichtinger and T.~{S}trohmer, editors, {\em {G}abor 
{A}nalysis and {A}lgoritheorems: {T}heory and {A}pplications}, pages 323--352. 
{B}irkh{\"a}user {B}oston, {B}oston, 1998. 

\bibitem{Kozek:2005} 
W.~{K}ozek and G.~E. {P}fander. 
\newblock {I}dentification of operators with bandlimited symbols. 
\newblock {\em {S}{I}{A}{M} {J.} {M}ath. {A}nal.}, 37(3):867--888, 
2005. 

\bibitem{Luef:2009}
{F}. {L}uef.
\newblock {P}rojective modules over non-commutative tori are multi-window
  {G}abor frames for modulation spaces.
\newblock {\em {J}. {F}unct. {A}nal.}, 257(6):1921--1946, 2009.

\bibitem{Luef:2017}
F.~Luef and E.~Skrettingland.
\newblock {Convolutions for localization operators}.
\newblock {\em J. Math. Pure. Appl.}, 118:288 -- 316,
  2018.

\bibitem{Luef:2018b}
{F}. {L}uef and {E}. {S}krettingland.
\newblock Mixed-state localization operators: Cohen's class and trace class
  operators.
\newblock {\em {J}. {F}ourier {A}nal. {A}ppl.}, 25(4):2064--2108, 2019.

\bibitem{Luef:2018a} 
F.~{L}uef and E.~{S}krettingland. 
\newblock {O}n accumulated {C}ohen’s class distributions and mixed-state 
localization operators. 
\newblock {\em Constr. Approx.}, 52(1):31--64, 2020.  

\bibitem{Keville:2003} 
B.~{K}eville. 
\newblock {\em {M}ultidimensional {S}econd {O}rder {G}eneralised {S}tochastic 
{P}rocesses on {L}ocally {C}ompact {A}belian {G}roups}. 
\newblock PhD thesis, {T}rinity {C}ollege {D}ublin, 2003. 

\bibitem{Megginson:1998} 
R.~{M}egginson. 
\newblock {\em {A}n {I}ntroduction to {B}anach {S}pace {T}heory}, volume 183 of 
{\em {G}raduate {T}exts in {M}athematics}. 
\newblock {S}pringer-{V}erlag, {N}ew {Y}ork, 1998. 

\bibitem{Rieffel:1988} 
M.~A. {R}ieffel. 
\newblock {P}rojective modules over higher-dimensional noncommutative tori. 
\newblock {\em Canad. J. Math.}, 40(2):257--338, 1988. 

\bibitem{Romero:2012}
{J}.~{L}. {R}omero.
\newblock Characterization of coorbit spaces with phase-space covers.
\newblock {\em {J}. {F}unct. {A}nal.}, 262(1):59 -- 93, 2012.

\bibitem{Ryan:2002}
{R}.~{A}. {R}yan.
\newblock {\em {I}ntroduction to {T}ensor {P}roducts of {B}anach {S}paces.}
\newblock {S}pringer {M}onographs in {M}athematics. {S}pringer, 2002.

\bibitem{Sun:2006}
{W}. {S}un.
\newblock {G}-frames and g-{R}iesz bases.
\newblock {\em {J}. {M}ath. {A}nal. {A}ppl.}, 322:437--452, 2006.

\bibitem{Tschurtschenthaler:2000} 
T.~{T}schurtschenthaler. 
\newblock {T}he {G}abor frame {O}perator (its structure and numerical 
consequences). 
\newblock Master's thesis, {U}niversity of {V}ienna, 2000. 

\bibitem{Werner:1984}
{R}.~F. {W}erner.
\newblock {Q}uantum harmonic analysis on phase space.
\newblock {\em {J}. {M}ath. {P}hys.}, 25(5):1404--1411, 1984.

\bibitem{Wexler:1990} 
J.~{W}exler and S.~{R}az. 
\newblock {D}iscrete {G}abor expansions. 
\newblock {\em Signal Process.}, 21:207--220, 1990. 

\bibitem{Zibulski:1997}
{M.} {Z}ibulski and {Y}.~{Y}. {Z}eevi.
\newblock {A}nalysis of multi-window {G}abor-type schemes by frame methods.
\newblock {\em {A}ppl. {C}omput. {H}armon. {A}nal.}, 4(2):188--221, 1997.

\end{thebibliography}
\end{document}